\documentclass[12pt, draftclsnofoot, onecolumn]{IEEEtran}
\usepackage[T1]{fontenc}   
\usepackage[utf8]{inputenc} 
\usepackage[T1]{fontenc}
\usepackage[utf8]{inputenc} 
\usepackage{url}          
\usepackage{booktabs}       
\usepackage{hyperref}
\usepackage{amsfonts}       
\usepackage{nicefrac}       
\usepackage{microtype}      
\usepackage{lipsum}		
\usepackage[final]{graphicx}
\usepackage{float}
\usepackage{cite}
\usepackage{doi}
\usepackage{wrapfig}
\usepackage{algorithm,algorithmic}
\usepackage{amsmath, amssymb, bbm, amsthm, dsfont}
\usepackage{cuted}
\usepackage{xcolor} 
\usepackage{enumitem}
\usepackage{diagbox}
\allowdisplaybreaks

\usepackage{subcaption}

\usepackage{aliascnt}
\usepackage{bm}
\usepackage{amsfonts}
\usepackage{mathrsfs}

\def\bx{\boldsymbol{x}}
\def\by{\boldsymbol{y}}
\def\bz{\boldsymbol{z}}

\def\bg{\boldsymbol{m}}

\def\bs{\boldsymbol{s}}

\def\be{\boldsymbol{e}}
\def\bd{\boldsymbol{d}}
\def\bm{\boldsymbol{g}}

\def\bpi{\boldsymbol{\pi}}

\def\mF{\boldsymbol{\mathcal{F}}}

\DeclareMathOperator{\mG}{{\scriptstyle\boldsymbol{\mathcal{M}}}}

\DeclareMathOperator{\mx}{{\scriptstyle\mathcal{X}}}
\DeclareMathOperator{\md}
{{\scriptstyle\mathcal{D}}}
\DeclareMathOperator{\my}{{\scriptstyle\mathcal{Y}}}

\DeclareMathOperator{\mhE}{{\scriptstyle\boldsymbol{\mathcal{E}}}}

\DeclareMathOperator{\mX}{{\scriptstyle\boldsymbol{\mathcal{X}}}}

\DeclareMathOperator{\mY}{{\scriptstyle\boldsymbol{\mathcal{Y}}}}

\DeclareMathOperator{\mZ}
{{\scriptstyle\boldsymbol{\mathcal{Z}}}}

\DeclareMathOperator{\mM}
{{\scriptstyle\boldsymbol{\mathcal{G}}}}

\DeclareMathOperator{\mD}
{{\scriptstyle\boldsymbol{\mathcal{D}}}}

\def\mJ{\mathcal{J}}
\def\mQ{\mathcal{Q}}
\def\mT{\mathcal{T}}
\def\mU{\mathcal{U}}
\def\mW{\mathcal{W}}

\def\bxi{\boldsymbol{\xi}}
\def\bzeta{\boldsymbol{\zeta}}
\DeclareMathOperator{\mA}{{\mathcal{A}}}
\DeclareMathOperator{\mB}{{\mathcal{B}}}
\def\mC{\mathcal{C}}

\def\mP{\mathcal{P}}
\def\mR{\mathbb{R}}

\def\bxi{\boldsymbol{\xi}}

\def\mE{\mathbb{E}}

\def\cblue{\textcolor{blue}}

\newtheorem{Lemma}{Lemma}
\newtheorem{Theorem}{Theorem}
\newtheorem{Assumption}{Assumption}

\usepackage{booktabs}
\usepackage{tabularx}
\usepackage{makecell}
\usepackage[table]{xcolor}
\usepackage{threeparttable}
\usepackage{pifont} 
\definecolor{tableShade}{RGB}{245,248,252}
\newcolumntype{s}{>{\hsize=.0\hsize\centering\arraybackslash}X}
\newcolumntype{B}{>{\centering\arraybackslash}X}  
\newcolumntype{L}{>{\centering\arraybackslash}X} 

\title{\textbf{DAMA}: A Unified Accelerated Approach for Decentralized Nonconvex Minimax Optimization---Part I: Algorithm Development and Results}

\author{Haoyuan Cai,
        Sulaiman A. Alghunaim,
        and Ali H. Sayed
\thanks{Haoyuan Cai and Ali H. Sayed are with  the Institute of Electrical and Micro Engineering,  École Polytechnique Fédérale
de Lausanne, Switzerland (emails: \{haoyuan.cai, ali.sayed\}@epfl.ch).}
\thanks{Sulaiman A. Alghunaim is with Kuwait University, Kuwait (email:
sulaiman.alghunaim@ku.edu.kw).}
}

\begin{document}
\maketitle

\begin{abstract}
In this work and its accompanying Part II \cite{cai2025dama2},
we develop an accelerated algorithmic framework, \textbf{DAMA} (\textbf{D}ecentralized \textbf{A}ccelerated \textbf{M}inimax \textbf{A}pproach), for nonconvex Polyak–Łojasiewicz (PL) minimax optimization over decentralized multi-agent networks.
Our approach seamlessly integrates online and offline stochastic minimax algorithms with various decentralized learning strategies, yielding a versatile and unified framework that offers greater flexibility than existing methods.
Our unification is threefold:
i) We propose a unified decentralized learning strategy for minimax optimization over multi-agent networks that subsumes existing bias-correction methods (e.g., gradient tracking) and introduces new variants that achieve tighter network-dependent bounds,
ii) we introduce a new probabilistic gradient estimator, termed \textbf{GRACE} (\textbf{GR}adient \textbf{AC}celeration \textbf{E}stimator), which unifies the implementation of several momentum-based methods (e.g., STORM) and loopless variance-reduction techniques (e.g., PAGE, Loopless SARAH) for constructing accelerated gradients within \textbf{DAMA}. Moreover, \textbf{GRACE} is broadly applicable to a wide range of stochastic optimization problems, making it of independent interest, and iii) through a unified analytical framework, rather than case-by-case arguments, we establish a general performance bound for the proposed approach, achieving new state-of-the-art results with the best-known sample complexity in this setting.
To the best of our knowledge, \textbf{DAMA} is the first framework to achieve a multi-level unification of decentralized learning strategies and accelerated gradient techniques, thereby addressing a key gap in the literature. 
In Part I of this work, we focus on the algorithm development of \textbf{DAMA} and the presentation of the main results. In its accompanying Part II \cite{cai2025dama2}, 
we develop the theoretical analysis for \textbf{DAMA}  to substantiate the main results. Moreover,
we validate the effectiveness of \textbf{DAMA} and its variants across different network topologies using both synthetic and real-world datasets.

\end{abstract}


\section{Introduction}
\label{sec:introduction}

Minimax optimization underpins a wide range of applications, including generative adversarial networks (GANs) \cite{goodfellow2014gan}, reinforcement learning \cite{li2019robust}, domain adaptation \cite{acuna2021f}, and distributionally robust optimization \cite{sinha2017certifying}. 
At the same time, the ever-increasing scale of modern artificial intelligence (AI) tasks necessitates distributing data and computation across many agents (e.g., GPU clusters), highlighting the need for efficient decentralized minimax learning algorithms with provable convergence guarantees.
Motivated by this fact, we aim to solve nonconvex-nonconcave minimax optimization problems by utilizing the collective power of $K \ge 1$ cooperating agents, which are connected by a network topology.
In particular, we consider the following distributed stochastic minimax problem:
\begin{subequations}
\begin{align}
&\min_{x \in \mathbb{R}^{d_1}}
\max_{y \in \mathbb{R}^{d_2}}
J(x, y)
= \frac{1}{K}\sum_{k=1}^K J_k(x,y),   
\label{main:sec1:problem}\\
&\text{where }J_k(x, y) = 
\begin{cases}
\mathbb{E}_{\bzeta_k}[Q_k(x, y; \bzeta_k)] & \textbf{(Online)}\\
\frac{1}{N_k} \sum_{s=1}^{N_k} Q_{k}(x, y; \bzeta_{k}(s)) & \textbf{(Offline)}
\end{cases}\label{main:sec1:problem_stochastic},
\end{align}
\end{subequations}
where each agent $k$ seeks an optimal solution through local interactions with its immediate neighbors under a decentralized network topology. Here, the global risk (cost) function  $J(x,y)$ is defined as the average of $K$ local objectives $J_k(x,y)$, each privately owned by an agent $k$.  We assume that $J(x,y)$ is smooth and nonconvex in $x \in \mR^{d_1}$
but 
{\em possibly} nonconcave in $y \in \mR^{d_2}$ while satisfying the
Polyak-Łojasiewicz (PL) condition
(see  \autoref{main:assumption:costfunction} or reference \cite{karimi2016linear}).
The focus of this work is on two important learning scenarios, where each agent $k$ can draw independent and identically distributed (i.i.d.) samples: $\bzeta_k$ in the online streaming setting or $\bzeta_k(s)$ in the offline finite-sum (or empirical) setting, with $\bzeta_k(s)$ denoting the $s$-th local sample of agent $k$. These samples are used to evaluate the local instantaneous loss $Q_k(x,y;\bzeta_k)$ or $Q_k(x,y;\bzeta_k(s))$. Unlike prior works, we address {\em both} decentralized {\em online} stochastic and decentralized {\em offline} finite-sum scenarios. The former models online streaming  regimes, while the latter fits many practical deep learning tasks.

Recent years have witnessed considerable progress in the algorithmic and theoretical development of minimax optimization \cite{lin2020gradient,zhang2020single,yang2022faster, jin2020local, luo2020stochastic}.
Given that general minimax problems are notoriously difficult \cite{daskalakis2021complexity, hsieh2021limits}, prior works have primarily focused on a subclass of \eqref{main:sec1:problem}---\eqref{main:sec1:problem_stochastic},
which impose structural assumptions on the risk function.
In particular, nonconvex–(strongly) concave minimax optimization has received significant attention
\cite{lin2020gradient,luo2020stochastic, lin2020near, zhang2020single}, since its nonconvex formulation is well-suited for nonlinear models such as neural networks that are essential for complex learning tasks.
In this work,  however, we adopt a nonconvex–PL formulation that relaxes the strong concavity requirement on the $y$-variable, with the PL condition serving as a more suitable assumption for over-parameterized neural networks \cite{liu2022loss}.
Hence, our assumptions can more accurately represent important scenarios, such as cases where $x$ and $y$ parameterize different neural networks and one is over-parameterized.
Under a nonconvex-PL formulation, vanilla gradient descent–ascent (GDA) with a shared step size for both optimization variables tends to diverge since a uniform step size produces unbalanced dynamics under the asymmetric nonconvex–PL structure of $J(x,y)$; see the diverging example in \cite{yang2022nest}.
To address this issue, prior works adopted a two-time-scale step-size policy to calibrate the learning dynamics, where the step-size ratio is strictly constrained by polynomial orders of the condition number $\kappa \triangleq L_f/\nu$, where $L_f$ is the smoothness constant of the cost and $\nu$ is the strong concavity or PL constant.
A further challenge in stochastic nonconvex-PL minimax optimization is that the loss gradients of both variables are often computed from the same sample, resulting in correlated noise that breaks the martingale difference property.
This limitation may render the use of large batch sizes necessary in the stochastic GDA \cite{lin2020gradient, cai2024diffusion}, unless gradient noise is controlled through noise-reduction techniques \cite{cai2024accelerated,xian2021faster,chen2024efficient,sharma2022federated} or alternative sampling schemes \cite{yang2022faster,cai2024diffusion}.
Notably, several acceleration techniques have demonstrated effectiveness in reducing gradient noise while enhancing convergence speed \cite{luo2020stochastic, cai2024accelerated, huang2023near, huang2023enhanced, zhang2024jointly,chen2024efficient,gao2022decentralized}.
However,
{\em the majority of existing acceleration techniques are often implemented and analyzed independently, }
leaving the field without a unified gradient estimation strategy. This limitation motivates us to ask:

\textit{\textbf{Q1: Can we design a unified gradient estimator that enables fast and efficient learning for stochastic minimax optimization?}}

We answer this question by proposing a new probabilistic gradient estimator, which unifies several key acceleration techniques, see \autoref{main:unified_gradient} for a detailed discussion.
We further extend this strategy to the multi-agent setting, which brings additional challenges. In particular, sparse communication topologies and data heterogeneity across the network complicate algorithmic design and can markedly impair learning performance.
To mitigate the data heterogeneity issue, bias-correction strategies such as gradient tracking (GT)  can be employed \cite{xu2015augmented, nedic2017achieving}.
Actually,
{\em existing minimax works rely heavily on the GT strategy to combat the imbalanced data challenge} \cite{xian2021faster,huang2023near,ghiasvand2025robust, gao2022decentralized, zhang2024jointly, cai2025communication}. Nevertheless, this GT scheme is not guaranteed to achieve the state-of-the-art performance under a sparsely connected network \cite{alghunaim2022unified}.
Alternative strategies, including Exact Diffusion (ED) \cite{yuan2018exact,yuan2018exact2} and EXTRA \cite{shi2015extra}, have been shown to achieve better rates compared with GT  for minimization problems \cite{alghunaim2022unified,sayed2022inference}. Yet these alternative methods remain underexplored in the minimax setting, and it is uncertain whether their advantages carry over to the minimax optimization domain. This limitation prompts our second question:

\textit{\textbf{Q2:
Can we develop a generic approach that integrates the unified probabilistic gradient estimator with bias-correction decentralized learning strategies (e.g., ED and EXTRA) to achieve enhanced robustness against sparse communication networks and heterogeneous local data?}}

To this end, we present \textbf{DAMA}—a novel \textbf{D}ecentralized \textbf{A}ccelerated \textbf{M}inimax \textbf{A}pproach that seamlessly integrates our unified gradient estimator with mainstream decentralized learning strategies into a single, versatile framework.
In this way, the widely adopted GT method in \cite{xian2021faster,huang2023near,ghiasvand2025robust, gao2022decentralized, zhang2024jointly, cai2025communication}
can be viewed as a special instance of \textbf{DAMA}.  Leveraging the unified gradient estimator, our approach accommodates both online and offline learning settings. While such flexibility is essential in practice, rigorous convergence guarantees remain equally crucial. Existing works \cite{zhang2024jointly,xian2021faster,gao2022decentralized,liu2024decentralized} only established sample complexities that depend on the network spectral gap $\mathcal{O}(1-\lambda)$ in a weaker manner than the best-known results in the minimization setting \cite{xin2021hybrid,alghunaim2022unified}, where $\lambda$ denotes the mixing rate of the network (e.g., 0 for fully connected network and approaches 1 for sparse networks). Since the network spectral gap is extrinsic to the problem itself, this gap in performance bound suggests that sharper results should be attainable. These insights motivate our final question:

\textit{\textbf{Q3: Is it possible to develop a unified analysis of DAMA with improved convergence guarantees?}}

We answer this question in the affirmative by establishing rigorous convergence guarantees for the proposed framework. Overcoming these challenges is far from trivial, as they do not combine in a simple additive manner; instead, our analysis requires delicate, nontrivial techniques. In particular, we develop a transformed recursion and carry out the analysis in the transformed domain rather than directly relying on the original one (see \autoref{main:sec:tranasformation}). 

We address Q1 and Q2 in Part I of this work and lay the groundwork for a unified theoretical analysis that will answer Q3 in Part II \cite{cai2025dama2}.

\subsection{Related works}

Early works on minimax optimization primarily considered settings with exact gradient oracles \cite{nouiehed2019solving, jin2020local, lin2020near, lu2020hybrid, xu2023unified, yang2020catalyst,zhang2020single}, where many methods required solving a maximization subproblem at each iteration or resorting to nested-loop schemes \cite{nouiehed2019solving,jin2020local,lin2020near,lu2020hybrid,yang2020catalyst}.
Recently, a body of works has focused on the online stochastic setting, including methods \cite{lin2020gradient,yang2022faster,mahdavinia2022tight,luo2020stochastic,cai2024accelerated} tailored to the single-node case. In particular, \cite{lin2020gradient} proved convergence of two-time-scale stochastic GDA for smooth nonconvex–(strongly) concave minimax problems.
The work \cite{yang2022faster} extended the smoothed alternating GDA \cite{zhang2020single} to a stochastic nonconvex-PL setting, showing improved sample complexity compared to stochastic GDA.
The reference
\cite{mahdavinia2022tight} focused on classical extragradient-type algorithms for minimax optimization.
Note that \cite{yang2022faster,zhang2020single,mahdavinia2022tight} are closely related to proximal-type methods, which may help stabilize updates in ill-conditioned regimes (i.e., with large condition number $\kappa$). However, they do not improve 
{\color{black} the dependency of the} sample complexity on the accuracy tolerance $\varepsilon$, as this is primarily dictated by stochastic noise variance.
Incorporating variance-reduction techniques to suppress gradient noise can substantially improve convergence rates \cite{yang2020global,luo2020stochastic,zhang2022sapd}, with many minimax algorithms extending these methods originally developed for minimization \cite{reddi2016stochastic,nguyen2017sarah}. While these techniques guarantee accelerated convergence in theory, they face practical limitations under memory constraints. Moreover, these methods often depend on nested-loop structures with predetermined intervals for large-batch gradient computations, making it difficult to adapt to noise variance in an online learning scenario. Momentum-based methods \cite{cutkosky2019momentum,tran2022better,sayed2022inference} offer a simple yet flexible mechanism for controlling gradient noise and are more memory-efficient than classical variance-reduction schemes.
They have also been shown to accelerate stochastic minimax optimization \cite{huang2023enhanced,cai2024accelerated}.
In contrast to earlier studies, this paper proposes a unified gradient estimator that integrates several powerful noise-control strategies. The resulting strategies offer practitioners the flexibility to tailor method selection to suit practical scenarios and to adapt more effectively to varying noise levels.

{\color{black} Another growing} body of research has also been devoted to the development of distributed and decentralized methods for solving minimax problems.
Federated minimax learning represents an important scenario, in which a central entity coordinates the learning process of all agents \cite{wu2024solving,deng2021local,sharma2022federated,shen2024stochastic,cai2025communication}. The algorithms studied therein are mainly derived by extending momentum methods \cite{cutkosky2019momentum,cai2024accelerated} and (smoothed alternating) GDA \cite{yang2022faster,zhang2020single} from the single-agent context. 
Different from these studies, we investigate minimax optimization under  {\em decentralized setups} (without a central coordinator), a paradigm that enables greater scalability in large-scale applications.
Recent progress on decentralized minimax optimization can be found in \cite{cai2024diffusion, xian2021faster, ghiasvand2025robust,gao2022decentralized, huang2023near, cai2025communication, liu2024decentralized,chen2024efficient,mancino2023variance}.
Among these works,
 reference \cite{cai2024diffusion}
proposed a
decentralized
optimistic 
minimax method based on adapt-then-combine diffusion strategies \cite{alghunaim2020decentralized, sayed2022inference}.
A subsequent work \cite{huang2025optimistic} investigated the optimistic algorithm with the GT strategy.
Although optimistic gradients help stabilize training in a multi-agent minimax problem, their convergence continues to be limited by stochastic gradient noise. Another line of research developed two-time-scale momentum methods with GT \cite{huang2023near,xian2021faster}, which efficiently reduce gradient noise and accelerate convergence.
It is worth emphasizing that the dominant sample complexity reported in \cite{xian2021faster} depends on the network spectral gap $\mathcal{O}(1-\lambda)$, whereas momentum-based GT algorithms in the minimization domain \cite{xin2021hybrid} do not suffer from this dependence, suggesting that sharper results should be attainable.
The works 
\cite{gao2022decentralized, mancino2023variance,chen2024efficient, mancino2023variance} considered variance reduction \cite{li2020convergence, li2021page} with GT or fast mixing
to improve communication complexity.
It should be noted that \cite{chen2024efficient} relies on the network’s spectral information to implement a fast-mixing strategy with multiple communications per iteration, and moreover, it requires one to periodically compute batch gradient, which reduces its practicality.
The works \cite{ghiasvand2025robust} and \cite{cai2025communication}
considered gradient/momentum tracking with local updates, respectively, and the latter achieved enhanced communication complexity.
Different from prior works, we develop a {\em unified} gradient estimator that subsumes both accelerated momentum and variance reduction as special cases. We further integrate this estimator into a {\em unified} decentralized learning framework, yielding numerous algorithmic variants that, to the best of our knowledge, have not been explored in the minimax context. More importantly, our unified method achieves improved sample complexity over existing works. Moreover, for certain variants, we establish, for the first time, meaningful results on the transient time required to achieve linear speedup where the convergence rate improves linearly with number of agents $K$ (see Table \ref{tab:algo-comparison:transient} in subsection \ref{main:sec:main_result}).

\subsection{Main contributions}

We summarize the main contributions of this work and its accompanying Part II \cite{cai2025dama2} as follows:
\begin{itemize}
    \item 
    Unlike prior works that explore limited synergies between decentralized learning strategies and gradient acceleration techniques,  we bridge this gap by proposing a unified approach
    \textbf{DAMA} (\textbf{D}ecentralized \textbf{A}ccelerated \textbf{M}inimax \textbf{A}pproach) based on a new probabilistic gradient estimator 
    \textbf{GARCE} (\textbf{GR}adient \textbf{AC}celeration \textbf{E}stimator).
Our framework integrates a decentralized learning design, extended from the minimization setting \cite{alghunaim2020decentralized,sayed2022inference,alghunaim2022unified} to minimax optimization, with \textbf{GRACE}, a novel probabilistic gradient estimator that unifies several prominent variance reduction schemes. Notably, \textbf{GRACE} can be applied to other optimization domains, making it of independent interest beyond minimax problems. By combining these two components, our approach yields a family of new algorithmic variants not covered in the existing literature. The resulting methods are simple, flexible, free of double loop structures, and are easily adaptable to noise variance.

\item 
In contrast to prior works that analyzed each specific strategy separately \cite{xian2021faster,huang2023near,gao2022decentralized,zhang2024jointly,cai2025communication,cai2024diffusion,liu2024decentralized,mancino2023variance}, we conduct a direct and unified analysis for the proposed algorithm. For the first time, the performance of several well-known accelerated methods, including Loopless SARAH \cite{li2020convergence}, PAGE \cite{li2021page}, and STORM \cite{cutkosky2019momentum}, in conjunction with diverse decentralized strategies \cite{sayed2022inference}, is jointly captured under a generic performance bound. This analysis is highly nontrivial and requires the development of a new transformed recursion that is distinct from the gradient-based work \cite{alghunaim2022unified, alghunaim2020decentralized,sayed2022inference}. Moreover, our problem domains are inherently more challenging than minimization problems, as evidenced by the hyperparameter conditions.

\item 
We present a range of results derived from the generic performance bounds of the unified approach, many of which have not been established previously.  Table~\ref{tab:algo-comparison:complexity} summarizes some of our results compared to existing works. For more comprehensive results, please refer to Table~\ref{tab:algo-comparison}.
 To the best of our knowledge, we obtain the best-known guarantees for multiple important instances. For example, our sample complexity of STORM+ATC-GT is $\mathcal{O}\Big(
\frac{\kappa^3\varepsilon^{-3}}{K} + \frac{\kappa^2\varepsilon^{-2}}{(1-\lambda)^3}\Big)$, where the dominant term is independent of the network spectral gap $\mathcal{O}(1-\lambda)$. This stands in sharp contrast to prior results \cite{xian2021faster}, as shown in Table~\ref{tab:algo-comparison:complexity}, where the dominant term depends on $\mathcal{O}(1-\lambda)$. The only exception is DREAM~\cite{chen2024efficient}, however, it requires prior knowledge of the network’s spectral properties to implement a fast mixing strategy involving multiple communications per iteration.

 We also show that the STORM+ED variant achieves a per-agent sample complexity of
 $\mathcal{O}\Big(
\frac{\kappa^3\varepsilon^{-3}}{K} + \frac{\kappa^2\varepsilon^{-2}}{(1-\lambda)^2}
 \Big)$ in the online stochastic setting, where the dependence on the spectral gap $\mathcal{O}(1-\lambda)$ improves over GT methods, which aligns with the results in the minimization setting \cite{alghunaim2022unified}. 
  For the Loopless SARAH+ED variant,
its dominant sample complexity $\mathcal{O}\Big(
\frac{\kappa^2\sqrt{N}\varepsilon^{-2}}{K}
 \Big)$ improves upon 
 \textcolor{black}{the prior state-of-the-art results established in}
 DREAM \cite{chen2024efficient} by a factor of $\mathcal{O}(\sqrt{K})$ times when $N$ is large enough.
Importantly, our algorithm avoids the need for the spectral information of the mixing matrix for explicit implementation and does not require the multi-step communication used in DREAM. 

\begin{table*}[!t]
\centering
\begin{threeparttable}
\caption{
Sample complexity comparison between existing algorithms and \textbf{DAMA} for decentralized stochastic nonconvex–strongly concave/PL minimax optimization.}
\footnotesize
\label{tab:algo-comparison:complexity}
\rowcolors{3}{tableShade}{white}
\setlength{\tabcolsep}{4pt}
\renewcommand{\arraystretch}{1.25}
\begin{tabularx}{\linewidth}{l c c B}
\toprule
\makecell{\textbf{Algorithm}} & 
\makecell{\textbf{PS}$^\circ$} & 
\makecell{\textbf{Studied before?}}& 
\makecell{\textbf{Sample Complexity}$^{\ddagger}$}  \\
\midrule
\textbf{Existing works} &&&\\
DM-HSGD\cite{xian2021faster}&On& Yes& $\mathcal{O}\Big(
\frac{\kappa^3\varepsilon^{-3}}{K(1-\lambda)^2}
\Big)$ \\
DGDA-VR \cite{zhang2024jointly}&On&Yes& $\mathcal{O}\Big(\frac{\kappa^3\varepsilon^{-3}}{(1-\lambda)^2}\Big)$ \\
DSGDA\cite{gao2022decentralized}&Off&Yes& $\mathcal{O}\Big(\frac{\kappa^3\sqrt{N}\varepsilon^{-2}}{(1-\lambda)^2}\Big)$ \\
DREAM \cite{chen2024efficient}& On/Off &Yes
& $\mathcal{O}\Big(
\frac{\kappa^3\varepsilon^{-3}}{K}
\Big)$ / $\mathcal{O}\Big(
\frac{\kappa^2\sqrt{N}\varepsilon^{-2}}{\sqrt{K}}
\Big)$ 
\\
Local-HMMT\cite{cai2025communication}&On&Yes& $\mathcal{O}\Big(\frac{\kappa^3\varepsilon^{-3}}{K(1-\lambda)^{1.5}}\Big)$
\\
\midrule 
\textbf{Proposed (DAMA)}&&&
\\
STORM+ED (\cite[Collorary 1]{cai2025dama2}) & On & No &  $\mathcal{O}\Big(
    \frac{\kappa^3\varepsilon^{-3}}{K}+\frac{\kappa^2\varepsilon^{-2}}{(1-\lambda)^2} +\frac{\kappa^{1.5} \lambda^{1.5}K^{0.5}\varepsilon^{-1.5}}{(1-\lambda)^{9/4}}
\Big)$ \\
STORM+EXTRA (\cite[Collorary 2]{cai2025dama2})       & On  & No &    $\mathcal{O}\Big(
    \frac{\kappa^3\varepsilon^{-3}}{K}+\frac{\kappa^2\varepsilon^{-2}}{(1-\lambda)^2} +\frac{\kappa^{1.5}K^{0.5}\varepsilon^{-1.5}}{(1-\lambda)^{9/4}}
\Big)$\\
STORM+ATC-GT (\cite[Collorary 3]{cai2025dama2})& On & Yes & $\mathcal{O} \Big(\frac{\kappa^3\varepsilon^{-3}}{K}+\frac{\kappa^2\varepsilon^{-2}}{(1-\lambda)^3}+\frac{\kappa^{1.5}\lambda^3K^{0.5}\varepsilon^{-1.5}}{(1-\lambda)^{3}} 
\Big)$ \\ 
PAGE+ED (\cite[Collorary 7]{cai2025dama2}) &On& No& $\mathcal{O}\Big( 
\frac{\kappa^3\varepsilon^{-3}}{(1-\lambda)^{1.5}K}
\Big)$
\\
PAGE+EXTRA (\cite[Collorary 8]{cai2025dama2})&On&No& $\mathcal{O}\Big( 
\frac{\kappa^3\varepsilon^{-3}}{(1-\lambda)^{1.5}K}
\Big)$
\\ 
PAGE+ATC-GT (\cite[Collorary 9]{cai2025dama2})&On& Yes & $\mathcal{O}\Big( 
\frac{\kappa^3\varepsilon^{-3}}{(1-\lambda)^{2}K}
\Big)$  \\
PAGE+ED (\cite[Collorary 4]{cai2025dama2})& Off& No &
$\mathcal{O}\Big(
\frac{\kappa^2\sqrt{N}\varepsilon^{-2}}{\sqrt{K}(1-\lambda)^{1.5}}
\Big)$\\
PAGE+EXTRA (\cite[Collorary 5]{cai2025dama2})&Off& No & $\mathcal{O}\Big(\frac{\kappa^2\sqrt{N}\varepsilon^{-2}}{\sqrt{K}(1-\lambda)^{1.5}} \Big)$ \\
PAGE+ATC-GT (\cite[Collorary 6]{cai2025dama2}) &Off & Yes & $\mathcal{O}\Big(\frac{\kappa^2\sqrt{N}\varepsilon^{-2}}{\sqrt{K}(1-\lambda)^{2}}\Big)$
\\
L-SARAH +ED (\cite[Collorary 10]{cai2025dama2})& Off &No& $\mathcal{O}\Big(\frac{\kappa^2 \sqrt{N} \varepsilon^{-2}}{K}
+ \frac{\kappa^2\varepsilon^{-2}}{(1-\lambda)^2}
+\frac{\sqrt{N}}{K}\Big)$  \\
L-SARAH +EXTRA (\cite[Collorary 11]{cai2025dama2})& Off &No& $\mathcal{O}\Big(\frac{\kappa^2 \sqrt{N} \varepsilon^{-2}}{K}
+ \frac{\kappa^2\varepsilon^{-2}}{(1-\lambda)^2}
+\frac{\sqrt{N}}{K}\Big)$ \\
L-SARAH +ATC-GT (\cite[Collorary 12]{cai2025dama2})& Off &No&  $\mathcal{O}\Big(\frac{\kappa^2 \sqrt{N} \varepsilon^{-2}}{K}
+ \frac{\kappa^2\varepsilon^{-2}}{(1-\lambda)^3}
+\frac{\sqrt{N}}{K}\Big)$
\\
\bottomrule
\end{tabularx}
\begin{tablenotes}[flushleft]
\footnotesize
\item 
Notes: $^\circ$PS = Problem setups (On = Online stochastic, Off = Offline finite-sum).   
$^\ddagger$ The last column reports the dominant per-agent sample complexity of each algorithm in achieving an $\varepsilon$-stationary point (see, e.g., the convergence criterion \eqref{main:convergence_criterion}).
L-SARAH = Loopless SARAH. Here, the comparison of sample complexity is meaningful, as it reflects the intrinsic complexity bound of the algorithm, whereas communication complexity can often be reduced through some standard techniques.
$\varepsilon$: accuracy level. $K$: the number of agents. For simplicity, we assume $N_k \equiv N, \forall k \in\{1, \ldots, K\}$ throughout the paper.  $\kappa$ is the condition number $\kappa \triangleq L_f/\nu$. $1-\lambda$ denotes the network spectral gap; the sparser the network, the smaller its value.
The technical details of \textbf{DAMA} can be found in Part II of this work \cite{cai2025dama2}.
\end{tablenotes}
\end{threeparttable}
\end{table*}

 For PAGE+ED variants,
 we achieve a sample complexity of $\mathcal{O}\Big(
 \frac{\kappa^2 \sqrt{N}\varepsilon^{-2}}{(1-\lambda)^{1.5}\sqrt{K}}
 \Big)$ in the offline setting, which outperforms \cite{ gao2022decentralized}
 under a sparsely connected network. 
Furthermore, we present some meaningful and new results of the transient time for STORM-based variants to achieve linear speedup in terms of the number of agents $K$ -- see TABLE \ref{tab:algo-comparison:transient}. Finally, we present new results for the minibatch version of PAGE, focusing on variants built on Loopless SARAH \cite{li2020convergence}. These variants need to compute minibatch gradient only between successive large-batch communication rounds, unlike \cite{chen2024efficient}.
For clarity, the performance bounds and transient times of key instances of \textbf{DAMA} are summarized in Tables \ref{tab:algo-comparison}--\ref{tab:algo-comparison:transient}, respectively.
Technical details for these results are provided in Part II \cite{cai2025dama2}.

\end{itemize}

\textbf{Notation.}
We use the normal font, e.g., $x$,
to denote a deterministic quantity.
We use bold font, e.g., 
$\bx$, to denote a stochastic quantity.
We use bold calligraphic font,
e.g, $\mX_{i} \triangleq \mbox{\rm col}\{\bx_{1,i}, \dots, \bx_{K,i}\}  \triangleq \{\bx_{k,i}\}_{k=1}^{K}\in \mR^{Kd_1}$, to denote
networked
stochastic  variable. Here,
$\bx_{k,i}$ stands for the state vector of  agent $k$ at communication round $i$.
In addition,
we use the superscripts  
$(\cdot)^x$
to denote a quantity associated with the 
$x$-variable,
such as the gradient estimator $\bm^x_{k,i}$ of agent $k$ at communication round $i$.
Note that we adopt the same nomenclature for the 
$y$-variable.
For a matrix $B$, we denote its null space by $\mathrm{null}(B)$
and the subspace spanned by its columns (range space) by $\mathrm{span}(B)$ .
The symbols $\mathrm{I}_{N}, \mathrm{0}_N, \mathds{1}_N$ denote the $N \times N$ identity matrix, the $N$-dimensional zero and one vectors, respectively. $\mathcal{N}_k$ denotes the neighboring agent index set of the agent $k$.
\(\mathrm{blkdiag}\{A, B, \dots\}\) denotes a block-diagonal matrix formed by placing the matrices \(A, B, \dots\) on the diagonal and zeros elsewhere.
Given a set $\mathcal{S}$, we use $|\mathcal{S}|$ to denote its cardinality.
Furthermore, $\|\cdot\|$ is the $\ell_2$-norm, 
$\langle \cdot, \cdot \rangle$ 
stands for the inner product in the Euclidean space, $(\cdot)^\top$ denotes the matrix transpose, and $\otimes$ denotes the Kronecker product.
$b = \mathcal{O}(a)$ means there exists a constant $c>0$ such that 
$b \leq c\,a$, while 
$b = \Theta(a)$ means there exist constants $c_1,c_2>0$ such that 
$c_1 a \leq b \leq c_2 a$.

\section{Algorithm development}
\label{sec:unifiedframework}


In this section, we present the essential algorithmic components for the unified approach. The content is organized into two parts: 1) in subsection \ref{main:unified_distributed}, we present a unified view of decentralized learning strategies by casting decentralized minimax optimization as a constrained two-level primal-dual optimization problem, which is then addressed through an iterative primal–dual approach, and 2) secondly, in subsection \ref{main:unified_gradient}, we revisit several prominent variance-reducing techniques in stochastic optimization and introduce a new probabilistic gradient estimator to unify them.

\subsection{Unified decentralized learning framework}
\label{main:unified_distributed}

Our argument builds on prior works in minimization \cite{alghunaim2022unified,sayed2022inference,alghunaim2020decentralized}, which provide the foundation for extending decentralized learning design to the minimax setting. 
For convenience, we introduce the following network-augmented block vector quantities:
\begin{align}
\mx = \mbox{\rm col}\{x_1, \dots, x_{K}\} \in \mR^{Kd_1},
\quad 
\my = \mbox{\rm col}\{y_1, \dots, y_{K}\} \in \mR^{Kd_2}.
\end{align}
Our goal is to solve \eqref{main:sec1:problem} in a decentralized manner with 
$K$ cooperating agents, where each agent has access only to its local data and communicates {\color{black} only with its neighbors over } a partially connected network. A fundamental requirement in many multi-agent systems is that all participating agents reach a consensus. That being said, we can rewrite \eqref{main:sec1:problem}  as the following equivalent constrained minimax problem
\begin{align}
&\min_{\mx \in \mathbb{R}^{Kd_1}}
\max_{\my \in \mathbb{R}^{Kd_2}}
\mJ(\mx, \my), \ {\rm s.t.} \begin{cases}  x_1= \cdots = x_{K} \in \mR^{d_1} \\ 
 y_1= \cdots = y_{K} \in \mR^{d_2}
 \end{cases},
\label{main:sec1:problem_stochastic2}
\end{align}
where 
\begin{align}
\mJ(\mx, \my)
\triangleq \frac{1}{K}\sum_{k=1}^K J_k(x_k,y_k).
\end{align} 
The same problem can also be written in the form of a constrained two-player game \cite{bacsar1998dynamic}:
\begin{align}
\begin{cases}
 \displaystyle    \min_{\mx \in \mR^{Kd_1}} \mJ(\mx, \my) \\
 \displaystyle      \min_{\my \in \mR^{Kd_2}} \textcolor{blue}{-}\mJ(\mx, \my) 
\end{cases} \quad   \operatorname{s.t.} 
\begin{cases}
    x_1= \cdots = x_{K} \in \mR^{d_1} \\
    y_1= \cdots = y_{K} \in \mR^{d_2}
\end{cases} 
\label{main:sec2.1:two_level_X}.
\end{align}
Problem  \eqref{main:sec2.1:two_level_X} seeks a stationary pair $(x^\star, y^\star)$ while respecting the consensus constraint. 
 To unify several decentralized strategies for solving \eqref{main:sec2.1:two_level_X},
we start by introducing auxiliary block matrices
$\mathcal{B}_x \in \mR^{Kd_1 \times Kd_1}, \mathcal{B}_y \in \mR^{Kd_2 \times Kd_2}, \overline{\mathcal{C}}_x \in \mR^{Kd_1 \times Kd_1}, \overline{\mathcal{C}}_y \in \mR^{Kd_2 \times Kd_2}$ that satisfy the equivalence conditions 
\begin{align}
\begin{cases}
&x_1 = \cdots = x_K = \bar{x}
\Longleftrightarrow \mathcal{B}_x \mx = 0 \Longleftrightarrow \overline{\mathcal{C}}_x \mx = 0 \\
& y_1 = \cdots = y_K =\bar{y} 
\Longleftrightarrow \mathcal{B}_y \my = 0 \Longleftrightarrow \overline{\mathcal{C}}_y \my = 0 \\
&\overline{\mathcal{C}}_x \text{ and }\overline{\mathcal{C}}_y  \text{ are positive semi-definite}
\end{cases}.
\label{main:BxbYyCxCy}
\end{align}
In other words, these matrices are chosen such that the nullspaces of ${\cal B}_x,{\cal B}_y, \overline{\cal C}_x, \overline{\cal C}_y$ are given by $\mathrm{null}(\mathcal{B}_x) = \mathrm{span}(\mathds{1}_{K}\otimes \bar{x}),
\mathrm{null}(\overline{\cal C}_x) = \mathrm{span}(\mathds{1}_{K}\otimes \bar{x})$,
 $\mathrm{null}(\mathcal{B}_y) = \mathrm{span}(\mathds{1}_{K}\otimes \bar{y}) $, and $\mathrm{null}(\overline{\cal C}_y) = \mathrm{span}(\mathds{1}_{K}\otimes \bar{y})$.
 We will explain further ahead how these matrices are chosen; for now, we view them as degrees of freedom that are available to the designer. Different choices for these quantities will lead to diffferent algorithms. 
Using \eqref{main:BxbYyCxCy},
we can rewrite
\eqref{main:sec2.1:two_level_X}
in the form of a constrained problem with linear constraints: {\color{black} (see \cite{sayed2022inference})}
\begin{align}
&\begin{cases}
  \displaystyle   \min_{\mx \in \mR^{Kd_1}} \mJ(\mx, \my) \\
    \displaystyle   \min_{\my \in \mR^{Kd_2}} \textcolor{blue}{-} \mJ(\mx, \my) 
\end{cases} \quad   \operatorname{s.t.} \begin{cases}
\mB_x  \mx = 0 \\  
\mB_y \my = 0
\end{cases}.
\label{main:sec2.1:two_level_X_formu1}
\end{align}
 We will now solve \eqref{main:sec2.1:two_level_X_formu1} by  relying on the augmented Lagrangian approach,  
which motivates us to introduce the following two-level primal-dual form:

\begin{align}
&\begin{cases}
 \displaystyle  \min_{\mx \in \mR^{Kd_1}} \max_{\md_x \in \mR^{Kd_1}} \Big\{\mJ(\mx, \my)+ \frac{1}{\mu_x} \md^\top_x \mB_x \mx + \frac{1}{2\mu_x}\|\mx\|^2_{\overline{\mC}_x} \Big\} ,\\
 \\
  \displaystyle     \min_{\my \in \mR^{Kd_2}} \max_{\md_y \in \mR^{Kd_2}} \Big\{ \textcolor{blue}{-}\mJ(\mx, \my)+ \frac{1}{\mu_y} \md^\top_y \mB_y \my + \frac{1}{2\mu_y}\|\my\|^2_{\overline{\mC}_y} \Big\},
\end{cases}
\label{main:sec2.1:two_level_X_formu2}
\end{align}
where $\mu_x>0, \mu_y>0$ are constants and {\color{black} the network-augmented vectors} $\md_x \in \mR^{Kd_1}$ and $\md_y \in \mR^{Kd_2}$
are dual variables associated with constraints $\mB_x \mx=0$ and $\mB_y \my=0$,  respectively. 
Furthermore, the rightmost quadratic terms impose penalties when the local models deviate from one another.
To solve \eqref{main:sec2.1:two_level_X_formu2}, we will adopt an iterative primal-dual decentralized approach as follows. We introduce two other free parameters in the form of block symmetric doubly-stochastic matrices $\mA_x$ and $\mA_y$ satisfying $\mA_x\mx=\mx$ and $\mA_y\my=\my$.  These matrices will define the combination policies over the graph structure and will control information sharing among the agents, as will become evident soon. Using the parameters $\{{\cal A}_x, {\cal A}_y, {\cal B}_x, {\cal B}_y, {\cal C}_x, {\cal C}_y\}$, we  write down the following primal-dual iterations over $i\geq 0$:
\begin{subequations}
\begin{align}
 &\mx_{i+1}= \mathcal{A}_x \Big[\mx_{i}
-\mu_x \nabla_x \mJ(\mx_i, \my_i) - \overline{\mathcal{C}}_x
\mx_i \Big]- \mathcal{B}_x \md_{x,i} ,
\label{main:sec2.1:two_level_Y_reformu1}
\\
 &\md_{x,i+1} 
 = \md_{x,i} 
 + \mathcal{B}_x \mx_{i+1} 
\label{main:sec2.1:two_level_Y_reformu2}
,\\
&\my_{i+1} = \mA_{y} \Big[\my_{i}
  {\cblue +}\mu_y\nabla_y  \mJ(\mx_i, \my_i) - \overline{\mathcal{C}}_y \my_i\Big]- \mathcal{B}_y \md_{y,i} ,
\label{main:sec2.1:two_level_Y_reformu3}
\\
 &\md_{y,i+1}
 = \md_{y,i} 
 +\mathcal{B}_y \my_{i+1}.
\label{main:sec2.1:two_level_Y_reformu4}
\end{align}
\end{subequations}
These recursions are obtained from \eqref{main:sec2.1:two_level_X_formu2} by writing down gradient descent  steps over the primal variables and gradient ascent steps over the dual variables.

Setting $\mathcal{C}_x = \mathrm{I}_{Kd_1} -\overline{\mathcal
{C}}_x, \mathcal{C}_y = \mathrm{I}_{Kd_2} -\overline{\mathcal
{C}}_y$,
we simplify recursions
\eqref{main:sec2.1:two_level_Y_reformu1}---\eqref{main:sec2.1:two_level_Y_reformu4}
to 
\begin{subequations}
\begin{align}
&\mx_{i+1}
= \mA_x (\mC_x \mx_i - \mu_x \nabla_x \mJ(
\mx_i, 
\my_i
)) - \mB_x \md_{x,i},
\label{main:sec2.1:determin_final1}
\\
(\textbf{DAMA}) \ &\my_{i+1}
=\mA_y
(\mC_y \my_i {\cblue +} \mu_y \nabla_y \mJ(
\mx_i, 
\my_i
)) - \mB_y \md_{y,i}, 
\label{main:sec2.1:determin_final2}\\
&\md_{x,i+1} = \md_{x,i} + \mB_x \mx_{i+1}, 
\label{main:sec2.1:determin_final3}\\
&\md_{y,i+1} = \md_{y,i}
+\mB_y \my_{i+1}.
\label{main:sec2.1:determin_final4}
\end{align}
\end{subequations}
In \eqref{main:sec2.1:determin_final2}, the 
$y$-gradient arises with the opposite sign of the 
$x$-gradient, reflecting conflicting optimization goals between the two variables due to the minimization and maximization steps.

$\bullet$ \textbf{Stochastic implementation}

In the stochastic environment, the actual gradients are not known and, therefore, we need to resort to approximate gradient calculations.  We denote the approximation for $\nabla_x \mJ (\mx, \my)$ and $\nabla_y \mJ (\mx, \my)$ by $\mM_{x,i}$  and $\mM_{y,i}$, respectively.  Using these approximations, we replace \eqref{main:sec2.1:determin_final1}---\eqref{main:sec2.1:determin_final4} by 
\begin{subequations}
\begin{align}
\mX_{i+1}
&= \mathcal{A}_x(\mathcal{C}_x \mX_i- \mu_{x}
\mM_{x,i}) - \mathcal{B}_x
\mD_{x,i} ,\label{main:X_update} \\
\mY_{i+1}
&= \mathcal{A}_y (\mathcal{C}_y \mY_i {\cblue + } \mu_{y}
\mM_{y,i}) - \mathcal{B}_y
\mD_{y,i}  ,\label{main:Y_update}\\
\mD_{x,i+1}
&= 
\mD_{x,i}
+\mB_x\mX_{i+1}\label{main:Ux_update} ,\\
\mD_{y,i+1}
&= 
\mD_{y,i}
+\mB_y\mY_{i+1}.
\label{main:Uy_update}
\end{align}
\end{subequations}

We will denote the individual entries of the block gradient approximations, as well as the entries of the primal and dual block variables by 
\begin{subequations}
\begin{align}
\mM_{x,i} &\triangleq \mbox{col}\{\bm^x_{1,i}, \dots, \bm^x_{K,i}\} \ \in \mR^{Kd_1}, \quad& \mM_{y,i} \triangleq & \mbox{col}\{ \bm^y_{1,i}, \dots, \bm^y_{K,i}\} \in \mR^{Kd_2}, \\
\mX_{i} &\triangleq  \mbox{\rm col}\{\bx_{1,i}, \dots, \bx_{K,i}\} \in \mR^{Kd_1}, \quad& 
\mY _{i}\triangleq & \mbox{\rm col}\{\by_{1,i}, \dots, \by_{K,i}\} \in \mR^{Kd_2},\\
\mD_{x,i} &\triangleq  \mbox{\rm col}\{\bd^x_{1,i}, \dots, \bd^x_{K,i}\} \in \mR^{Kd_1}, \quad& 
\mD_{y,i}\triangleq & \mbox{\rm col}\{\bd^y_{1,i}, \dots, \bd^y_{K,i}\} \in \mR^{Kd_2}.
\end{align}
\end{subequations} 
We provide some examples of choices for the design matrices in Table \ref{tab:matrix_choices}, along with the names of the learning strategies that correspond to these choices. These choices, and others as well, will be motivated in the discussions that follows.

\begin{table*}[!htbp]
\centering
\caption{Matrix choices for different decentralized {\em minimax} learning strategies. Below, $W \in \mR^{K\times K}$ is a symmetric, doubly stochastic matrix.}
\label{tab:matrix_choices}
\renewcommand{\arraystretch}{1.2}
\setlength{\tabcolsep}{3pt}
\begin{tabular}{l|cccccc}
\toprule
\diagbox[width=2.5cm]{\textbf{Strategy}}{\textbf{Choices}}
& $\mA_x$ & $\mA_y$ & $\mB_x$ & $\mB_y$ & $\mC_x$ & $\mC_y$ \\ 
\toprule 
ED     & $W \otimes \mathrm{I}_{d_1}$ & $W \otimes \mathrm{I}_{d_2}$ & $(\mathrm{I}_{Kd_1} - W \otimes \mathrm{I}_{d_1})^{1/2}$ & $(\mathrm{I}_{Kd_2} -W \otimes \mathrm{I}_{d_2})^{1/2}$ & $\mathrm{I}_{Kd_1}$ & $\mathrm{I}_{Kd_2}$ \\ 
EXTRA  & $\mathrm{I}_{Kd_1}$ & $\mathrm{I}_{Kd_2}$ & $(\mathrm{I}_{Kd_1} - W \otimes \mathrm{I}_{d_1})^{1/2}$ & $(\mathrm{I}_{Kd_2} - W \otimes \mathrm{I}_{d_2})^{1/2}$ & $W \otimes \mathrm{I}_{d_1}$ & $W \otimes \mathrm{I}_{d_2}$ \\
ATC-GT & $(W \otimes \mathrm{I}_{d_1})^2$ & $(W \otimes \mathrm{I}_{d_2})^2$ & $\mathrm{I}_{Kd_1} - W \otimes \mathrm{I}_{d_1}$ & $\mathrm{I}_{Kd_2} - W \otimes \mathrm{I}_{d_2}$  & $\mathrm{I}_{Kd_1}$ & $\mathrm{I}_{Kd_2}$  \\ 
semi-ATC-GT & $W \otimes \mathrm{I}_{d_1}$ & $W \otimes \mathrm{I}_{d_2}$ & $\mathrm{I}_{Kd_1} - W \otimes \mathrm{I}_{d_1}$ & $\mathrm{I}_{Kd_2} - W \otimes \mathrm{I}_{d_2}$  & $W \otimes \mathrm{I}_{d_1}$ & $W \otimes \mathrm{I}_{d_2}$  \\ 
non-ATC-GT & $ \mathrm{I}_{Kd_1}$ & $\mathrm{I}_{Kd_2}$ & $\mathrm{I}_{Kd_1} - W \otimes \mathrm{I}_{d_1}$ & $\mathrm{I}_{Kd_2} - W \otimes \mathrm{I}_{d_2}$  & $(W \otimes \mathrm{I}_{d_1})^2$ & $(W \otimes \mathrm{I}_{d_2})^2$  \\
\bottomrule
\end{tabular}
\end{table*}

Note that all the choices in the table 
are chosen in the form of the Kronecker product of an identity matrix with a matrix $W = [w_{k\ell}]$. 
To illustrate how the general framework we are developing can be implemented at the node level, we provides in Algorithm~\ref{alg:sample_basic:ED}---\ref{alg:sample_basic:EXTRA} several examples.

\begin{algorithm}[!htbp]
\footnotesize
\caption{\hspace{-2pt}: \ Minimax Exact Diffusion (ED) algorithm}
\label{alg:sample_basic:ED}
\begin{algorithmic}[1]
    \STATE \textbf{Initialize}: Setting $\bx_{k,0} = \bx_{k,-1}, \by_{k,0} =\by_{k,-1} \ \forall k \in\{1, \dots, K\}$,  and choose appropriate $\mu_{x}, \mu_{y}$ 
    \FOR{$i = 0, \dots, T$}
\FOR{every agent $k$ in parallel} 
\STATE 
\hspace{-1.5ex}{} 
\vspace{-1.5em}
\begin{align}
\bx_{k,i+1} &= \sum_{\ell \in \mathcal{N}_k} w_{k\ell} \Big(
 2\bx_{k,i} - \bx_{k,i-1} - \mu_x (\bm^x_{k,i} - \bm^x_{k,i-1})
\Big), \label{main:EDstrategy:step1}\\
\by_{k,i+1} &= \sum_{\ell \in \mathcal{N}_k} w_{k\ell} \Big(
2\by_{k,i} - \by_{k,i-1}\textcolor{blue}{+} \mu_y (\bm^y_{k,i} - \bm^y_{k,i-1})
\Big).
\label{main:EDstrategy:step2}
\end{align}
\vspace{-2em}
\ENDFOR
\STATE $i = i+1$
\ENDFOR
\end{algorithmic}
\end{algorithm}
\begin{algorithm}[!htbp]
\footnotesize
\caption{\hspace{-2pt}: \ Minimax Adapt-then-Combine Gradient Tracking (ATC-GT) algorithm}
\label{alg:sample_basic:ATC-GT}
\begin{algorithmic}[1]
    \STATE \textbf{Initialize}: Setting $\bx_{k,0} = \bx_{k,-1} , \by_{k,0} =\by_{k,-1} \ \forall k \in\{1, \dots, K\}$, and choose appropriate $\mu_{x}, \mu_{y}$  
    \FOR{$i = 0, \dots, T$}
\FOR{every agent $k$ in parallel} 
\STATE \hspace{-1.5ex}{} 
\vspace{-1.5em}
\begin{align}
\bg^x_{k,i} &= \sum_{\ell \in \mathcal{N}_k} w_{k\ell} \Big(
 \bg^x_{k,i-1} - \bm^x_{k,i-1} + \bm^x_{k,i}
\Big),   \label{main:GTstrategy:step1} \\
\bg^y_{k,i} &= \sum_{\ell \in \mathcal{N}_k} w_{k\ell} \Big(
 \bg^y_{k,i-1} - \bm^y_{k,i-1} + \bm^y_{k,i} \Big),
 \label{main:GTstrategy:step2}\\
\bx_{k,i+1} &= 
\sum_{\ell \in \mathcal{N}_k} w_{k\ell} \Big(
\bx_{k,i} - \mu_x \bg^x_{k,i}
\Big),  \label{main:GTstrategy:step3}\\
\by_{k,i+1} &= 
\sum_{\ell \in \mathcal{N}_k} w_{k\ell} \Big(
\by_{k,i} \textcolor{blue}{+} \mu_y \bg^y_{k,i}
\Big).
 \label{main:GTstrategy:step4}
\end{align}
\vspace{-2em}
\ENDFOR
\STATE $i = i+1$
\ENDFOR
\end{algorithmic}
\end{algorithm}
\begin{algorithm}[!htbp]
\footnotesize
\caption{ \hspace{-2pt}: \ Minimax EXTRA algorithm}
\label{alg:sample_basic:EXTRA}
\begin{algorithmic}[1]
    \STATE \textbf{Initialize}: Setting $\bx_{k,0} = \bx_{k,-1} , \by_{k,0} =\by_{k,-1}  \ \forall k \in\{1, \dots, K\}$, and choose appropriate $\mu_{x}, \mu_{y}$ 
    \FOR{$i = 0, \dots, T$}
\FOR{every agent $k$ in parallel} 
\STATE 
\hspace{-1.5ex}{} 
\vspace{-1.5em}
\begin{align}
\bx_{k,i+1} &= \sum_{\ell \in \mathcal{N}_k} w_{k\ell} \Big(
 2\bx_{k,i} - \bx_{k,i-1}\Big) - \mu_x (\bm^x_{k,i} - \bm^x_{k,i-1})  \label{main:EXTRAstrategy:step1}
,\\
\by_{k,i+1} &= \sum_{\ell \in \mathcal{N}_k} w_{k\ell} \Big(
2\by_{k,i} - \by_{k,i-1}\Big)\textcolor{blue}{+} \mu_y (\bm^y_{k,i} - \bm^y_{k,i-1}) \label{main:EXTRAstrategy:step2}
.
\end{align}
\vspace{-2em}
\ENDFOR
\STATE $i = i+1$
\ENDFOR
\end{algorithmic}
\end{algorithm}

It is important to note that prior works on minimax optimization have focused almost exclusively on the GT strategy \cite{xian2021faster,huang2023near,cai2025communication,ghiasvand2025robust, huang2025optimistic, zhang2024jointly}, leaving ED and EXTRA unexplored in this setting. Thus, Algorithms~\ref{alg:sample_basic:ED} and~\ref{alg:sample_basic:EXTRA} presented here represent new minimax algorithms.
Moving forward now, the key challenge is how to design efficient estimators for the gradient vectors, which is the subject of the next subsection.

\subsection{Unified probabilistic gradient estimator}
\label{main:unified_gradient}
In this subsection, we devise efficient strategies for constructing stochastic gradient approximators $\bm^x_{k,i}, \bm^y_{k,i}$. Our approach builds on two well studied families of algorithms: variance reduction techniques \cite{nguyen2017sarah,li2020convergence,li2021page,zhang2024jointly} and momentum methods \cite{cutkosky2019momentum,tran2022better,cai2024accelerated,sayed2022inference,huang2023enhanced,chen2024efficient}, both of which have demonstrated strong empirical performance in machine learning applications.
Momentum strategies are typically more memory efficient since they avoid large batch gradient computations. In contrast, variance reduction techniques periodically require high quality gradient estimates, often obtained through large batches, which can reduce their practicality for large data sizes. Nevertheless, in the finite sum setting, variance reduction methods can achieve better sample complexity than accelerated momentum when the sample size is below a threshold determined by the accuracy 
$\varepsilon$. Hence, we will consider both strategies in order to address different scenarios more effectively.
In what follows, we will demonstrate how the centralized minimax momentum strategy \cite{sharma2022federated} and various accelerated momentum methods \cite{cutkosky2019momentum,tran2022better,cai2024accelerated,sayed2022inference,huang2023enhanced,chen2024efficient} can be jointly represented within a single form.
At the same time, a rich body of variance reduction techniques have been developed in the literature \cite{nguyen2017sarah,luo2020stochastic,zhang2024jointly,li2020convergence,li2021page,yang2020global}, including classical double loop methods such as SARAH \cite{nguyen2017sarah} and SVRG-type \cite{reddi2016stochastic,luo2020stochastic,yang2020global}, as well as loopless variants such as Loopless SARAH \cite{li2020convergence} and PAGE \cite{li2021page}. The loopless variance reduction approach employs a gradient estimator that probabilistically alternates between large batch and mini-batch computations with a prescribed distribution. In this work, we primarily focus on loopless variance reduction techniques, as their simple design naturally suits online learning by allowing the switching probability to adapt to the noise level.
Before we revisit several efficient gradient estimators, we introduce the following
notation for compactness:
\begin{align}
\bz_{k,i} &\triangleq {\rm col}\{\bx_{k,i}, \by_{k,i}\}, \\
q_x(\bz_{k,i};\bxi_{k,i})&\triangleq
\nabla_xQ_k(\bx_{k,i-1},\by_{k,i-1};\bxi_{k,i}) -  
\nabla_x Q_k(\bx_{k,i},\by_{k,i};\bxi_{k,i}), \\
{\rm h}_x(\bz_{k,i}; \bxi_{k,i}) 
&\triangleq \Big[\nabla^2_x Q_k(\bx_{k,i}, \by_{k,i}; \bxi_{k,i}), 
\nabla^2_{xy} Q_k(\bx_{k,i}, \by_{k,i}; \bxi_{k,i})\Big]
\begin{bmatrix}
  \bx_{k,i-1}-\bx_{k,i}\\  
  \by_{k,i-1} - \by_{k,i}
\end{bmatrix},
\end{align}
where $\nabla^2_x(\cdot)$ 
and $\nabla^2_{xy}(\cdot)$ are the Hessian and Jacobian operators, respectively,  $\bz_{k,i}$ is a concatenated vector,  $q_x(\bz_{k,i};\bxi_{k,i})$ represents the difference of loss gradient around  
$\bz_{k,i}$ and its neighboring point, and 
$h_x(\bz_{k,i};\bxi_{k,i})$ is the partial second-order information around the same point.

\noindent $\bullet$ \textbf{Accelerated momentum}

We focus on the updating rule for the momentum vector $\bm^x_{k,i}$ associated with the $x$-variable of agent $k$. The form for the 
$y$-variable follows analogously.
One class of momentum methods for updating  $\bm^x_{k,i}$ in the minimax context can be constructed recursively as follows: 
\begin{align}
\bm^x_{k,i} &= (1-\beta_x)\Big[\bm^x_{k,i-1} - \gamma_1q_x(\bz_{k,i};\bxi_{k,i}) \Big] + \beta_x\nabla_x Q_k(\bz_{k,i}; \bxi_{k,i}),
\label{main:STORMexpression}
\end{align}
where $\gamma_1 \in \{0, 1\}$ is an integer; $\bxi_{k,i}=\bzeta_k$ in the online setting, and  $\bxi_{k,i}=\bzeta_k(s)$ in the offline setting for some $s \in [N_k]$; $\beta_x \in [0,1]$ is a smoothing factor. Note that 
by setting $\beta_x =1$,
we recover the stochastic gradient method studied in \cite{lin2020gradient} and by setting $\gamma_1 = 0$
and $\beta_x \in (0, 1)$, we recover the heavy-ball momentum method  \cite{sharma2022federated}.
Furthermore, by setting $\gamma_1 = 1$
and $\beta_x \in (0,1)$,
we get the accelerated momentum algorithm STORM, which was introduced in \cite{cutkosky2019momentum} and  later investigated in the minimax context by several works \cite{xian2021faster, huang2023near, huang2023enhanced, wu2024solving}.
STORM resembles heavy-ball momentum augmented with an additional correction term, namely, $q_x(\bz_{k,i};\bxi_{k,i})$.  
The additional term refines the past momentum by diminishing the influence of stale directions and reorienting it toward the latest direction with a new sample  $\bxi_{k,i}$, which effectively reduces gradient variance over iterations.
Using a first-order Taylor expansion of the gradient function $\nabla_xQ_k(\bx,\by;\bxi_{k,i})$ at  $(\bx_{k,i},\by_{k,i})$ yields
\begin{align}
&q_x(\bz_{k,i};\bxi_{k,i})
=
{\rm h}_x(\bz_{k,i};\bxi_{k,i}) 
+\mathcal{O}(\|\bz_{k,i-1} -\bz_{k,i}\|^2) 
\label{main:sec2.2:hessiancorrect},
\end{align}
where 
the increment of the weight can be approximated by $\mathcal{O}(\mu^2_x+\mu^2_y)$ under a small step size regime. That being said, the correction term
$q_x(\bz_{k,i};\bxi_{k,i})$
can be approximated by ${\rm h}_x(\bz_{k,i};\bxi_{k,i})$.
Replacing $q_x(\bz_{k,i};\bxi_{k,i})$ 
in \eqref{main:STORMexpression}
by ${\rm h}_x(\bz_{k,i};\bxi_{k,i})$, we arrive at another line of approach known as the Hessian-corrected momentum scheme, studied in the works \cite{tran2022better,cai2024accelerated, cai2025communication}.
This scheme achieves a convergence rate on par with STORM under the alternative assumption that the cost function is second-order Lipschitz.
In the small step-size regime, we {\color{black} assume} that both correction terms yield a similar effect, as the higher-order terms in \eqref{main:sec2.2:hessiancorrect} can be neglected.
Putting these schemes together, we obtain the following generic form:
\begin{align}
\bm^x_{k,i} 
&=(1-\beta_x)\Big[\bm^x_{k,i-1} -\gamma_1 q_x(\bz_{k,i};\bxi_{k,i}) - \gamma_2 h_x(\bz_{k,i};\bxi_{k,i})\Big]
+ \beta_x \nabla_x Q_k(\bz_{k,i};\bxi_{k,i}),
\label{main:unifiedmomentum}
\end{align}
where $\gamma_2 \in \{0, 1\}$.
By tuning $\gamma_1, \gamma_2 \in \{0,1\}$, we can activate either of the momentum methods. 

\noindent $\bullet$ \textbf{Loopless variance reduction technique}

Unlike momentum schemes, which rely on exponential averaging of past stochastic gradients, variance-reduction methods reduce noise by intermittently computing high-quality gradients with large batches.
Large-batch gradients are typically computed periodically in an outer loop as in the classical works \cite{reddi2016stochastic, nguyen2017sarah}. Recently, loopless variance-reduction methods have gained attention \cite{li2020convergence,li2021page}, as they eliminate the need for an explicit double-loop structure by probabilistically switching between two gradient estimators. 
To implement the strategy, 
we start by 
drawing 
a random Bernoulli variable $\bpi_{k,i} \sim  \operatorname{Bernoulli}(p)$ at communication round $i$ with a shared random seed across agents, where $p \in [0,1]$ is the probability of the event  $\bpi_{k,i} =1$, and then 
perform:
\begin{align}
 &\bm^x_{k, i} =
\begin{cases}\displaystyle
\frac{1}{|\mathcal{E}_{k,i}|} 
\sum_{\bxi_{k,i} \in \mathcal{E}_{k,i} }
\nabla_x Q_k(\bz_{k,i}; \bxi_{k,i}) & \quad 
 (\text{if } \bpi_{k,i} = 1)
 \\
 \displaystyle
 \bm^x_{k,i-1} -
\frac{1}{b}\sum_{\bxi_{k,i} \in  \mathcal{E}^\prime_{k,i} } q_x(\bz_{k,i};\bxi_{k,i})   &\quad  (\text{if } \bpi_{k,i} = 0).
\end{cases}
\label{main:looplesssarah}
\end{align}
Here, $\mathcal{E}_{k,i}$ denotes a  large-batch set of i.i.d. samples 
and $|\mathcal{E}_{k,i}| = B $ in the online setting, or a full batch with $|\mathcal{E}_{k,i}| = N_k$ in the finite-sum setting.
Moreover,
$\mathcal{E}^\prime_{k,i}$ is a mini-batch set of  i.i.d. samples with size $b$.
By setting $b = \mathcal{O}(1)$, construction 
\eqref{main:looplesssarah} reduces to the 
Loopless SARAH \cite{li2020convergence}, which 
probabilistically
alternates between the large-batch (or full-batch) gradient and SARAH.
By setting larger mini-batch size $b \le \mathcal{O}(\sqrt{|\mathcal{E}_{k,i}|})$,
we achieve 
 PAGE \cite{li2021page, chen2024efficient}, which probabilistically switches between a large-batch  (or full-batch) gradient and a minibatch variant of SARAH.
Furthermore, in the finite-sum setting, if we set the Bernoulli variable $p=1$, then 
\eqref{main:looplesssarah} becomes a deterministic gradient method \cite{lin2020gradient, jin2020local}.
Note that $p$ can also be made adaptive to the noise level in online scenarios, providing significant flexibility.
Typically, $p \ll 1$ is selected, which guarantees that large batch gradients are computed only rarely.

\noindent $\bullet$ \textbf{Unified probabilistic accelerated gradient estimator}

We can observe that the frequently performed part of \eqref{main:looplesssarah}, namely the step associated with the case $\bpi_{k,i} = 0$,
is structurally similar to strategy
\eqref{main:unifiedmomentum} (i.e., $\gamma_1 =1, \gamma_2 = 0, \beta_x =0, b=1$).
In order to unify the loopless variance reduction and momentum schemes,
we propose to replace the second part of 
\eqref{main:looplesssarah}
by a minibatch variant of \eqref{main:unifiedmomentum}.
By  doing so, we can draw
a random Bernoulli variable $\bpi_{k,i} \sim  \operatorname{Bernoulli}(p)$ at communication round $i$
and perform:
{\footnotesize
\begin{align}
&\bm^x_{k, i}=
\notag \\
& 
\begin{cases} \displaystyle
\frac{1}{|\mathcal{E}_{k,i}|} 
\sum_{\bxi_{k,i} \in \mathcal{E}_{k,i} }
\nabla_xQ_k(\bz_{k,i};\bxi_{k,i})
&(\text{if } \bpi_{k,i} = 1)  
 \\
 \\
 (1-\beta_x)\Big[\bm^x_{k,i-1} -
\frac{1}{b}\underset{{\bxi_{k,i} \in  \mathcal{E}^\prime_{k,i} }}{\sum}  \Big(\gamma_1q_x(\bz_{k,i};\bxi_{k,i}) 
+ \gamma_2  h_x(\bz_{k,i};\bxi_{k,i})
\Big)  
\Big] + \frac{\beta_x}{b} \underset{{\bxi_{k,i} \in  \mathcal{E}^\prime_{k,i} }}{\sum}
\nabla_x Q_k(\bz_{k,i};\bxi_{k,i}) \hspace{-1em}
&(\text{if } \bpi_{k,i} = 0),
\end{cases} 
\label{main:GRACE}
\end{align}
}
where the batch of $\mathcal{E}_{k,i}$ and
$\mathcal{E}^\prime_{k,i}$ are constructed similar to \eqref{main:looplesssarah}. 
 We note that 
the proposed scheme  \eqref{main:GRACE} provides a unified implementation for several momentum and loopless variance reduction schemes. 
As shown in Table \ref{main:table:hyperparameter}, choosing appropriate values for the hyperparameters $p, \beta_x, \beta_y, b, B, \gamma_1, \gamma_2$, all previously discussed special cases can be recovered. 
To the best of our knowledge, the above loopless scheme has not been previously explored under a unified form in the literature. For this reason, we refer to it as \textbf{GRACE} (\textbf{GR}adient \textbf{AC}celeration \textbf{E}stimator) throughout the paper.

\begin{table}[!htbp]
\centering
\caption{Hyperparameter choices used to unify existing gradient estimator strategy. Below, $N \equiv N_k$ is the local sample size; $\beta_x, \beta_y$ are smoothing factors; $\gamma_1, \gamma_2 \in \{0, 1\}$ are integer; $B \ge 1$ is the large batch size parameter; $b \ge 1$ is the minibatch size parameter; $p \in [0, 1]$ is Bernoulli parameter. We assume learning rates $\mu_x, \mu_y \in (0, 1)$.}
\label{main:table:hyperparameter}
\renewcommand{\arraystretch}{1.0}
\setlength{\tabcolsep}{3pt}
\begin{tabular}{l|cccccc}
\toprule
\diagbox[width=2.5cm]{\textbf{Strategy}}{\textbf{Choices}}
& $\gamma_1$ & $\gamma_2$ & $\beta_x, \beta_y$ & $B$ & $b$ & $p$ \\ 
\toprule 
GDA \cite{jin2020local, lin2020gradient}     & $0$ & $0$ & $ 0$ & $N$ & $0$ & $1$ \\ 
GDA\cite{jin2020local, lin2020gradient}  & $0$ & $0$ & $1$ & $0$ & $N$ & $0$ \\
Stochastic-GDA\cite{ lin2020gradient}  & $0$ & $0$ & $1$ & $0$  & $\mathcal{O}(1)$ & $0$  \\ 
HB Momentum \cite{sharma2022federated} & $0$ & $0$ & $(0,1)$ & $0$  & $\mathcal{O}(1)$ & $0$  \\ 
STORM \cite{cutkosky2019momentum,xian2021faster} & $ 1$ & $0$ & $(0,1)$ & $0$  & $\mathcal{O}(1)$ & $0$  \\
HC-Momentum \cite{cai2024accelerated} & $0$ & $1$ & $(0,1)$ & $0$  & $\mathcal{O}(1)$ & $0$  \\
Loopless SARAH \cite{li2020convergence}
& $1$ & $0$ & $0$ & $N$  & $\mathcal{O}(1)$ & $(0,1)$  \\
PAGE \cite{li2021page, chen2024efficient}
& $1$ & $0$ & $0$ & $N$  & $\mathcal{O}(\sqrt{N})$ & $(0,1)$ \\
\textbf{GRACE (Ours)} &$\{0, 1\}$ & $\{0, 1\}$ & $[0, 1]$& $[1, N]$ &
$[1, N]$ & $[0, 1]$ \\
\bottomrule
\end{tabular}
\end{table}

There exist other hybrid variants covered by \textbf{GRACE}, which  are not shown in Table \ref{main:table:hyperparameter}.
In this work, our goal is to achieve faster convergence, which requires setting either $\gamma_1=1$ or $\gamma_2=1$. 
Notably, two special cases of \textbf{GRACE}, namely, the configuration with $\beta_x ,\beta_y \in (0, 1),p = 0, b = 1, \gamma_1=1, \gamma_2=0$ \cite{cutkosky2019momentum,xian2021faster, huang2023near}, and the configuration with
$\beta_x ,\beta_y \in (0, 1), p = 0, b=1, \gamma_1=0, \gamma_2=1$ \cite{tran2022better,cai2024accelerated}, have been shown to yield improved convergence rates. This suggests that choosing $\gamma_1 =1$ or $\gamma_1 =1$ is essential for acceleration. We will therefore focus on \textbf{GRACE} with
 $\beta_x ,\beta_y \in [0, 1], p \in [0,1], b, B \in [1, N], \gamma_1=1,\gamma_2=0$,  which includes a broad range of previously unstudied cases and leads to the simplified form shown in \textbf{Algorithm} \ref{alg:sample}. 
The analyses with the hyperparameter case $\gamma_1 = 0, \gamma_2 =1$ can be developed following the approach in \cite{tran2022better,cai2024accelerated}.

\vspace{-1em}
\subsection{Unified decentralized accelerated minimax algorithm}
We now present our main algorithmic framework, called \textbf{DAMA}, which integrates 
the unified decentralized learning strategy
developed in subsection \ref{main:unified_distributed} and the
unified gradient strategy \textbf{GRACE}
developed in subsection \ref{main:unified_gradient}.
The description of \textbf{DAMA} is presented at a node-level in Algorithm \ref{alg:sample}.
This algorithm begins by initializing the appropriate values for local models and the hyperparameters $\mu_x, \mu_y,\beta_x, \beta_y, b,  p, B$  for each agent. We refer to Table \ref{main:table:hyperparameter} for possible choices. We note that
the Bernoulli parameter $p$ is set to a small value to ensure that  large-batch computation  will only be rarely activated. Each agent uses a shared random seed to generate the random variable $\bpi_{k,i} \sim {\rm Bernoulli}(p)$ at each communication round.  To ensure convergence, the ratio $\mu_x/\mu_y$ must remain below a certain threshold, which can be chosen empirically.
In addition, we set the minibatch size $b$ much smaller than $B$.
After setting appropriate initialization for all hyperparameters, the algorithm starts at communication round $i =0$
and repeatedly executes the steps outlined in lines 3 to 15 until the $T$ communication rounds are finished.
Specifically, in line 8, the \textbf{GRACE} strategy \eqref{main:GRACE} is employed to 
estimate the gradient direction associated with each variable:

$\bullet$
If $\bpi_{k,i}=1$, each agent draws $B$-batch i.i.d.~samples
    $\mathcal{E}_{k,i} = \{\bzeta_{k,1},  \dots, \bzeta_{k,B}\}$ (\textbf{online}) or local batch
    $\mathcal{E}_{k,i} =  \{\bzeta_{k}(1), \dots, \bzeta_{k}(N_k)\}$ (\textbf{finite-sum}), and then 
computes:
\begin{align} 
\bm^x_{k, i}
&= \textstyle 
\dfrac{1}{|\mathcal{E}_{k,i}|} 
\sum\limits_{\bxi_{k,i} \in \mathcal{E}_{k,i} }
\nabla_x Q_k(\bz_{k,i}; \bxi_{k,i})
, \label{main:DAMA:largebatch}\\ 
\bm^y_{k, i}
&= \dfrac{1}{|\mathcal{E}_{k,i}|} 
\sum\limits_{\bxi_{k,i} \in \mathcal{E}_{k,i} }
\nabla_y Q_k(\bz_{k,i}; \bxi_{k,i}).
\end{align}
$\bullet$ If $\bpi_{k,i} = 0$, each agent draws a 
   $b$-minibatch of i.i.d~samples  
$\mathcal{E}^\prime_{k,i}$ (\textbf{online/finite-sum})
and computes: 
\begin{align}
\displaystyle \bm^x_{k, i} &=
 (1-\beta_x)\Big(\bm^x_{k,i-1} -\dfrac{1}{b}\underset{{\bxi_{k,i} \in  \mathcal{E}^\prime_{k,i} }}{\sum}
\nabla_x Q_k(\bz_{k,i-1}; \bxi_{k,i})
\Big)  + \dfrac{1}{b}\underset{{\bxi_{k,i} \in  \mathcal{E}^\prime_{k,i} }}{\sum}
\nabla_x Q_k(\bz_{k,i};\bxi_{k,i}),  
\\
\displaystyle \bm^y_{k,i}&= 
 (1-\beta_y)\Big(\bm^y_{k,i-1}-
\dfrac{1}{b}\underset{{\bxi_{k,i} \in  \mathcal{E}^\prime_{k,i} }}{\sum}  \nabla_y Q_k(\bz_{k,i-1}; \bxi_{k,i})\Big) 
 + \dfrac{1}{b}\underset{{\bxi_{k,i} \in  \mathcal{E}^\prime_{k,i} }}{\sum}
\nabla_y Q_k(\bz_{k,i};  
\bxi_{k,i}). \label{main:DAMA:minibatch}
\end{align}
Once the gradient direction is formed for each agent, 
we use the strategies from 
Algorithms \ref{alg:sample_basic:ED}---\ref{alg:sample_basic:EXTRA}
to update the local variables.
We refer to Appendix \ref{appendix:special_realization}
for additional implementations.
\vspace{-1em}
\begin{center}
\begin{minipage}{0.8\textwidth}
\begin{algorithm}[H]
\footnotesize
\caption{\hspace{-1pt}: \textbf{DAMA} (Unified \textbf{D}ecentralized \textbf{A}ccelerated \textbf{M}inimax \textbf{A}lgorithm)}
\label{alg:sample}
\begin{algorithmic}[1]
    \STATE \textbf{Initialize}: $\bx_{k,0}= \bx_{k,-1} , \by_{k,0} = \by_{k,-1} \ \forall k \in \{1, \ldots,K\}$,
    Bernoulli parameter $p$, hyperparameters $\mu_{x}, \mu_{y}, \beta_{x}, \beta_{y}, b,B(\text{online}),$ $b_0 (\text{initial batch size}).$
    \FOR{$i = 0, \dots, T$}
    \FOR{agent $k$ in parallel} 
    \STATE 
\underline{\texttt{Estimate gradient}} \\
    \IF{$i=0$}
    \STATE
    Draw a $b_0$-minibatch of i.i.d. samples $\mathcal{E}_{k,0}$ and let $\bm^x_{k, -1} =0, \bm^y_{k,-1}=0$ and
\begin{align}
\bm^x_{k, 0}
&= \frac{1}{b_0} \sum_{\bxi_{k,0} \in \mathcal{E}_{k,0}} 
\nabla_x Q_k(\bz_{k,0}; \bxi_{k,0}), \quad
\bm^y_{k, 0}=\frac{1}{b_0} \sum_{\bxi_{k,0} \in \mathcal{E}_{k,0}}
\nabla_y Q_k(\bz_{k,0}; \bxi_{k,0}).   
\end{align}
\ELSE
    \STATE
    Draw a random variable $\bpi_{k,i} \sim \operatorname{Bernoulli}(p)$ with a shared random seed and 
    update $\bm^x_{k,i}, \bm^y_{k,i}$ according to \eqref{main:DAMA:largebatch}---\eqref{main:DAMA:minibatch}.
\ENDIF 
\STATE 
\underline{\texttt{Diffusion learning}}
\STATE 
If \textbf{ED}: update $\bx_{k,i+1}, \by_{k,i+1}$
using steps \eqref{main:EDstrategy:step1}---\eqref{main:EDstrategy:step2} from Algorithm \ref{alg:sample_basic:ED}. 
\STATE 
If \textbf{ATC-GT}: update $\bx_{k,i+1}, \by_{k,i+1}$
using steps \eqref{main:GTstrategy:step1}---\eqref{main:GTstrategy:step4} from Algorithm \ref{alg:sample_basic:ATC-GT}. 
\STATE 
If \textbf{EXTRA}: update $\bx_{k,i+1}, \by_{k,i+1}$
using steps \eqref{main:EXTRAstrategy:step1}---\eqref{main:EXTRAstrategy:step2} from Algorithm \ref{alg:sample_basic:EXTRA}.
\ENDFOR
\STATE $i = i+1$
\ENDFOR
\end{algorithmic}
\end{algorithm}
\end{minipage}
\end{center}
\section{Theoretical results}
\label{sec:analysis}
In this section, we report a general performance bound for \textbf{DAMA} (see Theorem \ref{main:theorem:main}) and refined bounds for specific strategies (see Table \ref{tab:algo-comparison}).
For clarity, we focus on the presentation of the main results in Part I, while the technical details and derivations are provided in Part II \cite{cai2025dama2}.
Since our approach unifies a broad range of learning methods, the analysis becomes considerably more intricate. To address this challenge, we reformulate the recursions by introducing a useful variable transformation that exploits a fundamental factorization of the transition matrix.
We then perform the theoretical analysis in the transformed domain. Our analyses rely on standard assumptions widely adopted in many minimax works \cite{yang2022faster,xian2021faster,huang2023enhanced, chen2024efficient, cai2025communication,cai2024diffusion, nouiehed2019solving}. We state these assumptions next.

\subsection{Assumptions}
We will seek  an approximate stationary pair $(\bx_{c,i}, \by_{c,i})$ such that the gradient norms evaluated at this point are bounded by a tolerance $\varepsilon$, i.e.,
\begin{align}
\label{main:convergence_criterion}
\mE\|\nabla_x J(\bx_{c,i}, \by_{c,i})\|^2 \le \varepsilon^2, \quad  \mE\|\nabla_y J(\bx_{c,i}, \by_{c,i})\|^2 \le \varepsilon^2,
\end{align}
where $(\bx_{c,i}, \by_{c,i})$ is taken as the network centroid defined by
\begin{align}
\bx_{c,i} \triangleq \frac{1}{K}\sum_{k=1}^K \bx_{k,i}, \quad 
\by_{c,i} \triangleq \frac{1}{K}\sum_{k=1}^K \by_{k,i}.
\end{align}
Such a stationary pair provides a first-order approximation to the Nash equilibrium.
For ease of comparison, we express the network centroid in block form as
\begin{align}
    \mX_{c,i} = \mathds{1}_K \otimes \bx_{c,i} \in \mR^{Kd_1},   \quad   \mY_{c,i} = \mathds{1}_K \otimes \by_{c,i} \in \mR^{Kd_2}.
\end{align}
To examine the convergence and performance guarantees of the \textbf{DAMA} strategy, we introduce the following assumptions. 
\begin{Assumption}[\textbf{Cost function}]
\label{main:assumption:costfunction}
The global cost function $J(x,y)$
is nonconvex in $x$ 
and nonconcave in $y$ but satisfies a
$\nu$-PL condition for any fixed $x$, i.e.,
\begin{align}
\|\nabla_y J(x,y)\|^2 \ge 2\nu\Big(\max_y J(x,y) - J(x,y)\Big),
\end{align}
where $\nu >0$ is a positive constant.
Furthermore, the maximum envelope function $P(x) = \max_y J(x, y)$ is lower bounded, i.e., 
\begin{align}
P^\star \triangleq \inf_x P(x) > - \infty.
\end{align}
\end{Assumption}
We assume the local loss gradients are smooth in the following sense.
\begin{Assumption}[\textbf{Expected smoothness}]
\label{main:assumption:expectedsmooth}
The gradients of the local loss function $Q_k(x,y;\bxi)$ are assumed to be $L_f$-Lipschitz (or $L_f$-smooth) in expectation, i.e., for any $x_1, x_2 \in \mathbb{R}^{d_1},y_1, y_2 \in \mathbb{R}^{d_2}$, there exists a constant $L_f > 0$ such that
\begin{align}
&\mathbb{E}\|\nabla_x Q_k(x_1,y_1;\bxi) - \nabla_x Q_k(x_2,y_2;\bxi)\|^2 \le L_f^2\big(\|x_1-x_2\|^2 + \|y_1-y_2\|^2\big) , \\ 
&\mathbb{E}\|\nabla_y Q_k(x_1,y_1;\bxi) - \nabla_y Q_k(x_2,y_2;\bxi)\|^2 \le L_f^2\big(\|x_1-x_2\|^2 + \|y_1-y_2\|^2\big). 
\end{align}
\end{Assumption}
Under Assumptions \ref{main:assumption:costfunction} and \ref{main:assumption:expectedsmooth}, we can establish a Danskin-type Lemma \cite{nouiehed2019solving} which implies that the maximum envelope function $P(x)$ is $L \triangleq (L_f+ \frac{\kappa L_f}{2})$-smooth  and 
$\nabla P(x) = \nabla_x J(x, y^{o}(x))$, where $ y^{o}(x)  =\arg\max_y J(x,y)$. These are essential properties underlying our analysis.
Since we consider stochastic learning, we impose the following standard assumptions on the stochastic loss gradients.
\begin{Assumption}[\textbf{Gradient noise process}]
\label{main:assumptions:gradientnoise}
The local stochastic loss gradients $\nabla_w Q_k(\bx_{k,i}, \by_{k,i}; \bxi_{k,i})$, where $ w \in \{x,y\}$, are assumed to be conditionally unbiased with bounded variance for all $k,i$,  
\begin{subequations}
\begin{align}
&\mE[\nabla_w Q_k(\bx_{k,i},\by_{k,i}; \bxi_{k,i}) \mid \mF_{i}] = \nabla_w J_k(\bx_{k,i},\by_{k,i}),
\\
&\mE[\|\nabla_w Q_k(\bx_{k,i},\by_{k,i}; \bxi_{k,i})  - \nabla_w J_k(\bx_{k,i},\by_{k,i})\|^2 \mid \mF_{i}] \le \sigma^2.
\end{align} 
where $\mF_i = \sigma({\bx_{k,j}, \by_{k,j} \mid k = 1,\ldots,K, j \le i})$ is the $\sigma$-algebra generated by the state vectors of all agents up to iteration $i$, and $\sigma^2 \geq 0$ is a constant. In addition, the random samples ${\bxi_{k,i}}$ are assumed to be spatially and temporally independently and identically distributed (i.i.d.).
\end{subequations}
\end{Assumption}
Since we focus on decentralized learning, we impose the following assumption on the combination matrices. 
\begin{Assumption}[\textbf{Combination matrices}]
\label{main:assumptions:combinationmatrix}
Consider a combination matrix $W \in \mathbb{R}^{K \times K}$ that is symmetric, doubly stochastic, and primitive (this latter condition is guaranteed if $W$ corresponds to a connected graph and at least one diagonal entry is nonzero).  Let 
\begin{align}
\mW_x \triangleq W \otimes I_{d_1} \in \mathbb{R}^{Kd_1 \times Kd_1}, \ \mW_y \triangleq W \otimes I_{d_2} \in \mathbb{R}^{Kd_2 \times Kd_2}.
\end{align} 
The matrices $\mA_x, \mB_x^2, \mC_x \in \mathbb{R}^{Kd_1 \times Kd_1}$ and $\mA_y, \mB_y^2, \mC_y \in \mathbb{R}^{Kd_2 \times Kd_2}$ are assumed to be polynomial functions of $\mW_x$ and $\mW_y$, respectively, such that
\begin{enumerate}[label=(\roman*)]
    \item $\mA_x, \mA_y, \mC_x, \mC_y$ are symmetric and doubly stochastic;
    \item $\mathrm{null}(\mB_x) = \mathrm{span}\{\mathds{1}_{K} \otimes \bar{x}\}, 
\mathrm{null}(\mB_y) = \mathrm{span}\{\mathds{1}_{K}\otimes \bar{y}\}$ for some nonzero vectors $\bar{x}\in \mR^{d_1}, \bar{y} \in \mR^{d_2}$.
\end{enumerate}
\end{Assumption}
The above assumption is also made in \cite{alghunaim2022unified, alghunaim2020decentralized, sayed2022inference} and is mild, since, as shown in Table~\ref{tab:matrix_choices}, it is satisfied by the matrix choices of all key special cases. Because the combination matrix $W$ is symmetric, doubly stochastic, and primitive, according to the Perron-Frobenius theorem \cite{horn2012matrix}, it has a simple eigenvalue at $1$, while all other eigenvalues have magnitude strictly less than $1$. Let $\{\lambda_i\}_{i=2}^K$  denote the eigenvalues of $W$ excluding $1$. The mixing rate for sharing local quantities is defined as the second-largest eigenvalue in magnitude, namely
\begin{align}
\lambda \triangleq \max_{i=2,\ldots,K} |\lambda_i| < 1.
\end{align}
\vspace{-1em}
\subsection{Fundamental transformation}
\label{main:sec:tranasformation}
To facilitate the theoretical analysis, we reformulate recursions \eqref{main:X_update}--\eqref{main:Uy_update} through a sequence of transformations. Detailed derivations of these transformations are provided in Appendix ~\ref{appendix:lineartransform}. Our first step is to introduce the following auxiliary block vector variables:
\begin{subequations}
\begin{align}
    \mZ_{x,i} &\triangleq \mu_x \mA_x\mM_{x,i} + \mB_x\mD_{x,i}-\mB^2_x\mX_{i} , \label{main:ZxExpression}\\
    \mZ_{y,i} &\triangleq {\cblue -}\mu_y \mA_y \mM_{y,i} + \mB_y \mD_{y,i}-\mB^2_{y} \mY_{i},
\label{main:ZyExpression}
\end{align}
\end{subequations}
and rewrite  \eqref{main:X_update}--\eqref{main:Uy_update} as
\begin{subequations}
\begin{align}
    \mX_{i+1} &= (\mA_x\mC_x - \mB^2_x)\mX_i - \mZ_{x,i} \label{main:x_update_rule},\\
     \mY_{i+1}&= (\mA_y\mC_y - \mB^2_y) \mY_i- \mZ_{y,i} \label{main:y_update_rule},\\
    \mZ_{x,i+1} &=  \mZ_{x,i}  + \mB^2_x \mX_i +\mu_x \mA_x(\mM_{x,i+1} - \mM_{x,i}) \label{main:zx_update_rule},\\
    \mZ_{y,i+1} &= \mZ_{y,i}   + \mB^2_y \mY_i {\cblue -} \mu_y \mA_y(\mM_{y,i+1} - \mM_{y,i}) 
\label{main:zy_update_rule}.
\end{align}
\end{subequations}
    The above reparameterization trick is inspired by \cite{alghunaim2022unified}. However, it is important to emphasize that our transformed recursions \eqref{main:x_update_rule}--\eqref{main:zy_update_rule} have fundamental differences from \cite{alghunaim2022unified}, both in the way of constructing the auxiliary variables and in the way of designing the gradient estimator.
First, we consider an accelerated gradient estimator while they consider a stochastic gradient. Second, they incorporate the true gradient at the network centroid into the auxiliary variables, whereas we rely on the gradient estimators $\mM_{x,i}, \mM_{y,i}$. 
In fact, we found that directly incorporating a similar reformulation to \cite{alghunaim2022unified} leads to loose performance bounds and fails to achieve linear speedup, as it artificially introduces undesired error terms. These observations motivated us to introduce the alternative construction \eqref{main:ZxExpression}---\eqref{main:ZyExpression}.

In what follows, we exploit properties of the combination matrix $W$ and apply a linear transformation to recursions \eqref{main:x_update_rule}---\eqref{main:zy_update_rule}.
To begin with, since  $W$ is symmetric, doubly-stochastic, and primitive, it admits an eigendecomposition of the form:
\begin{align}
W = U\Lambda U^\top \triangleq 
\begin{bmatrix}
    \frac{1}{\sqrt{K}}\mathds{1}_{K},  \widehat{U}
\end{bmatrix}
\begin{bmatrix}
1 &0\\
0& \widehat{\Lambda}
\end{bmatrix}
\begin{bmatrix}
    \frac{1}{\sqrt{K}} \mathds{1}^\top_K\\
    \widehat{U}^\top
\end{bmatrix}
\in \mathbb{R}^{K\times K},
\end{align}
where $\widehat{\Lambda}$ is a diagonal matrix with entries $\lambda_2,\dots,\lambda_K$ and $\widehat{U}\in \mathbb{R}^{K \times (K-1)}$ is an orthogonal matrix.
It follows that the decompositions of its extended forms via the Kronecker product, namely   $\mW_x = W \otimes \mathrm{I}_{d_1}, \mW_y = W \otimes \mathrm{I}_{d_2}$,  are  given by
\begin{subequations} \label{main_decomp}
    \begin{align}
\mW_x &= U\Lambda U^\top \otimes \mathrm{I}_{d_1}\triangleq \mathcal{U}_x\Lambda_x\mathcal{U}^\top_x \triangleq
\begin{bmatrix}
    \frac{1}{\sqrt{K}}\mathds{1}_{K} \otimes \mathrm{I}_{d_1},  \widehat{\mU}_x
\end{bmatrix}
\begin{bmatrix}
\mathrm{I}_{d_1} &0\\
0& \widehat{\Lambda}_x
\end{bmatrix}
\begin{bmatrix}
    \frac{1}{\sqrt{K}} \mathds{1}^\top_K \otimes \mathrm{I}_{d_1}\\
    \widehat{\mU}^\top_x
\end{bmatrix},
\label{main:decompositionW_x}
\end{align}
\begin{align}
    \mW_y &= U\Lambda U^\top \otimes \mathrm{I}_{d_2}  \triangleq \mathcal{U}_y\Lambda_y\mathcal{U}^\top_y
    \triangleq
\begin{bmatrix}
    \frac{1}{\sqrt{K}}\mathds{1}_{K} \otimes \mathrm{I}_{d_2},  \widehat{\mU}_y
\end{bmatrix}
\begin{bmatrix}
\mathrm{I}_{d_2} &0\\
0& \widehat{\Lambda}_y
\end{bmatrix}
\begin{bmatrix}
    \frac{1}{\sqrt{K}} \mathds{1}^\top_K \otimes \mathrm{I}_{d_2}\\
    \widehat{\mU}^\top_y
\end{bmatrix},
\label{main:decompositionW_y}
\end{align}
\end{subequations}
where \(\mU_x \in \mathbb{R}^{Kd_1 \times Kd_1}\), \(\mU_y \in \mathbb{R}^{Kd_2 \times Kd_2}\)
are orthogonal matrices and
each block matrix is defined as
\begin{align}
\widehat{\Lambda}_x &\triangleq \widehat{\Lambda} \otimes \mathrm{I}_{d_1} \in \mathbb{R}^{(K-1)d_{1} \times (K-1) d_1},  &\quad
 \widehat{\Lambda}_y &\triangleq \widehat{\Lambda} \otimes \mathrm{I}_{d_2} \in \mathbb{R}^{(K-1)d_{2} \times (K-1) d_2}, \notag \\
 \widehat{\mU}_x &\triangleq \widehat{U}\otimes \mathrm{I}_{d_1} \in \mathbb{R}^{K d_1 \times (K-1)d_1},  &\quad \widehat{\mU}_y &\triangleq \widehat{U} \otimes \mathrm{I}_{d_2} \in \mathbb{R}^{K d_2 \times (K-1)d_2}.
\end{align}
Since $\mA_x, \mB^2_x, \mC_x$
are polynomial functions of $\mW_x$,
they share the same set of eigenvectors \cite{sayed2022inference,alghunaim2022unified}, and thus, we can define 
\begin{subequations}
\begin{align}
\mA_x &= \mU_x \Lambda_{a_x} \mU^\top_x \triangleq\begin{bmatrix}
\frac{1}{\sqrt{K}}
\mathds{1}_K \otimes \mathrm{I}_{d_1}, \widehat{\mU}_x
\end{bmatrix}
\begin{bmatrix}
    \mathrm{I}_{d_1} &0\\
    0& \widehat{\Lambda}_{a_x}
\end{bmatrix}
\begin{bmatrix}
\frac{1}{\sqrt{K}}
\mathds{1}^\top_K \otimes \mathrm{I}_{d_1} \\
\widehat{\mU}^\top_x 
\end{bmatrix} , \label{proof:Ax_eigen}
\\
\mC_x &= \mU_x \Lambda_{c_x} \mU^\top_x \triangleq\begin{bmatrix}
\frac{1}{\sqrt{K}}
\mathds{1}_K \otimes \mathrm{I}_{d_1}, \widehat{\mU}_x
\end{bmatrix}
\begin{bmatrix}
    \mathrm{I}_{d_1} &0\\
    0& \widehat{\Lambda}_{c_x}
\end{bmatrix}
\begin{bmatrix}
\frac{1}{\sqrt{K}}
\mathds{1}^\top_K \otimes \mathrm{I}_{d_1}\\
\widehat{\mU}^\top_x
\end{bmatrix},  \label{proof:Cx_eigen}
\\
\mB^2_x &= \mU_x \Lambda^2_{b_x}\mU^\top_x
\triangleq \begin{bmatrix}
\frac{1}{\sqrt{K}}
\mathds{1}_K \otimes \mathrm{I}_{d_1}, \widehat{\mU}_x
\end{bmatrix}
\begin{bmatrix}
    0 &0\\
    0& \widehat{\Lambda}^2_{b_x}
\end{bmatrix}
\begin{bmatrix}
\frac{1}{\sqrt{K}}
\mathds{1}^\top_K \otimes \mathrm{I}_{d_1}\\
\widehat{\mU}^\top_x
\end{bmatrix},  \label{proof:Bx_eigen}
\end{align}
\end{subequations}
where $\widehat{\Lambda}_{a_x}, \widehat{\Lambda}_{c_x}, \widehat{\Lambda}^2_{b_x}$ are diagonal matrices whose entries are polynomials of some eigenvalues of
 $W$. We define 
\begin{align}
\widehat{\Lambda}_{a_x} &= {\rm diag}\{ \lambda_{a_x, i}\}_{i=2}^{K} \otimes \mathrm{I}_{d_1}, \ \widehat{\Lambda}_{c_x} = {\rm diag}\{ \lambda_{c_x, i}\}_{i=2}^{K} \otimes \mathrm{I}_{d_1}, \\ \widehat{\Lambda}^2_{b_x} &= {\rm diag}\{ \lambda^2_{b_x, i}\}_{i=2}^{K} \otimes \mathrm{I}_{d_1}.
\end{align}
Similarly,
for $\mA_y, \mB^2_y, \mC_y$, we have
\begin{subequations}
\begin{align}
\mA_y &= \mU_y \Lambda_{a_y} \mU^\top_y\triangleq\begin{bmatrix}
\frac{1}{\sqrt{K}}
\mathds{1}_K \otimes \mathrm{I}_{d_2}, \widehat{\mU}_y
\end{bmatrix}
\begin{bmatrix}
    \mathrm{I}_{d_2} &0\\
    0& \widehat{\Lambda}_{a_y}
\end{bmatrix}
\begin{bmatrix}
\frac{1}{\sqrt{K}}
\mathds{1}^\top_K \otimes \mathrm{I}_{d_2} \\
\widehat{\mU}^\top_y 
\end{bmatrix} ,\\
\mC_y &= \mU_y \Lambda_{c_y} \mU^\top_y \triangleq\begin{bmatrix}
\frac{1}{\sqrt{K}}
\mathds{1}_K \otimes \mathrm{I}_{d_2}, \widehat{\mU}_y
\end{bmatrix}
\begin{bmatrix}
    \mathrm{I}_{d_2} &0\\
    0& \widehat{\Lambda}_{c_y}
\end{bmatrix}
\begin{bmatrix}
\frac{1}{\sqrt{K}}
\mathds{1}^\top_K \otimes \mathrm{I}_{d_2}\\
\widehat{\mU}^\top_y
\end{bmatrix} ,
\\
\mB^2_y &= \mU_y \Lambda^2_{b_y}\mU^\top_y\triangleq\begin{bmatrix}
\frac{1}{\sqrt{K}}
\mathds{1}_K \otimes \mathrm{I}_{d_2}, \widehat{\mU}_y
\end{bmatrix}
\begin{bmatrix}
   0 &0\\
    0& \widehat{\Lambda}^2_{b_y}
\end{bmatrix} 
\begin{bmatrix}
\frac{1}{\sqrt{K}}
\mathds{1}^\top_K \otimes \mathrm{I}_{d_2}\\
\widehat{\mU}^\top_y
\end{bmatrix},
\end{align}
\end{subequations}
where $\widehat{\Lambda}_{a_y}, \widehat{\Lambda}_{c_y}, \widehat{\Lambda}^2_{b_y}$
are
diagonal matrices with the eigenvalues of  $\mA_y, \mC_y, \mB^2_y$  defined as 
\begin{align}
\widehat{\Lambda}_{a_y} &= {\rm diag}\{ \lambda_{a_y, i}\}_{i=2}^{K} \otimes \mathrm{I}_{d_2}, \ \widehat{\Lambda}_{c_y} = {\rm diag}\{ \lambda_{c_y, i}\}_{i=2}^{K} \otimes \mathrm{I}_{d_2},  \ \widehat{\Lambda}^2_{b_y} = {\rm diag}\{ \lambda^2_{b_y, i}\}_{i=2}^{K} \otimes \mathrm{I}_{d_2} .
\end{align}
Based on these eigenvalue entries, we introduce the following matrices for all $i \in \{2, \dots, K\}$:
\begin{align}
G_{x,i} = 
\begin{bmatrix}
    \lambda_{a_x, i} \lambda_{c_x,i} - \lambda^2_{b_x, i}, & -\lambda_{b_x, i}  \\
    \lambda_{b_x, i},  &1 
\end{bmatrix} \in \mR^{2 \times 2} ,\\
G_{y,i} = \begin{bmatrix}
    \lambda_{a_y, i} \lambda_{c_y,i} - \lambda^2_{b_y, i}, & -\lambda_{b_y, i}  \\
    \lambda_{b_y, i},  &1 
\end{bmatrix} \in \mR^{2 \times 2} .
\end{align}
We recall a result from \cite{alghunaim2022unified} that ensures the stability of the transition matrix that will arise in our transformed recursion when selecting combination matrices according to Table \ref{tab:matrix_choices}.
\begin{Lemma}[\textbf{Property of transition matrices} \cite{alghunaim2022unified}]
\label{main:lemma:jordan}
Consider the transition matrices
\begin{align}
 \mP_x &\triangleq \begin{bmatrix}
\widehat{\Lambda}_{a_{x}}\widehat{\Lambda}_{c_{x}} -\widehat{\Lambda}^2_{b_x} & - \widehat{\Lambda}_{b_x} \\
\widehat{\Lambda}_{b_x}  & \mathrm{I}_{(K-1)d_1} 
\end{bmatrix}\in \mathbb{R}^{2(K-1)d_1\times 2(K-1)d_1}, \\
 \mP_y &\triangleq \begin{bmatrix}
\widehat{\Lambda}_{a_{y}}\widehat{\Lambda}_{c_{y}} -\widehat{\Lambda}^2_{b_y}& - \widehat{\Lambda}_{b_y} \\
\widehat{\Lambda}_{b_y}  & \mathrm{I}_{(K-1)d_2}
\end{bmatrix} \in \mathbb{R}^{2(K-1)d_2\times 2(K-1)d_2}.
\end{align}
There exist permutation matrices $\Pi_x, \Pi_y$
such that 
\begin{align}
    \Pi_x\mP_x \Pi^\top_x = {\rm blkdiag}\Big\{G_{x,i}\Big\}_{i=2}^{K} \otimes \mathrm{I}_{d_1}, \quad  \Pi_y\mP_y \Pi^\top_y = {\rm blkdiag}\Big\{G_{y,i}\Big\}_{i=2}^{K} \otimes \mathrm{I}_{d_2}.
\end{align}
Moreover, the transition matrices $\mP_x, \mP_y$ admit similarity transformations
of the form
\begin{align}
    \mP_x = \widehat{\mQ}_x \mT_x \widehat{\mQ}^{-1}_x, \quad \mP_y = \widehat{\mQ}_y \mT_y \widehat{\mQ}^{-1}_y.
\end{align} 
where $\widehat{\mQ}_x, \widehat{\mQ}_y, \mT_x, \mT_y$ are constructed based on two possible cases: 

$\bullet$ \textbf{Case 1}: When each matrix $G_{x,i} \in \mR^{2 \times 2}$ (and $G_{y,i}  \in \mR^{2 \times 2}$) has two distinct eigenvalues, we have
\begin{align}
    G_{x,i} = V_{x,i}\Lambda_{x,i}V^{-1}_{x,i}, \quad G_{y,i} = V_{y,i}\Lambda_{y,i}V^{-1}_{y,i},
\end{align}
where $\Lambda_{x,i}, \Lambda_{y,i}$ are eigenvalue matrices of $G_{x,i} ,G_{y,i}$, respectively.
Then, $\widehat{\mQ}_x, \widehat{\mQ}_y, \mT_x, \mT_y$ are constructed by
\begin{align}
\widehat{\mQ}_{x} &\triangleq \Pi^\top_{x}\Bigg({\rm blkdiag}\Big\{V_{x,i}\Big\}_{i=2}^{K} \otimes \mathrm{I}_{d_1}\Bigg), &\ \widehat{\mQ}_{y} &\triangleq \Pi^\top_{y}\Bigg( {\rm blkdiag}\Big\{V_{y,i}\Big\}_{i=2}^{K} \otimes \mathrm{I}_{d_2}\Bigg),
\\
\mT_{x}  &\triangleq  {\rm blkdiag}\Big\{\Lambda_{x,i} \Big\}_{i=2}^{K}\otimes \mathrm{I}_{d_1}, &\ \mT_{y}  &\triangleq  {\rm blkdiag}\Big\{\Lambda_{y,i} \Big\}_{i=2}^{K}\otimes \mathrm{I}_{d_2}.
\end{align}

$\bullet$ \textbf{Case 2}: When the
eigenvalues of each $G_{x,i} \in \mR^{2 \times 2}, G_{y,i} \in \mR^{2 \times 2}$ are identical, they admit the Jordan canonical form, say 
\begin{align}
G_{x,i} = V'_{x,i}\Lambda'_{x,i} V^{'-1}_{x,i} , \quad  \Lambda'_{x,i} = \begin{bmatrix}
\gamma_{x,i} & 1 \\
0 & \gamma_{x,i}
\end{bmatrix}, \quad G_{y,i} = V'_{y,i}\Lambda'_{y,i}V^{'-1}_{y,i}, \quad \Lambda'_{y,i} = \begin{bmatrix}
\gamma_{y,i}& 1 \\
0 &\gamma_{y,i}
\end{bmatrix}.
\end{align}
where $\gamma_{x,i}, \gamma_{y,i}$
are eigenvalues of $G_{x,i}, G_{y,i}$, respectively.
In this case, $\widehat{Q}_{x}, \widehat{Q}_{y},\mT_x, \mT_y$ are constructed by 
\begin{align}
\widehat{\mQ}_x &\triangleq \Pi^\top_x {\rm blkdiag}\Big\{
V'_{x,i}E_i
\Big\}_{i=2}^{K} \otimes \mathrm{I}_{d_1}, &\quad \widehat{\mQ}_y&\triangleq \Pi^\top_y{\rm blkdiag}\Big\{
V'_{y,i}E_i
\Big\}_{i=2}^{K} \otimes \mathrm{I}_{d_2} , \\
\mT_x &\triangleq {\rm blkdiag}\Big\{
E^{-1}_{i} \Lambda'_{x,i} E_{i}
\Big\}_{i=2}^{K} \otimes \mathrm{I}_{d_1}, &\quad \mT_y &\triangleq {\rm blkdiag}\Big\{
E^{-1}_{i} \Lambda'_{y,i} E_i
\Big\}_{i=2}^{K}  \otimes \mathrm{I}_{d_2},
\end{align}
where 
$E_{i}= {\rm diag}\{1, \epsilon_i\} $, $\epsilon_i >0$ is an arbitrary constant.
Furthermore, we have $\|\mT_x\|<1 , \|\mT_y\|<1$ by choosing matrices from Table \ref{tab:matrix_choices}.
\end{Lemma}
Using the eigenmatrices $\mU_x, \mU_y$, we apply a linear transformation to \eqref{main:x_update_rule}–\eqref{main:zy_update_rule}, which leads to the transformed recursion
presented in the following lemma.
\begin{Lemma}[\textbf{Transformed error recursion}]
Under \autoref{main:assumptions:combinationmatrix},
 recursions \eqref{main:x_update_rule}--\eqref{main:zy_update_rule}  can be transformed into the following equivalent recursions:
\begin{subequations}
\begin{align}
\bx_{c,i+1} &= \bx_{c,i}-\frac{\mu_x}{K} \sum_{k=1}^{K} \bm^x_{k,i}\label{main:eq:x_finalrecursion},\\
\by_{c,i+1} &= \by_{c,i} \textcolor{blue}{ + } \frac{\mu_y}{K}\sum_{k=1}^{K}
\bm^y_{k,i}\label{main:eq:y_finalrecursion} ,\\
\mhE_{x,i+1} &= \mT_x \mhE_{x,i} 
-\frac{\mu_x}{\tau_x}\mQ^{-1}_x 
\begin{bmatrix}
0\\
\widehat{\Lambda}^{-1}_{b_x} \widehat{\Lambda}_{a_x} \widehat{\mU}^\top_x(\mM_{x,i} -\mM_{x,i+1})
\end{bmatrix} \label{main:eq:ex_finalrecursion} ,\\
\mhE_{y,i+1} &= \mT_y \mhE_{y,i} \textcolor{blue}{ + } \frac{\mu_y}{\tau_y}  \mQ^{-1}_y \begin{bmatrix}
0\\ 
\widehat{\Lambda}^{-1}_{b_y}\widehat{\Lambda}_{a_y} \widehat{\mU}^\top_y(\mM_{y,i} - \mM_{y,i+1})
\end{bmatrix}\label{main:eq:ey_finalrecursion},
\end{align}
\end{subequations}
where  $\tau_x$ and $\tau_y$ are arbitrary positive constants, and the coupled block error terms
$\mhE_{x,i}, \mhE_{y,i}$ are defined as 
\begin{align}
\mhE_{x,i} &\triangleq
   \frac{1}{\tau_x} \widehat{\mQ}^{-1}_x
\begin{bmatrix}
 \widehat{\mU}^\top_x \mX_i \\
 \widehat{\Lambda}_{b_x}^{-1} \widehat{\mU}^\top_x \mZ_{x,i}
\end{bmatrix} \in  \mathbb{R}^{2(K-1)d_1}, \\ 
\mhE_{y,i}  &\triangleq
   \frac{1}{\tau_y} \widehat{\mQ}^{-1}_y
\begin{bmatrix}
 \widehat{\mU}^\top_y \mY_i \\
 \widehat{\Lambda}_{b_y}^{-1} \widehat{\mU}^\top_y \mZ_{y,i}
\end{bmatrix}  \in  \mathbb{R}^{2(K-1)d_2}.
\label{main:tranform:def_couplederror}
\end{align}
\end{Lemma}
\begin{proof}
See Appendix~\ref{appendix:lineartransform}.
\end{proof}
Using \eqref{main:tranform:def_couplederror},
we have 
\begin{align} \label{dev_avg_Q}
&\|\tau_x \widehat{\mQ}_x\mhE_{x,i}\|^2 +\|\tau_y \widehat{\mQ}_y\mhE_{y,i}\|^2 = \|\widehat{\mU}^\top_x\mX_i\|^2\notag\\
& \quad + \|\widehat{\mU}^\top_y\mY_i\|^2
+ \| \widehat{\Lambda}_{b_x}^{-1} \widehat{\mU}^\top_x \mZ_{x,i}\|^2
+\| \widehat{\Lambda}_{b_y}^{-1} \widehat{\mU}^\top_y \mZ_{y,i}\|^2.
\end{align}
Now note that 
\begin{align} \label{dev_avg}
&\|\widehat{\mU}^\top_x\mX_i\|^2 + \|\widehat{\mU}^\top_y\mY_i\|^2 \notag\\
&= 
\|\widehat{\mU}_x\widehat{\mU}^\top_x\mX_i\|^2 +\|\widehat{\mU}_y\widehat{\mU}^\top_y\mY_i\|^2 \notag \\
&\overset{\eqref{main:decompositionW_x}, \eqref{main:decompositionW_y}}{=} 
\Big\|\Big(I_{Kd_1} - \frac{1}{K} \mathds{1}_K\mathds{1}_K^\top\otimes I_{d_1}\Big)  \mX_{i}\Big\|^2 
+\Big\|\Big(I_{Kd_2} - \frac{1}{K} \mathds{1}_K\mathds{1}_K^\top\otimes I_{d_2}\Big)  \mY_{i}\Big\|^2 \notag\\
&= 
\Big\|\mX_{i} - \mathds{1}_K \otimes \bx_{c,i}\Big\|^2 +\Big\|\mY_{i} - \mathds{1}_K \otimes \by_{c,i}\Big\|^2.
\end{align}
The above term quantifies the consensus error between the local models and the network centroid. On the other hand, we can have
\begin{align}
&\|\widehat{\Lambda}_{b_x}^{-1} \widehat{\mU}^\top_x \mZ_{x,i}\|^2
+\| \widehat{\Lambda}_{b_y}^{-1} \widehat{\mU}^\top_y \mZ_{y,i}\|^2 \notag \\
&\le 
\|\widehat{\Lambda}_{b_x}^{-1} \|^2
\Big\|\Big(I_{Kd_1} - \frac{1}{K} \mathds{1}_K\mathds{1}^\top_K\otimes I_{d_1}\Big)\mZ_{x,i}\Big\|^2
+ 
\|\widehat{\Lambda}_{b_y}^{-1} \|^2
\Big\|\Big(I_{Kd_2} - \frac{1}{K} \mathds{1}_K\mathds{1}^\top_K\otimes I_{d_2}\Big)\mZ_{y,i}\Big\|^2 \notag \\
&\le 
\|\widehat{\Lambda}_{b_x}^{-1} \|^2
\Big\|\mZ_{x,i} - \Big(\frac{1}{K}\mathds{1}_K\mathds{1}^\top_K\otimes I_{d_1}\Big)\mZ_{x,i}\Big\|^2
 + 
\|\widehat{\Lambda}_{b_y}^{-1} \|^2
\Big\|\mZ_{y,i} - \Big(\frac{1}{K}\mathds{1}_K\mathds{1}^\top_K\otimes I_{d_2}\Big)\mZ_{y,i}\Big\|^2
\notag \\
&\overset{(a)}{\le}
\|\widehat{\Lambda}_{b_x}^{-1} \|^2
\Big\|\mZ_{x,i}- \Big(\frac{1}{K} \mathds{1}_K\mathds{1}^\top_K\otimes \mathrm{I}_{d_1}\Big) \mu_x \mM_{x,i}\Big\|^2
+ 
\|\widehat{\Lambda}_{b_y}^{-1} \|^2
\Big\|\mZ_{y,i} +  \Big(\frac{1}{K}\mathds{1}_K\mathds{1}^\top_K\otimes \mathrm{I}_{d_2} \Big)  \mu_y\mM_{y,i}\Big\|^2 \notag \\
&\le 
\|\widehat{\Lambda}_{b_x}^{-1} \|^2
\Big\|\mZ_{x,i} - \mu_x \mathds{1}_{K} \otimes  \bm^x_{c,i}\Big\|^2
+ 
\|\widehat{\Lambda}_{b_y}^{-1} \|^2
\Big\|\mZ_{y,i} - \mu_y(- \mathds{1}_{K} \otimes \bm^y_{c,i})\Big\|^2.
\end{align}
where 
\begin{align}
\bm^x_{c,i} \triangleq \frac{1}{K}\sum_{k=1}^{K} \bm^x_{k,i}, \quad \bm^y_{c,i} \triangleq \frac{1}{K}\sum_{k=1}^{K} \bm^y_{k,i},
\end{align}
and $(a)$ follows from the fact that ${\cal A}_x$ and ${\cal A}_y$ are doubly stochastic, as well as  $\Big(\mathds{1}_{K}\mathds{1}^\top_{K}\otimes I_{d_1}\Big) \mB_x = 0$ and $\Big(\mathds{1}_{K}\mathds{1}^\top_{K}\otimes I_{d_2}\Big) \mB_y = 0$ (see Assumption \ref{main:assumptions:combinationmatrix}).
The above term quantifies the  deviation of 
$\mZ_{x,i}, \mZ_{y,i}$ from  the {\em scaled}  centroid  of gradient 
estimators, $\mM_{x,i}$ and $-\mM_{y,i}$. 
Taken together, the error vectors $\mhE_{x,i}, \mhE_{y,i}$
account for the consensus errors of the local models and their scaled gradient estimators.
Unlike many existing decentralized minimax works 
\cite{xian2021faster,huang2023enhanced, cai2024accelerated, chen2023convergence, zhang2024jointly} that handle these error terms separately,
we exploit the coupling between them  and jointly analyze recursions \eqref{main:eq:ex_finalrecursion}--\eqref{main:eq:ey_finalrecursion},
which is essential to achieve a tighter performance bound.

\subsection{Main result}
\label{main:sec:main_result}
Based on the transformed model \eqref{main:eq:x_finalrecursion}---\eqref{main:eq:ey_finalrecursion}, we obtain the following main result:

\begin{Theorem}[\textbf{Main result}]
\label{main:theorem:main}
Under Assumptions \ref{main:assumption:costfunction}---\ref{main:assumptions:combinationmatrix},
choosing appropriate hyperparameters for \textbf{DAMA},
it holds that  
\begin{align}
&\frac{1}{T}
\sum_{i=0}^{T-1}
\Big(\mE\|\nabla_x J(\bx_{c,i}, \by_{c,i})\|^2
+\mE\|\nabla_y J(\bx_{c,i}, \by_{c,i})\|^2\Big) \notag \\
&\le 
\mathcal{O}
\Big(
\underbrace{\frac{\mE G_{p,0}}{T\mu_x}
+  \frac{\kappa^2\mE\Delta_{c,0}}{T\mu_y} +\frac{\kappa^2(a^\prime+c^\prime\mu^2_y+f^\prime)\mu^2_y\zeta^2_0}{ T}}_{\text{Initial discrepancy}}
+\underbrace{\frac{\kappa^2b^\prime\mu^2_y\sigma^2}{ T} 
 + \kappa^2(d^\prime+e^\prime) \sigma^2}_{\text{Noisy bound}}
\Big),
\end{align}
where  $\mE G_{p,0}, \mE\Delta_{c,0}, \zeta^2_0$ account for the error arising from the initial discrepancy and
\begin{subequations}
\begin{align}
\rho &\triangleq \max \Big\{\|\mT_x\|, \|\mT_y\|\Big\}, \
\beta^\prime \triangleq  p+\beta^2, \ 
\bar{\beta} \triangleq p+\beta - p \beta, \\
\lambda_a &\triangleq \max\Big\{\lambda_{a_x}, \lambda_{b_x}\Big\}, \quad \frac{1}{\underline{\lambda^2_b}} \triangleq \max \Big\{
\frac{1}{\underline{\lambda^2_{b_x}}},\frac{1}{\underline{\lambda^2_{b_y}}}
\Big\}, \\
a' &\triangleq \frac{L^2_f}{bK\bar{\beta}(1-\rho)\underline{\lambda^2_b}}, \quad 
b^\prime \triangleq \frac{L^2_f\lambda^2_a \beta^\prime}{bb_0K\bar{\beta}^2(1-\rho)^2\underline{\lambda^2_b}}, \\
c^\prime &\triangleq \frac{L^4_f\lambda^2_a\beta^\prime}{b^2K\bar{\beta}^2(1-\rho)^2\underline{\lambda^4_b}}, \\ 
d^\prime &\triangleq \frac{L^2_f\lambda^2_a}{bK\bar{\beta}(1-\rho)^2\underline{\lambda^2_b}}
\Big( \frac{p}{B} \mathbb{I}_{\text{online}}+\frac{\beta^2}{b}\Big) ,\\
e^\prime &\triangleq \frac{1}{b_0\bar{\beta}KT}
+\frac{\beta^2}{Kb\bar{\beta}}
+\frac{p}{KB\bar{\beta}} \mathbb{I}_{\text{online}}, \quad 
f^\prime \triangleq \frac{L^2_f}{bK\bar{\beta}\underline{\lambda^2_b}}.
\end{align}
\end{subequations}
Here, $\mathbb{I}_{\text{online}} \in \{0,1\}$
is an indicator of online streaming setup ($\mathbb{I}_{\text{online}} =1$) or offline finite-sum setup ($\mathbb{I}_{\text{online}} =0$).
\end{Theorem}
\begin{proof}
See Part II of this work \cite{cai2025dama2}.
\end{proof}
Theorem \ref{main:theorem:main} establishes a performance bound for the unified strategy \textbf{DAMA}. By tuning the hyperparameters
$\{b, B, p, \beta_x, \beta_y\}$, one can recover results for existing accelerated gradient estimators, such as STORM, PAGE, and loopless SARAH. Likewise, by choosing the combination matrices
$\{\mA_x, \mA_y, \mB_x, \mB_y, \mC_x, \mC_y\}$
in specific ways, one can recover results for minimax decentralized strategies of the type ED, EXTRA, and several variants of GT. Note that the choice of combination matrices directly affects the quantities
$\{\lambda^2_a, \underline{\lambda^2_b}\}$, which in turn influence the algorithm's performance over sparsely connected networks.

As a result of Theorem \ref{main:theorem:main},
we obtain the results for 
the specific strategies of \textbf{DAMA}, summarized in Tables \ref{tab:algo-comparison:complexity} and  \ref{tab:algo-comparison}. Several of these results are new.

\begin{table*}[!htbp]
\centering
\begin{threeparttable}
\caption{Performance bounds for different instances of \textbf{DAMA} for solving decentralized nonconvex-PL minimax optimization problems.
The technical details can be found in the accompanying Part II of this work \cite{cai2025dama2}.}
\footnotesize
\label{tab:algo-comparison}
\rowcolors{3}{tableShade}{white}
\setlength{\tabcolsep}{4pt}
\renewcommand{\arraystretch}{1.25}
\begin{tabularx}{\linewidth}{l c c B }
\toprule
\makecell{\textbf{Algorithm}} & 
\makecell{\textbf{PS}$^\circ$} & 
\makecell{\textbf{Studied before?}}& 
\makecell{\textbf{Performance Bound}$^{\ddagger}$}  \\
\midrule
STORM+ED (\cite[Corollary 1]{cai2025dama2}) & On & No & $
 \mathcal{O}
\Big( 
\frac{\kappa^2
\mE G_{p,0} 
}{(TK)^{2/3}} +  \frac{\kappa^2\mE\Delta_{c,0}}{(KT)^{2/3}} + \frac{\kappa^2\sigma^2}{(KT)^{2/3}} + \frac{\kappa^2 \zeta^2_0}{T(1-\lambda)^2}
+ \frac{\kappa^2\lambda^2K^{2/3}\sigma^2}{T^{4/3}(1-\lambda)^3}+ \frac{\kappa^2\lambda^2K\sigma^2}{T^2(1-\lambda)^3}
\Big)$  \\
STORM+EXTRA (\cite[Corollary 2]{cai2025dama2})       & On  & No & $\mathcal{O}
\Big( 
\frac{\kappa^2
\mE G_{p,0} 
}{(TK)^{2/3}} +  \frac{\kappa^2\mE\Delta_{c,0}}{(KT)^{2/3}} + \frac{\kappa^2\sigma^2}{(KT)^{2/3}} + \frac{\kappa^2 \zeta^2_0}{T(1-\lambda)^2}
+ \frac{\kappa^2K^{2/3}\sigma^2}{T^{4/3}(1-\lambda)^3}+ \frac{\kappa^2K\sigma^2}{T^2(1-\lambda)^3}
\Big)$   \\
STORM+ATC-GT (\cite[Corollary 3]{cai2025dama2})& On & Yes & $ \mathcal{O}
\Big( 
\frac{\kappa^2
\mE G_{p,0} 
}{(TK)^{2/3}} +  \frac{\kappa^2\mE\Delta_{c,0}}{(KT)^{2/3}} + \frac{\kappa^2\sigma^2}{(KT)^{2/3}} + \frac{\kappa^2 \zeta^2_0}{T(1-\lambda)^3}
+ \frac{\kappa^2\lambda^4K^{2/3}\sigma^2}{T^{4/3}(1-\lambda)^4}+ \frac{\kappa^2\lambda^4K\sigma^2}{T^2(1-\lambda)^4}
\Big)$ \\
PAGE+ED (\cite[Corollary 7]{cai2025dama2}) &On& No&
$\mathcal{O}
\Big( 
\frac{\kappa^2 \mE G_{p,0}}{T(1-\lambda)^{1.5}}+ \frac{\kappa^2\mE\Delta_{c,0}}{T(1-\lambda)^{1.5}}
+\frac{\kappa^2\sigma^2}{T(1-\lambda)^{1.5}}
+\frac{\kappa^2\zeta^2_0(1-\lambda)}{T}+\frac{\kappa^2\sigma^2}{T}  + \frac{\kappa^2\lambda^2K\sigma^2}{(1-\lambda)^{9/4}T^{3/2}}
\Big)$
\\
PAGE+EXTRA (\cite[Corollary 8]{cai2025dama2})&On&No& $\mathcal{O}
\Big( 
\frac{\kappa^2 \mE G_{p,0}}{T(1-\lambda)^{1.5}}+ \frac{\kappa^2\mE\Delta_{c,0}}{T(1-\lambda)^{1.5}}
+\frac{\kappa^2\sigma^2}{T(1-\lambda)^{1.5}}
+\frac{\kappa^2\zeta^2_0(1-\lambda)}{T}+\frac{\kappa^2\sigma^2}{T}  + \frac{\kappa^2K\sigma^2}{(1-\lambda)^{9/4}T^{3/2}}
\Big) $
\\ 
PAGE+ATC-GT (\cite[Corollary 9]{cai2025dama2})&On& Yes &  $\mathcal{O}
\Big( 
\frac{\kappa^2 \mE G_{p,0}}{T(1-\lambda)^{2}}+ \frac{\kappa^2\mE\Delta_{c,0}}{T(1-\lambda)^{2}}
+\frac{\kappa^2\sigma^2}{T(1-\lambda)^{2}}
+\frac{\kappa^2\zeta^2_0(1-\lambda)}{T}+\frac{\kappa^2\sigma^2}{T}  + \frac{\kappa^2\lambda^4K\sigma^2}{(1-\lambda)^{3}T^{3/2}}
\Big) $ 
\\ 
PAGE+ED (\cite[Corollary 4]{cai2025dama2})& Off& No &$ \mathcal{O}
\Big( 
\frac{\kappa^2 \mE G_{p,0}}{T(1-\lambda)^{1.5}}+ \frac{\kappa^2\mE\Delta_{c,0}}{T(1-\lambda)^{1.5}}
+\frac{\kappa^2\zeta^2_0(1-\lambda)}{T}+\frac{\kappa^2\sigma^2}{T}
+ \frac{\kappa^2\lambda^2\sqrt{K}\sigma^2}{\sqrt{N}T}
\Big) $\\
PAGE+EXTRA (\cite[Corollary 5]{cai2025dama2})&Off& No & $ \mathcal{O}
\Big( 
\frac{\kappa^2 \mE G_{p,0}}{T(1-\lambda)^{1.5}}+ \frac{\kappa^2\mE\Delta_{c,0}}{T(1-\lambda)^{1.5}}
+\frac{\kappa^2\zeta^2_0(1-\lambda)}{T}+\frac{\kappa^2\sigma^2}{T}
+ \frac{\kappa^2\sqrt{K}\sigma^2}{\sqrt{N}T}
\Big)$\\
PAGE+ATC-GT (\cite[Corollary 6]{cai2025dama2}) &Off & Yes &
$\mathcal{O}
\Big( 
\frac{\kappa^2 \mE G_{p,0}}{T(1-\lambda)^{2}}+ \frac{\kappa^2\mE\Delta_{c,0}}{T(1-\lambda)^{2}}
+\frac{\kappa^2\zeta^2_0(1-\lambda)}{T}+\frac{\kappa^2\sigma^2}{T}
+ \frac{\kappa^2\lambda^4\sqrt{K}\sigma^2}{\sqrt{N}T}
\Big)$
\\
L-SARAH +ED (\cite[Corollary 10]{cai2025dama2})& Off &No& $\mathcal{O}
\Big(
\frac{\kappa^2\sqrt{N}\mE G_{p,0}}{KT}
+\frac{\kappa^2 \zeta^2_0}{T(1-\lambda)^2}
+ \frac{\kappa^2\lambda^2K\sigma^2}{\sqrt{N}T(1-\lambda)^3}
+ \frac{\kappa^2\lambda^2K^2\zeta^2_0}{NT(1-\lambda)^4} 
+\frac{\kappa^2 \sqrt{N}\mE \Delta_{c,0}}{KT}
+ \frac{\kappa^2\sqrt{N}\sigma^2}{KT}
+ \frac{\kappa^2\zeta^2_0}{T(1-\lambda)}
\Big)$ \\
L-SARAH +EXTRA (\cite[Corollary 11]{cai2025dama2})& Off &No& $\mathcal{O}
\Big(
\frac{\kappa^2\sqrt{N}\mE G_{p,0}}{KT}
+\frac{\kappa^2 \zeta^2_0}{T(1-\lambda)^2}
+ \frac{\kappa^2 K\sigma^2}{\sqrt{N}T(1-\lambda)^3}
+ \frac{\kappa^2K^2\zeta^2_0}{NT(1-\lambda)^4} 
+\frac{\kappa^2 \sqrt{N}\mE \Delta_{c,0}}{KT}
+ \frac{\kappa^2\sqrt{N}\sigma^2}{KT}
+ \frac{\kappa^2\zeta^2_0}{T(1-\lambda)}
\Big)$\\
L-SARAH +ATC-GT (\cite[Corollary 12]{cai2025dama2})& Off &No& $\Big(
\frac{\kappa^2\sqrt{N}\mE G_{p,0}}{KT}
+\frac{\kappa^2  \zeta^2_0}{T(1-\lambda)^3}
+ \frac{\kappa^2 \lambda^4 K \sigma^2}{\sqrt{N}T(1-\lambda)^4}
+ \frac{\kappa^2\lambda^4K^2\zeta^2_0}{NT(1-\lambda)^6}  
+\frac{\kappa^2 \sqrt{N}\mE \Delta_{c,0}}{KT}
+ \frac{\kappa^2\sqrt{N}\sigma^2}{KT}
+ \frac{\kappa^2\zeta^2_0}{T(1-\lambda)^2}
\Big)$
\\
\bottomrule
\end{tabularx}
\begin{tablenotes}[flushleft]
\footnotesize
\item 
Notes: $^\circ$PS = Problem setups (On = Online stochastic, Off = Offline finite-sum).  
$^\ddagger$ The last column reports the performance bounds of each algorithm instance, with some minor terms absorbed into the dominant one.
L-SARAH = Loopless SARAH. 
$T$: communication round. $K$: the number of agents. $N$: the size of the local training sample in the offline scenario. $\sigma^2$: the bounded variance parameter. $\kappa$ is the condition number $\kappa \triangleq L_f/\nu$. $1-\lambda$ denotes the network spectral gap; the sparser the network, the smaller its value. $G_{p,0}, \Delta_{c,0}, \zeta^2_0$ are quantities associated with initial discrepancy (see \cite{cai2025dama2}).
The technical details of the above results can be found in Part II of this work \cite{cai2025dama2}.
\end{tablenotes}
\end{threeparttable}
\end{table*}

We highlight some important findings below,  which appear in the corollaries established in the Part II of this work \cite{cai2025dama2}:

1)
\cite[Corollary 1]{cai2025dama2} establishes the best-known sample complexity and the first meaningful linear speedup result for the decentralized minimax {\em accelerated  momentum} algorithm, shown in Tables \ref{tab:algo-comparison} and \ref{tab:algo-comparison:transient}, respectively.  In comparison, many existing works exhibit an $\mathcal{O}(1-\lambda)$ dependency in the dominant term of their sample complexity \cite{xian2021faster,chen2023convergence, zhang2024jointly, cai2025communication,ghiasvand2025robust, gao2022decentralized,mancino2023variance}, which implies that linear speedup might not be achieved for sparse networks.

2) The case STORM+ATC-GT has been studied in \cite{xian2021faster, huang2023near}. 
\cite[Corollary 3]{cai2025dama2} shows that
our sample complexity improves upon theirs by eliminating the 
$\mathcal{O}(1-\lambda)$ dependence in the dominant term, thereby demonstrating a meaningful linear speedup with respect to the number of agents 
$K$. Moreover, we observe that the transient time of ATC-GT–based algorithms is surpassed by ED-based algorithms. 
From Corollaries \cite[Corollary 2]{cai2025dama2} and \cite[Corollary 3]{cai2025dama2}, 
we observe that the EXTRA-based algorithm also outperforms ATC-GT.

3) Compared to the GT-based variance reduction algorithms \cite{gao2022decentralized, zhang2024jointly}, \cite[Corollary 4]{cai2025dama2} and \cite[Corollary 5]{cai2025dama2} show that
our performance bound has a better dependency on the network spectral gap $\mathcal{O}(1-\lambda)$. Moreover, PAGE-based algorithms outperform STORM-based methods in terms of performance bound under a large $T$, but require using a large batch size at each iteration.

4) 
\cite[Corollaries 1, 2, 4, 5, 7, 8, 10, 11]{cai2025dama2} indicate that the widely used GT-based method 
cannot guarantee superior results.

5) In addition, 
compared to the online results shown in \cite[Corollaries 1-3]{cai2025dama2} and \cite[Corollaries 7-9]{cai2025dama2}, 
the finite-sum result of the PAGE+ED variant shown in  \cite[Corollary 4]{cai2025dama2} can achieve better results when the local sample size $N$ is smaller than a certain order of $\mathcal{O}(\varepsilon^{-1})$.

6) 
Although the communication complexity of Loopless SARAH+ED is higher than that of PAGE+ED, it is more memory efficient since it avoids overusing large batch samples.
From \cite[Corollary 4]{cai2025dama2} and \cite[Corollary 10]{cai2025dama2}, we observe that Loopless SARAH+ED improves the sample complexity of PAGE+ED by a factor of
$\mathcal{O}(\sqrt{K}/(1-\lambda)^{1.5})$. The extra communication rounds introduced in Loopless SARAH+ED offer an additional advantage by allowing agents more time to reduce the consensus error, thereby improving sample complexity.
Most importantly, Loopless SARAH+ED improves upon the previously best-known sample complexity $\mathcal{O}\Big(
\frac{\kappa^2\sqrt{N}\varepsilon^{-2}}{\sqrt{K}}
\Big)$  \cite{chen2024efficient} by a factor of $\mathcal{O}(\sqrt{K})$, thereby establishing a new state-of-the-art sample complexity in an finite-sum setting.
In contrast to \cite{chen2024efficient}, this is obtained without resorting to multi-step fast mixing or prior knowledge of the combination matrix.

We report the transient time of the STORM-based algorithm below. We find that the STORM+ED algorithm achieves faster transient time for linear speedup than other methods.
\begin{table*}[!htbp]
\centering
\begin{threeparttable}
\caption{Transient time of the STORM-based algorithm in achieving linear speedups.}
\footnotesize
\label{tab:algo-comparison:transient}
\rowcolors{3}{tableShade}{white}
\setlength{\tabcolsep}{1pt}
\renewcommand{\arraystretch}{1.55}
\begin{tabularx}{\linewidth}{l  c  L}
\toprule
\makecell{\textbf{Algorithm}} & 
\makecell{\textbf{PS}$^\circ$} &
\makecell{\textbf{Transient time}$^{\star}$}  \\
\midrule
STORM+ED (\cite[Corollary 1]{cai2025dama2}) & On  & $\max \Big\{\mathcal{O}\Big(\frac{\lambda^{6/5}K^{7/5}}{(1-\lambda)^{2.4}}\Big), \mathcal{O}\Big(\frac{\lambda^{3/2}K^{5/4}}{(1-\lambda)^{9/4}}\Big), \mathcal{O}\Big(\frac{\lambda^{3}K^2}{(1-\lambda)^{4.5}}\Big), \mathcal{O}\Big( \frac{K^2}{(1-\lambda)^6}\Big)\Big\}$   \\
STORM+EXTRA    (\cite[Corollary 2]{cai2025dama2})     & On  &  $
  \max \Big\{\mathcal{O}\Big(\frac{K^{7/5}}{(1-\lambda)^{2.4}}\Big), \mathcal{O}\Big(\frac{K^{5/4}}{(1-\lambda)^{9/4}}\Big), \mathcal{O}\Big(\frac{K^2}{(1-\lambda)^{4.5}}\Big), \mathcal{O}\Big( \frac{K^2}{(1-\lambda)^6}\Big)\Big\}$ \\
STORM+ATC-GT (\cite[Corollary 3]{cai2025dama2})& On  & $ \max \Big\{\mathcal{O}\Big(\frac{\lambda^{12/5}K^{7/5}}{(1-\lambda)^{3.6}}\Big), \mathcal{O}\Big(\frac{\lambda^3K^{5/4}}{(1-\lambda)^{3}}\Big), \mathcal{O}\Big(\frac{\lambda^6K^2}{(1-\lambda)^{6}}\Big), \mathcal{O}\Big( \frac{K^2}{(1-\lambda)^9}\Big)\Big\}$ 
\\
\bottomrule
\end{tabularx}
\begin{tablenotes}[flushleft]
\footnotesize
\item Notes: $^\star$
Here, we report the {\em meaningful} transient time, namely, the number of communication rounds required to achieve linear speedup in the number of agents $K$ for the performance bound that is {\em independent} of $\mathcal{O}(1-\lambda)$.  $^\circ$ PS= Problem setups (On = Online stochastic). The  above technical details can be found in Part II of this work \cite{cai2025dama2}.
\end{tablenotes}
\end{threeparttable}
\end{table*}

\section{Simulations}
\label{sec:simulations}

In this section, we illustrate different instances of \textbf{DAMA} through both numerical examples and practical applications. The first experiment uses synthetic data, while the second experiment addresses deep learning tasks. For each case, we evaluate performance under various network topologies, ranging from well-connected to sparsely connected graphs.

\subsection{Numerical example}

The cost function of this example is defined as 
\begin{align}
&\min_{x\in \mR^{d_1}} \max_{y \in \mR^{d_2}} J(x,y) = \frac{1}{K}
\sum_{k=1}^{K}
J_{k}(x,y), \quad \text{where } \notag\\
 J_{k}(x,y) &=
\mE_{\boldsymbol{a}^\top_{k}, \boldsymbol{e}_{k}}
\Big[\frac{1}{2}(\boldsymbol{a}^\top_{k} x)^2  +y^\top(B_kx + \boldsymbol{e}_k)
- \frac{\nu}{2}
\|y\|^2\Big],
\end{align}
where $\boldsymbol{a}_{k} \in \mR^{d_1}$
are the local i.i.d. random sample, $\boldsymbol{e}_{k} \in \mR^{d_2}$
is the i.i.d. random error vector, $B_k \in \mR^{d_2 \times d_1}$ is a deterministic matrix that varies across different agents, and $\nu$ is the strong concavity factor.

We consider sparsely connected network topologies, such as line and ring graphs, as well as a well-connected graph topology {\color{black}following a  Metropolis rule}. Each network consists of $K=20$ agents. The problem setups are given as follows:
we set $\nu=10$; we generate the random coupling matrix
$B_{k}$
with i.i.d. entries $ [B]_{ij} \sim \mathcal{N}(0, 0.001)$.
The entries of the local samples $\boldsymbol{a}_{k}$ 
are independently sampled from 
the normal distribution $\mathcal{N}(1.0+0.01\times k, 10.0)$, where the mean value is perturbed according to agent index to ensure data heterogeneity.
The entries of the error vector $\boldsymbol{e}_{k}$
are independently sampled from 
the distribution $\mathcal{N}(0.0, 10.0)$.
Each agent owns a randomly generated dataset with size $N_k = 2000$.
The dimensions of the vectors $x$ and $y$
is set to $d_1 = 100$ and $d_2 =100$.

\begin{figure*}[!htbp]
    \centering
    \begin{subfigure}[b]{0.38\textwidth}
\includegraphics[width=\linewidth]{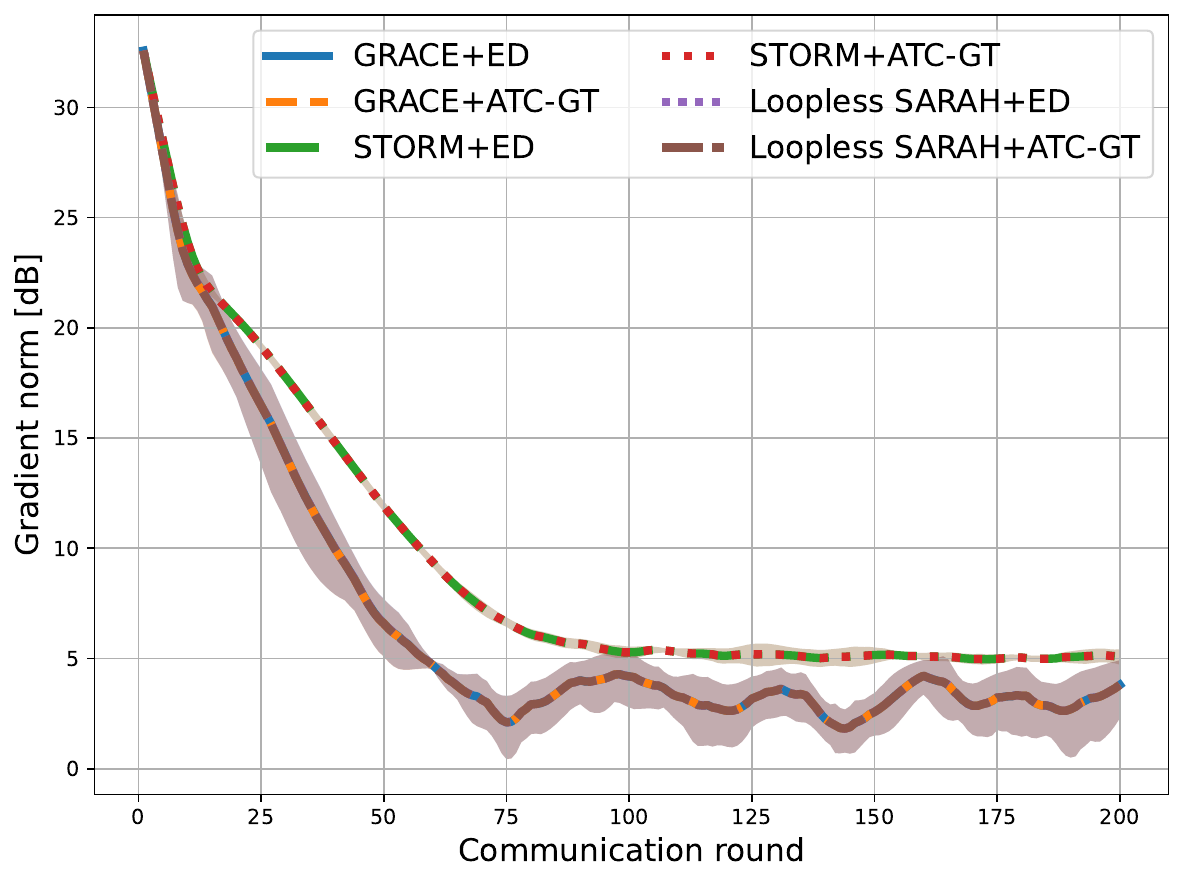}
        \caption{Network under Metropolis rule  (well-connected)}
        \label{main:fig:metropolis}
    \end{subfigure}
    \begin{subfigure}[b]{0.38\textwidth}
        \includegraphics[width=\linewidth]{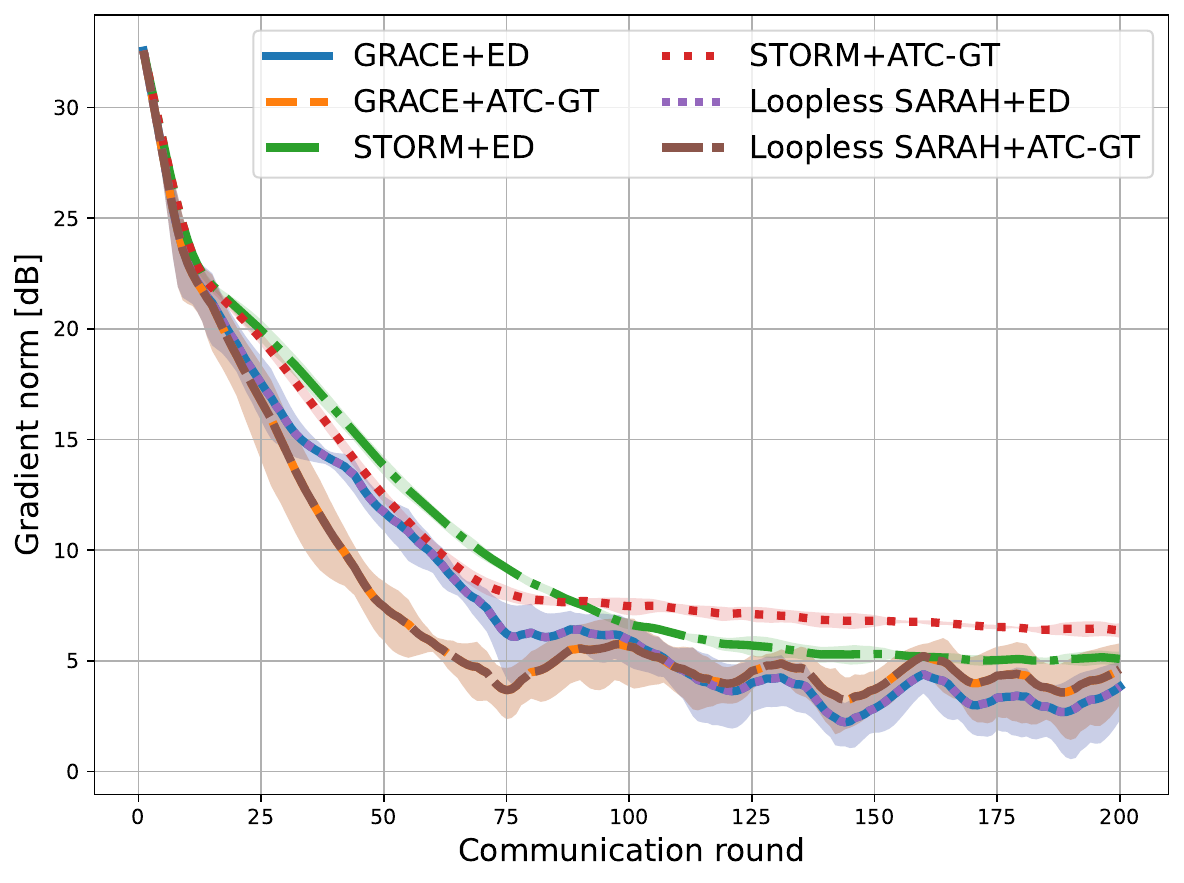}
        \caption{Ring network (sparsely-connected)}
        \label{main:fig:ring}
    \end{subfigure}
    \begin{subfigure}[b]{0.38\textwidth}
\includegraphics[width=\linewidth]{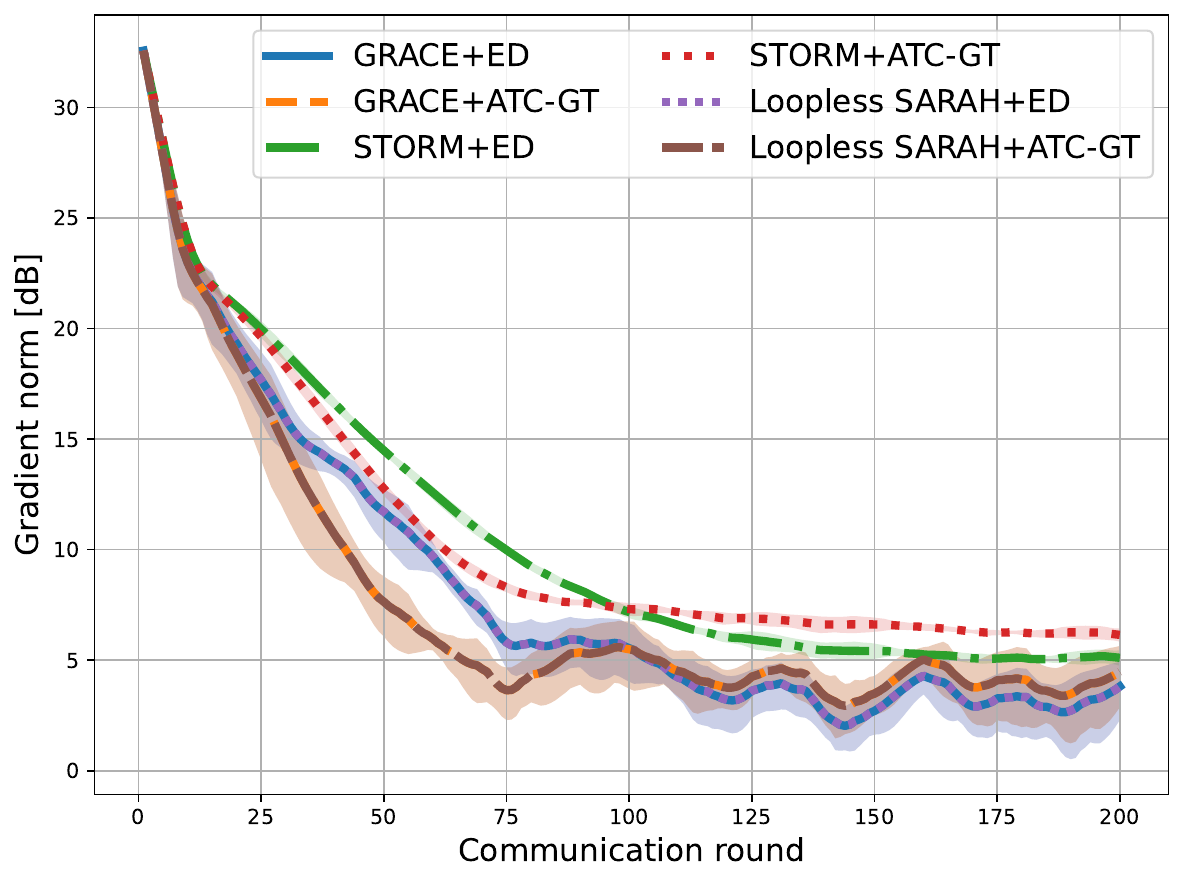}
        \caption{Line network (sparsely-connected)}
        \label{main:fig:line}
    \end{subfigure}
    \caption{Comparison of gradient norm across three different network topologies.}
\label{fig:threecases}
\end{figure*}

We compare important cases of \textbf{DAMA} formed by integrating gradient estimators from the set $\{$"\textbf{GRACE}", "STORM", "Loopless SARAH"$\}$ and decentralized strategies from the set $\{$"ED", "EXTRA", "ATC-GT"$\}$.
To ensure a fair comparison, we adopt the mini-batch variant of the PAGE method for the gradient estimator, i.e., Loopless SARAH-based methods. Note that the proposed \textbf{GRACE} is the hybrid approach of these individual gradient estimators. 
For each case, 
we set hyperparameters as follows: for all algorithms, the step sizes are set to the same value $\mu_x = 0.001$, $\mu_y =0.01$, and the initial batch size $b_0$ is set as $b_0 = 1000$;
for \textbf{GRACE} and STORM, the smoothing factor is set to $\beta_x =\beta_y =\beta =0.01$; for \textbf{GRACE} and Loopless SARAH, the mini-batch size is set as $b =5$ 
and the megabatch size is set as $B=2000$; the Bernoulli parameter $p$ of \textbf{GRACE} and Loopless SARAH is set as $p=0.1$.

The simulation results are presented in Figs.~\ref{main:fig:metropolis}--\ref{main:fig:line}, where we plot the norm of the true gradient at the network centroid over different communication rounds. We observe that, under the same set of hyperparameters, the EXTRA-based instances tend to diverge; therefore, their trajectories are omitted from the plots for clarity. This is because the stability range of hyperparameters for EXTRA-based methods is narrower than that of ED-based methods. Overall, all algorithmic variants achieve comparable performance on well-connected network topologies. However, as the network becomes sparser, the ED-based algorithms begin to outperform their ATC-GT-based counterparts at the steady-state regime. In particular, \textbf{GRACE}+ED and Loopless SARAH+ED achieve notably better performance than STORM+ED at the steady-state regime, albeit at the cost of increased computation.

\subsection{Fair Classifier}

The task of fair classification is to train a model that achieves balanced performance across different data classes \cite{mohri2019agnostic}. This purpose is achieved by reformulating the standard empirical risk minimization problem as a minimax optimization problem, which explicitly forces the neural model to minimize the worst-case (most biased) grouped loss \cite{nouiehed2019solving},
thereby improving the performance on the biased class of data.
In this problem, the optimization objective is given as follows
\begin{align}
&\min_{x} \max_{y} \frac{1}{K}
\sum_{k=1}^{K}
\sum_{c=1}^{C} y_c J_{k, c}(x) - \frac{\rho}{2} \|y\|^2, \notag\\
&\text{where } J_{k,c}(x) = \mE_{\bxi_{k,c} \sim \mathcal{D}_{k,c}}[Q_k(x;\bxi_{k,c})] \notag 
\\
&\text{s.t. } y_{i} \ge 0, ~\forall~ c \in \{1,\dots,C\} \text{ and } \sum_{c=1}^C y_{c} = 1.
\label{simulation:fairclassifier}
\end{align}
Here, $\bxi_{k,c}$ is the random sample of category $c$ at agent $k$ and  $Q_k(x;\bxi_{k,c})$
represents the local loss incurred by category $c$ at the agent $k$. Moreover,  $x$ is the neural network parameter and $\rho$ is the regularization parameter.

In this application, we consider the same network setups as the first numerical example. 
The regularization parameter is set to $\rho = 0.001$.
We use all classes of data of FashionMNIST \cite{xiao2017fashionmnist}, i.e., $C =10$, and each class of data is randomly sharded into $K=20$ agents. 
The neural network structure of the classifier follows \cite{wu2024solving}.
We compare algorithms formed by integrating gradient estimators from the set $\{$"\textbf{GRACE}", "STORM", "Loopless SARAH"$\}$ and decentralized strategies from the set $\{$"ED", "EXTRA", "ATC-GT"$\}$ as well.
The hyperparameters are tuned via grid search and, unless otherwise specified, are set to the same values across methods, or to smaller values to ensure stability.
For ED- and ATC-GT-based method, the step sizes are set as $\mu_x = 0.05, \mu_y =0.1$. The $\mu_x$ of the EXTRA-based method is set as $\mu_x = 0.02$ to ensure stability.
The smoothing factors are set as $\beta_x = \beta_y = 0.95$ for the \textbf{GRACE}- and STORM-based method.
The warm-up batch size $b_0$ and megabatch size $B$
is set as a full batch
while the minibatch size $b$
is set as $b=50$.
The Bornoulli parameter $p$
of \textbf{GRACE}- and Loopless SARAH-based method is set as $p=0.02$.

The simulation results over $10$ independent trials under different network topologies are shown in Figs.~\ref{main:fig:metropolis:classifier}--\ref{main:fig:line:classifier}, where we report the averaged test accuracy across all classes, along with the corresponding  averaged number of stochastic gradient oracle calls over different runs.
Since the training samples are uniformly distributed across agents, data heterogeneity is mild, and it is observed that the performance of the algorithms is comparable across different network topologies.
Under a network generated by the Metropolis rule, we observe that STORM+EXTRA and \textbf{GRACE}+EXTRA outperform the other variants.
Under a sparsely-connected network, 
we observe that 
\textbf{GRACE}+ED (EXTRA), STORM+ED (EXTRA) 
outperform the others.
Furthermore, 
the hybrid approach \textbf{GRACE}-based approach and the  Loopless-SARAH method
incur more stochastic gradient oracles.

\begin{figure}[!htbp]
    \centering
    \begin{subfigure}[b]{0.35\textwidth}
\includegraphics[width=\linewidth]{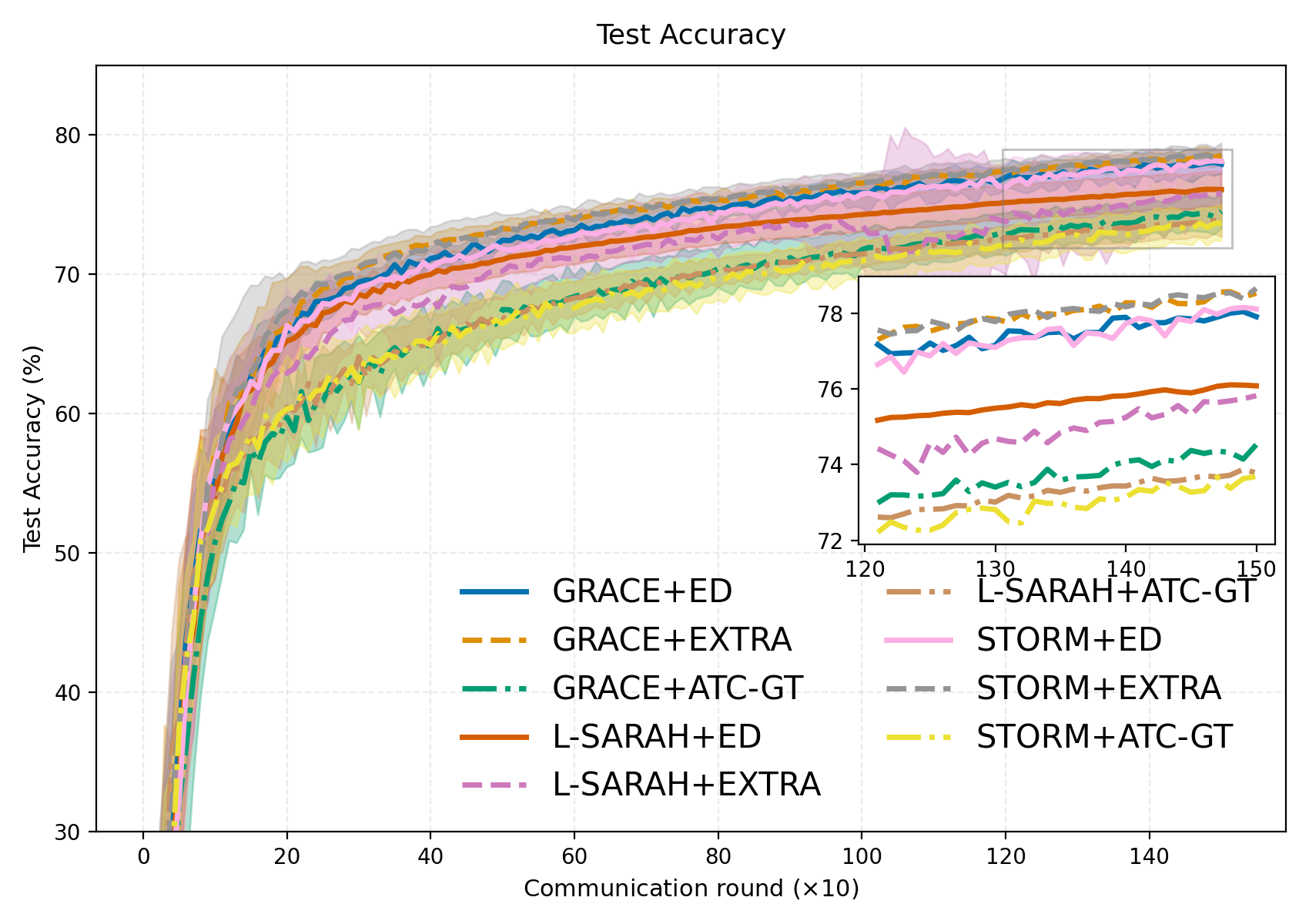}
        \caption{Test accuracy}
\label{main:fig:metropolis:classifier1}
    \end{subfigure}
    \begin{subfigure}[b]{0.35\textwidth}
\includegraphics[width=\linewidth]{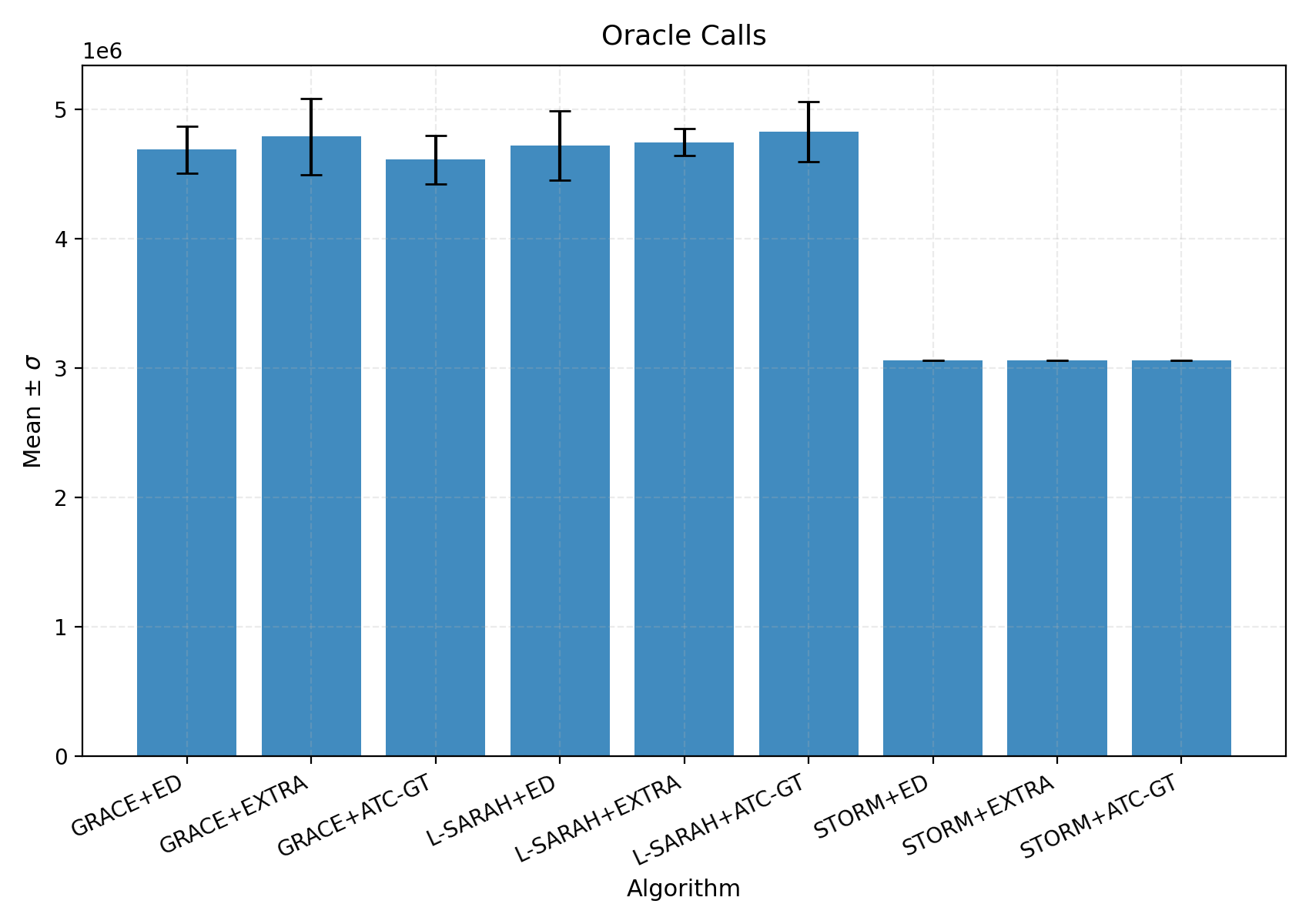}
        \caption{Number of oracle calls}
\label{main:fig:metropolis:classifier2}
    \end{subfigure}
\caption{Performance under a network following the Metropolis rule (well-connected)}
\label{main:fig:metropolis:classifier}
\end{figure}
\begin{figure}[!htbp]
    \centering
    \begin{subfigure}[b]{0.35\textwidth}
\includegraphics[width=\linewidth]{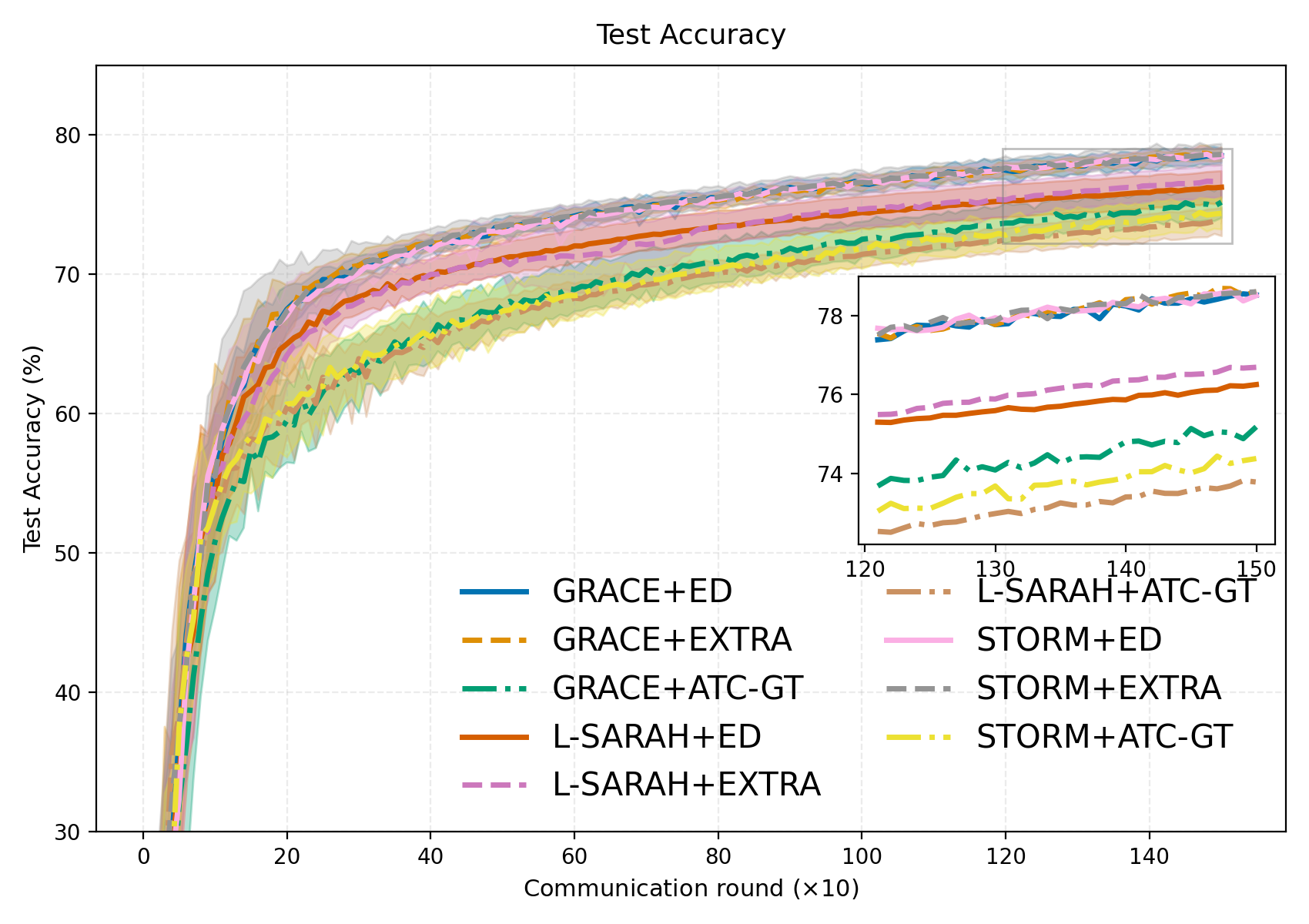}
        \caption{Test accuracy}
\label{main:fig:ring:classifier1}
    \end{subfigure}
    \begin{subfigure}[b]{0.35\textwidth}
\includegraphics[width=\linewidth]{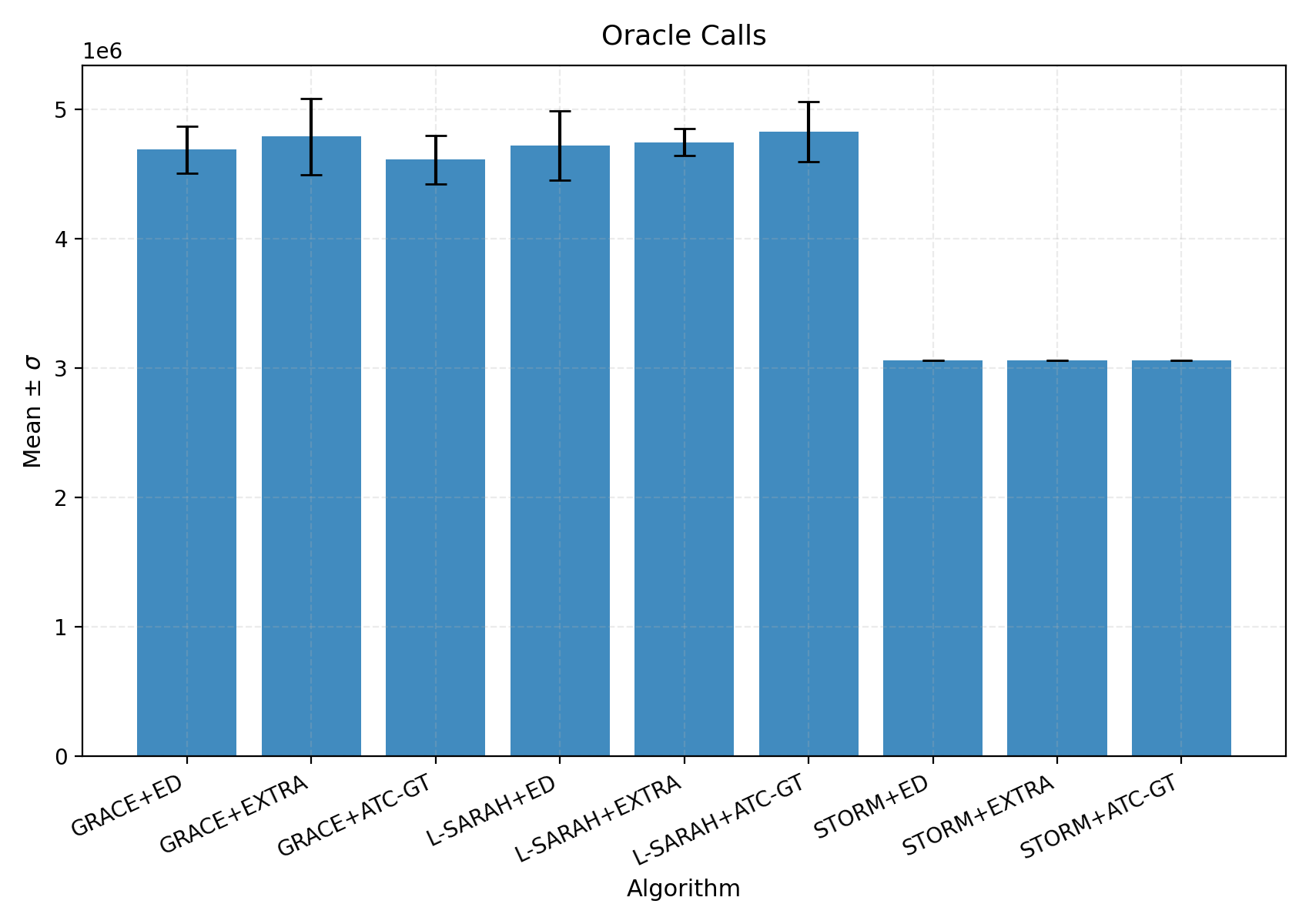}
        \caption{Number of oracle calls}
\label{main:fig:ring:classifier2}
    \end{subfigure}
\label{main:fig:ring:classifier}
\caption{Performance under a ring network (sparsely-connected)}
\end{figure}
\begin{figure}[!htbp]
    \centering
    \begin{subfigure}[b]{0.35\textwidth}
\includegraphics[width=\linewidth]{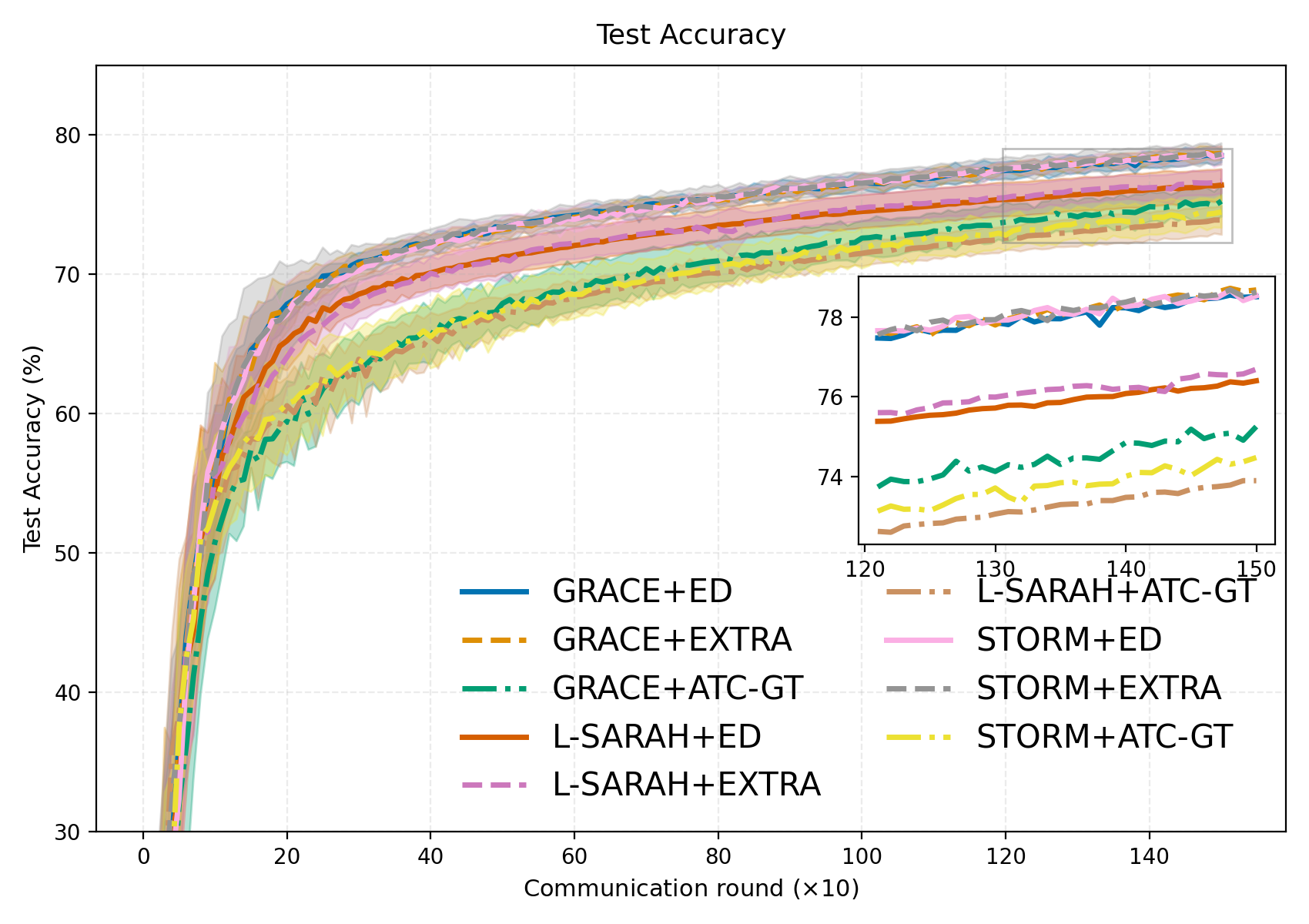}
        \caption{Test accuracy}
\label{main:fig:line:classifier1}
    \end{subfigure}
    \begin{subfigure}[b]{0.35\textwidth}
\includegraphics[width=\linewidth]{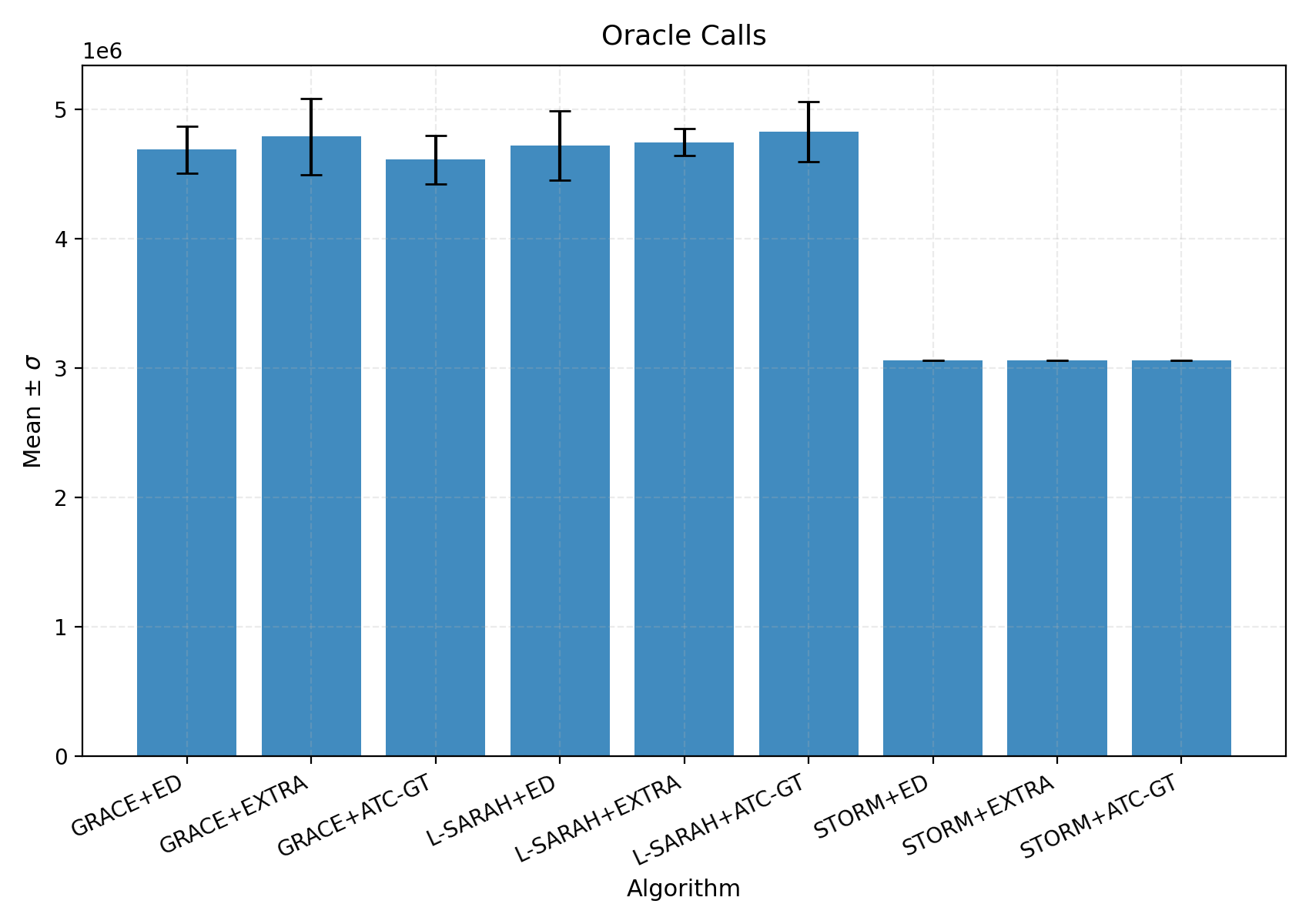}
        \caption{Number of oracle calls}
\label{main:fig:line:classifier2}
\end{subfigure}
\caption{Performance under a line network (sparsely-connected)}
\label{main:fig:line:classifier}
\end{figure}

\section{Conclusions}
In Part I of this work, we proposed a unified framework for decentralized minimax optimization. Within this framework, we introduced a unified gradient estimator \textbf{GRACE} to accelerate. The methodology proposed herein subsumes several existing schemes as special cases, and suggests several other new schemes including methods with superior performance. 
In the accompanying Part II \cite{cai2025dama2},
we conducted a unified theoretical analysis establishing stronger sample complexity results than previously published results, and validated key instances under both well- and sparsely-connected network topologies.

\bibliographystyle{IEEEtran}
\bibliography{refs}
 
\appendix
\section{Preliminaries}

\subsection{Special strategies obtained from \textbf{DAMA}}
\label{appendix:special_realization}
 \textbf{DAMA} framework subsumes several strategies in decentralized optimization, including EXTRA \cite{shi2015extra}, ED (also known as D$^2$) \cite{yuan2018exact,yuan2018exact2,tang2018d}, and a range of GT variants \cite{xu2015augmented,di2016next,qu2017harnessing,nedic2017achieving}, which correspond to specific choices of $\mA_x, \mB_x, \mC_x$ and $\mA_y, \mB_y, \mC_y$ relative to $\mathcal{W}_x$ and $\mathcal{W}_y$, respectively.
In what follows,
we demonstrate different instances of \textbf{DAMA}.
Following \cite{alghunaim2022unified}, we subtract 
\(\mX_{i+1}\)  from \(\mX_{i+2}\)
and 
\(\mY_{i+1}\) 
from 
\(\mY_{i+2}\) in \eqref{main:X_update}---\eqref{main:Y_update}, obtaining:
\begin{subequations}
\begin{align}
&\mX_{i+2} - \mX_{i+1}\overset{(a)}{=}  \mA_x\mC_x(\mX_{i+1} - \mX_{i})
- \mu_x \mA_x (\mM_{x,i+1}-\mM_{x,i})-\mB^2_x \mX_{i+1}  ,
\\
&\mY_{i+2} - \mY_{i+1}  \overset{(b)}{=} \mA_y\mC_y(\mY_{i+1} - \mY_i) 
{\cblue + }\mu_y \mA_y(\mM_{y,i+1} - \mM_{y,i})- \mB^2_y\mY_{i+1},
\end{align}
\end{subequations}
where $(a)$ and $(b)$ follow from the recursions
$\mD_{x,i+1}
= 
\mD_{x,i}
+\mB_x\mX_{i+1}$ and 
$\mD_{y,i+1}
= 
\mD_{y,i}
+\mB_y\mY_{i+1}$.
Rearranging the above equations,
we obtain the following relations 
\begin{subequations}
\begin{align}
    \mX_{i+2} &=
    (\mathrm{I}_{Kd_1} + \mA_x\mC_x - \mB^2_x)\mX_{i+1}-\mA_x\mC_x \mX_i -\mu_x \mA_x (\mM_{x,i+1} - \mM_{x,i}), \label{appendix:x_substract}\\
     \mY_{i+2} &=
    (\mathrm{I}_{Kd_2} + \mA_y\mC_y - \mB^2_y)\mY_{i+1}-\mA_y\mC_y \mY_i  {\cblue + } \mu_y \mA_y (\mM_{y,i+1} - \mM_{y,i}).
    \label{appendix:y_substract}
\end{align}
\end{subequations}

\(\bullet\) \textbf{ED/D$^2$}:

Choosing \(\mA_x = \mW_x\), \(\mC_x = \mathrm{I}_{Kd_1}, \mB_x = (\mathrm{I}_{Kd_1} - \mW_x)^{1/2}\) and 
 \(\mA_y = \mW_y\), \(\mC_y= \mathrm{I}_{Kd_2}, \mB_y = (\mathrm{I}_{Kd_2} - \mW_y)^{1/2}\),
 we get
 \begin{subequations}
\begin{align}
\mX_{i+2} &= \mW_x \Big(2 \mX_{i+1} - \mX_i -\mu_x (\mM_{x,i+1} - \mM_{x,i})\Big),\label{appendix:ED_case}
\\
\mY_{i+2} &= \mW_y \Big(2\mY_{i+1} - \mY_i {\cblue +} \mu_y(\mM_{y,i+1} - \mM_{y,i})\Big),
\end{align}
\end{subequations}
where the iterates are initialized as
$\mX_1 = \mW_x(\mX_0 - \mu_x \mM_{x,0})$ 
and $\mY_1 = \mW_y(\mY_0 {\cblue +} \mu_y \mM_{y,0})$.
We note that the above ED/D$^2$ based strategy has not been studied in the decentralized nonconvex minimax literature.

\(\bullet\) \textbf{EXTRA}

To implement the EXTRA strategy,
we select $\mA_x = \mathrm{I}_{Kd_1}, \mC_x = \mW_x, \mB_x = (\mathrm{I}_{Kd_1} -\mW_x)^{1/2}$
and $\mA_y = \mathrm{I}_{Kd_2}, \mC_y = \mW_y, \mB_y = (\mathrm{I}_{Kd_2} -\mW_y)^{1/2}$
to get 
\begin{subequations}
\begin{align}
\mX_{i+2} &= \mW_x (2 \mX_{i+1} - \mX_i ) -\mu_x (\mM_{x,i+1} - \mM_{x,i}) ,\label{appendix:EXTRA_case}
\\
\mY_{i+2} &= \mW_y (2\mY_{i+1} - \mY_i ) {\cblue +} \mu_y(\mM_{y,i+1} - \mM_{y,i}),
\end{align}
\end{subequations}
where the iterates are initialized as $\mX_{1} = \mW_x \mX_0 - \mu_x\mM_{x,0}$
and $\mY_{1} = \mW_y \mY_0 {\cblue +}\mu_y \mM_{y,0}$.
We also note that the above EXTRA strategy has not been studied in the decentralized nonconvex minimax literature.

\(\bullet\) \textbf{ATC-GT} \\
We next show that the ATC-GT–based method can be recovered from \eqref{appendix:x_substract}–\eqref{appendix:y_substract} as a special case.
If we use 
$\mG_{x,i} \in \mR^{Kd_1}$
and $\mG_{y,i} \in \mR^{Kd_2}$
to denote the tracking variable of the gradient associated with the networked $x$-
and $y$-variable, respectively, then the 
ATC-GT-based minimax algorithm employs the following steps
\begin{subequations}
\begin{align}
\mX_{i+1} &= \mW_x (\mX_{i} - \mu_x \mG_{x,i}) \label{proof:ATC-GTx}, \\
\mY_{i+1} &= \mW_y (\mY_{i} {\cblue +} \mu_y \mG_{y,i}) \label{proof:ATC-GTy},\\
\mG_{x, i+1} & = \mW_x(\mG_{x,i} +\mM_{x,i+1} - \mM_{x,i})  \label{proof:ATC-GTgx},\\
\mG_{y,i+1} & = \mW_y(\mG_{y,i}  + \mM_{y,i+1} - \mM_{y,i})\label{proof:ATC-GTgy},
\end{align}
\end{subequations}
where the quantities are initialized as
$\mG_{x,0} = \mW_x \mM_{x,0}, \mG_{y,0} = \mW_y \mM_{y,0}$ 
and $\mX_{0} = \mW_x \mX_0$ $(\bx_{1,0} = \cdots =\bx_{K,0}), \mY_{0} = \mW_y \mY_0$ $(\by_{1,0} = \cdots =\by_{K,0})$.
Using 
  relations \ref{proof:ATC-GTx}--\eqref{proof:ATC-GTgy} and 
subtracting \(\mW_x \mX_{i+1}\) from  \(\mX_{i+2}\) and \(\mW_y \mY_{i+1}\) from 
  \(\mY_{i+2}\), we obtain
\begin{subequations}
\begin{align}
&\mX_{i+2} - \mW_x \mX_{i+1} \notag\\
&= 
\mW_x (\mX_{i+1} -\mu_x \mG_{x,i+1}) - \mW_x \mX_{i+1} \notag \\
&= \mW_x (\mX_{i+1} -\mu_x \mG_{x,i+1}) - \mW^2_x(\mX_{i} - \mu_x \mG_{x,i}) \notag  \\
&= \mW_x \mX_{i+1} - \mW^2_x \mX_i -\mu_x \mW^2_x (\mM_{x,i+1} - \mM_{x,i}) \label{appendix:ATC-GT-diffx} ,
\end{align}
\begin{align}
&\mY_{i+2} - \mW_y \mY_{i+1} \notag\\
&= 
\mW_y (\mY_{i+1} {\cblue + }\mu_y \mG_{y,i+1}) - \mW_y \mY_{i+1} \notag \\
&= \mW_y (\mY_{i+1}  {\cblue +} \mu_y \mG_{y,i+1}) - \mW^2_y(\mY_{i} {\cblue +} \mu_y \mG_{y,i}) \notag \\
&= \mW_y \mY_{i+1} - \mW^2_y \mY_i  {\cblue + }\mu_y \mW^2_y (\mM_{y,i+1} - \mM_{y,i}) .
\label{appendix:ATC-GT-diffy} 
\end{align}
\end{subequations}
Rearranging the results, we get
\begin{subequations}
    \begin{align}
    \mX_{i+2} &= 2\mW_x \mX_{i+1} - \mW^2_x \mX_i - \mu_x\mW^2_x ( \mM_{x,i+1} -\mM_{x,i})  ,\\
    \mY_{i+2} &= 2\mW_y \mY_{i+1} - \mW^2_y \mY_i
{\cblue +}\mu_y \mW^2_y(\mM_{y,i+1} - \mM_{y,i}),
\end{align} 
\end{subequations}
where the iterates are initialized as 
$\mX_{1} = \mW^2_x(\mX_0 - \mu_x\mM_{x,0})$, $\mY_{1} = \mW^2_y(\mY_0 {\cblue + } \mu_y\mM_{y,0})$, $\mX_0 = \mW_x \mX_0$
and $\mY_0 = \mW_y \mY_0$.
We observe that \eqref{appendix:x_substract}-\eqref{appendix:y_substract} is 
equivalent to
 the above recursion
by choosing  $\mA_x = \mW^2_x,\mC_x = \mathrm{I}_{Kd_1}, \mB_x = (\mathrm{I}_{Kd_1} - \mW_x)$
and  $\mA_y = \mW^2_y,\mC_y = \mathrm{I}_{Kd_2}, \mB_y = (\mathrm{I}_{Kd_2} - \mW_y)$.
Specifically, when $\mM_{x,i}$ and $\mM_{y,i}$ are constructed via STORM \cite{cutkosky2019momentum}, 
i.e., 
with $p=0, \gamma_1 =1, \gamma_2 = 0, b = 1$ in the unified strategy \textbf{GRACE},
our algorithm recovers the DM-HSDA algorithm proposed by \cite{xian2021faster}.
When they are built with PAGE \cite{li2021page}, i.e., 
setting $\beta_x=\beta_y =0, p\not=0, \gamma_1 = 1, \gamma_2 =0$ 
for the unified strategy \textbf{GRACE}, 
\textbf{DAMA} specializes to DREAM \cite{chen2024efficient} with one-step mixing. 
Finally, when constructed as in \cite{cai2024accelerated} with $p=0, \gamma_1=0,\gamma_2 =1, b=1$,
our framework yields the  algorithm of \cite{cai2025communication} without the local step.

\(\bullet\) \textbf{semi-ATC-GT}

The semi-adapt-then-combine GT (semi-ATC-GT) employs a hybrid learning strategy. Specifically, it combines a consensus learning step for local models and the adapt-then-combine diffusion step for updating the tracking variable. 
When a semi-ATC-GT strategy is applied to solve a minimax problem, it admits the following form:
\begin{subequations}
\begin{align}
\mX_{i+1} &= \mW_x\mX_i - \mu_x \mG_{x,i}  ,\\
\mY_{i+1} &= \mW_y\mY_i {\cblue +} \mu_y \mG_{y,i} ,\\
\mG_{x,i+1} & = \mW_x (\mG_{x,i} +\mM_{x,i+1} - \mM_{x,i}),
\\
\mG_{y,i+1} & = \mW_y (\mG_{y,i} + \mM_{x,i+1} - \mM_{x,i}),
\end{align}
\end{subequations}
where the iterates are initialized as
$\mG_{x,0} =  \mM_{x,0}, \mG_{y,0} =  \mM_{y,0}$ 
and $\mX_{0} = \mW_x \mX_0 (\bx_{1,0} = \cdots =\bx_{K,0}), \mY_{0} = \mW_y \mY_0 (\by_{1,0} = \cdots =\by_{K,0})$.
Following an argument similar to 
\eqref{appendix:ATC-GT-diffx}--
\eqref{appendix:ATC-GT-diffy},
we can transform the above recursions into 
\begin{subequations}
   \begin{align}
\mX_{i+2} 
&= 2\mW_x \mX_{i+1} - \mW^2_x \mX_i -\mu_x \mW_x (\mM_{x,i+1} - \mM_{x,i}) \label{appendix:semi-ATC-GT-diffx}, \\
\mY_{i+2}  
&= 2\mW_y \mY_{i+1} - \mW^2_y \mY_i  {\cblue +}\mu_y \mW_y (\mM_{y,i+1} - \mM_{y,i}), \label{appendix:semi-ATC-GT-diffy} 
\end{align} 
with 
\(\mX_{1} = \mW_x(\mX_0 - \mu_x \mM_{x,0})\), \(\mY_{1} = \mW_y(\mY_0 {\cblue +} \mu_y \mM_{y,0})\), \(\mX_0 = \mW_x \mX_0\) and \(\mY_0 = \mW_y \mY_0\).
We note that \eqref{appendix:x_substract}--\eqref{appendix:y_substract}
is equivalent to the above recursions 
when choosing 
$\mA_x = \mW_x, \mC_x = \mW_x, \mB_x = (\mathrm{I}_{Kd_1} - \mW_x)$
and $\mA_y = \mW_y, \mC_y = \mW_y, \mB_y = (\mathrm{I}_{Kd_2} - \mW_y)$.
\end{subequations}
Specifically, if $\mM_{x,i}$ and $\mM_{y,i}$
are constructed by SARAH \cite{nguyen2017sarah} with $\beta_x =\beta_y =0, p \not = 0, \gamma_1=1, \gamma_2 =0$ and periodically activate the large-batch gradient computation in \textbf{GRACE},
we can recover the 
DGDA-VR method
studied in 
\cite{zhang2024jointly}.

\(\bullet\) \textbf{non-ATC-GT}

The non-adapt-then-combine GT (non-ATC-GT) employs a consensus learning strategy for both the local models and the 
tracking variables.
When applied to a minimax problem, it can be expressed in the following form:
\begin{subequations}
\begin{align}
\mX_{i+1} &= \mW_x\mX_i - \mu_x \mG_{x,i}  ,\\
\mY_{i+1} &= \mW_y\mY_i {\cblue +} \mu_y \mG_{y,i} ,\\
\mG_{x,i+1} & = \mW_x \mG_{x,i} +\mM_{x,i+1} - \mM_{x,i},
\\
\mG_{y,i+1} & = \mW_y \mG_{y,i} + \mM_{x,i+1} - \mM_{x,i},
\end{align}
\end{subequations}
where the iterates are initialized as
$\mG_{x,0} =  \mM_{x,0}, \mG_{y,0} =  \mM_{y,0}$ 
and $\mX_{0} = \mW_x \mX_0 (\bx_{1,0} = \cdots =\bx_{K,0}), \mY_{0} = \mW_y \mY_0 (\by_{1,0} = \cdots =\by_{K,0})$.
Following an argument similar to 
\eqref{appendix:ATC-GT-diffx}--
\eqref{appendix:ATC-GT-diffy},
we get 
\begin{subequations}
    \begin{align}
        \mX_{i+2} 
&= 2\mW_x \mX_{i+1} - \mW^2_x \mX_i -\mu_x  (\mM_{x,i+1} - \mM_{x,i}) ,\label{appendix:non-ATC-GT-diffx} \\
\mY_{i+2}  
&= 2\mW_y \mY_{i+1} - \mW^2_y \mY_i  {\cblue +}\mu_y (\mM_{y,i+1} - \mM_{y,i}) \label{appendix:non-ATC-GT-diffy},
\end{align}
\end{subequations}
with $\mX_1 = \mX_0 - \mu_x \mM_{x,0}$, $\mY_1 = \mY_0 {\cblue +} \mu_y \mM_{y,0}$, \(\mX_0 = \mW_x \mX_0\) and \(\mY_0 = \mW_y \mY_0\).
We note that \eqref{appendix:x_substract}-\eqref{appendix:y_substract}
is equivalent to 
the above recursions 
when choosing 
$\mA_x = \mathrm{I}_{Kd_1}, \mC_x = \mW^2_x, \mB_x = (\mathrm{I}_{Kd_1} - \mW_x)$
and $\mA_y = \mathrm{I}_{Kd_2}, \mC_y = \mW^2_y, \mB_y = (\mathrm{I}_{Kd_2} - \mW_y)$.
Such a case has not been considered in the existing literature.

\subsection{Fundamental transformation for the networked recursions}
\label{appendix:lineartransform}

This subsection provides the detailed technical development of Section \ref{main:sec:tranasformation}. The recursions \eqref{main:X_update}–\eqref{main:Uy_update} are formulated for implementation purposes; in what follows, we derive a transformed version more suitable for theoretical analysis. Specifically, we first obtain an intermediate transformation of \eqref{main:X_update}–\eqref{main:Uy_update}, and then apply a fundamental factorization of the associated transition matrix to arrive at the final form. The subsequent analysis is carried out on this transformed recursion. To this end, we rewrite the original networked recursions \eqref{main:X_update}–\eqref{main:Uy_update} as follows:
\begin{subequations}
\begin{align}
\mX_{i+1} &= \mA_x(\mC_x \mX_i -\mu_x \mM_{x,i}) - \mB_x \mD_{x,i} \notag \\
&\overset{(a)}{=} \mA_x(\mC_x \mX_i -\mu_x \mM_{x,i})  - \mB_x
(\mD_{x,i} -\mB_x \mX_i + \mB_x \mX_{i})
\notag \\
&= (\mA_x\mC_x -\mB^2_x) \mX_i
- (\mu_x \mA_x \mM_{x,i} +\mB_x \mD_{x,i}-\mB^2_x \mX_{i}) ,
\end{align}
\begin{align}
\boldsymbol{\mathcal{Y}}_{i+1}
&= \mathcal{A}_y(\mathcal{C}_y \mY_i {\cblue + } \mu_y
\mM_{y,i}) - \mathcal{B}_y
\mD_{y,i} \notag \\
&\overset{(b)}{=}  
\mathcal{A}_y(\mathcal{C}_y \mY_i {\cblue + } \mu_y
\mM_{y,i}) - \mathcal{B}_y
(\mD_{y,i} - \mB_y \mY_i+\mB_y \mY_i) \notag \\
&=
(\mA_y\mC_y -\mB^2_y)\mY_i
{\cblue + } (\mu_y\mA_y \mM_{y,i} -\mB_y\mD_{y,i}+\mB^2_y\mY_i),
\end{align}
\end{subequations}
where $(a)$ and $(b)$ follows by adding and subtracting
 $\mB_{x}\mX_{i}$, $\mB_y\mY_{i}$ respectively.
We now introduce
the following auxiliary variables
\begin{subequations}
\begin{align}
\label{proof:ZxExpression}
    \mZ_{x,i} &\triangleq \mu_x \mA_x\mM_{x,i} + \mB_x\mD_{x,i}-\mB^2_x\mX_{i} ,\\
    \mZ_{y,i} &\triangleq {\cblue -}\mu_y \mA_y \mM_{y,i} + \mB_y \mD_{y,i}-\mB^2_{y} \mY_{i}.
\end{align}
\end{subequations}
We then deduce recursions for $\mZ_{x,i}, \mZ_{y,i}$
as follows
\begin{subequations}
\begin{align}
&\mZ_{x,i+1} - \mZ_{x,i} \notag \\
&= \mu_x \mA_x(\mM_{x,i+1} - \mM_{x,i}) +
\mB_{x} (\mD_{x,i+1} - \mD_{x,i})  - \mB^2_x(\mX_{i+1} - \mX_{i}) \notag \\
&\overset{(a)}{=}\mu_x \mA_x(\mM_{x,i+1} - \mM_{x,i})
+\mB^2_{x} \mX_{i},
\end{align}
\begin{align}
&\mZ_{y,i+1} - \mZ_{y,i}  \notag \\
&={\cblue -}\mu_y\mA_y
(\mM_{y,i+1} - \mM_{y,i})
+
\mB_{y} (\mD_{y,i+1} - \mD_{y,i}) - \mB^2_y(\mY_{i+1} - \mY_{i}) \notag \\
&\overset{(b)}{=} {\cblue -}
\mu_y\mA_y
(\mM_{y,i+1}-\mM_{y,i}) +\mB^2_y \mY_i,
\end{align}
\end{subequations}
where $(a)$ and $(b)$ follow from the relation
$\mD_{x,i+1}
= 
\mD_{x,i}
+\mB_x\mX_{i+1}$ and 
$\mD_{y,i+1}
= 
\mD_{y,i}
+\mB_y\mY_{i+1}$.
Putting the above results together, we arrive at the following equivalnet recursions
\begin{subequations}
\begin{align}
    \mX_{i+1} &= (\mA_x\mC_x - \mB^2_x)\mX_i - \mZ_{x,i} \label{appendix:x_update_rule},\\
     \mY_{i+1}&= (\mA_y\mC_y - \mB^2_y) \mY_i- \mZ_{y,i} \label{appendix:y_update_rule},\\
    \mZ_{x,i+1} &=  \mZ_{x,i}  + \mB^2_x \mX_i +\mu_x \mA_x(\mM_{x,i+1} - \mM_{x,i}) \label{appendix:zx_update_rule},\\
    \mZ_{y,i+1} &= \mZ_{y,i}   + \mB^2_y \mY_i {\cblue -} \mu_y \mA_y(\mM_{y,i+1} - \mM_{y,i}) 
\label{appendix:zy_update_rule}.
\end{align}
\end{subequations}

Let us now apply a linear transformation to the recursions \eqref{appendix:x_update_rule}
--\eqref{appendix:zy_update_rule} by utilizing the structure of the combination matrix given by \eqref{main_decomp} in Section \ref{main:sec:tranasformation}. 
 Multiplying  both sides of  \eqref{appendix:x_update_rule} by $\mU^\top_{x}$, we get 
 \begin{align}
\mU^\top_x \mX_{i+1}
= \mU^\top_x (\mA_x\mC_x - \mB^2_x) \mX_i
- \mU^\top_x \mZ_{x,i}.
 \end{align}
Note that each component of the above recursions can be partitioned into smaller blocks as follows:
\begin{align}
\mU^\top_x \mX_{i+1} &= \begin{bmatrix}
\frac{1}{\sqrt{K}} \mathds{1}^\top_K \otimes \mathrm{I}_{d_1} \\
\widehat{\mU}^\top_x
\end{bmatrix}
\mX_{i+1}
= \begin{bmatrix}
    \sqrt{K} \bx_{c,i+1} \\
    \widehat{\mU}^\top_x \mX_{i+1}
\end{bmatrix} \label{proof:partition_1}
,\\
\mU^\top_x(\mA_x\mC_x - \mB^2_x)\mX_i&\overset{(a)}{=} \begin{bmatrix}
    \mathrm{I}_{d_1} &0\\
    0& \widehat{\Lambda}_{a_x} \widehat{\Lambda}_{c_x} - \widehat{\Lambda}^2_{b_x}
\end{bmatrix}
\begin{bmatrix}
    \sqrt{K}\bx_{c,i}\\
    \widehat{\mU}^\top_x\mX_i
\end{bmatrix}  \notag\\
&= \begin{bmatrix}
    \sqrt{K}\bx_{c,i}
    \\  (\widehat{\Lambda}_{a_x} \widehat{\Lambda}_{c_x} - \widehat{\Lambda}^2_{b_x})\widehat{\mU}^\top_x\mX_i
\end{bmatrix} ,
\label{proof:partition_2}
\end{align} 
where $(a)$ follows from the
eigen-decompositions 
\eqref{proof:Ax_eigen}--\eqref{proof:Bx_eigen}.
Moreover, we have
\begin{align}
\mU^\top_x \mZ_{x,i} &=
\begin{bmatrix}
\frac{1}{\sqrt{K}} \mathds{1}^\top_K \otimes \mathrm{I}_{d_1} \\
\widehat{\mU}^\top_x
\end{bmatrix} \mZ_{x,i} \notag \\
&\overset{(a)}{=}  \begin{bmatrix}
\frac{1}{\sqrt{K}} \mathds{1}^\top_K \otimes \mathrm{I}_{d_1} ( \mu_x \mA_x \mM_{x,i} + \mB_x\mD_{x,i}-\mB^2_x\mX_{i}) \\
\widehat{\mU}^\top_x \mZ_{x,i} 
\end{bmatrix}  \notag \\
&\overset{(b)}{=}\begin{bmatrix}
\mu_x \frac{1}{\sqrt{K}} \mathds{1}^\top_K \otimes \mathrm{I}_{d_1}\mM_{x,i}\\
\widehat{\mU}^\top_x \mZ_{x,i}
\end{bmatrix} \notag \\
&=
\begin{bmatrix}
\frac{\mu_x}{\sqrt{K}}
\sum_{k=1}^{K} \bm^x_{k,i}\\
\widehat{\mU}^\top_x \mZ_{x,i}
\end{bmatrix} , \label{proof:partition_3}
\end{align}
where $(a)$ follows from \eqref{proof:ZxExpression}, and $(b)$ is due to the fact that $\mA_x$ is doubly-stochastic and $\mathds{1}_K \otimes \mathrm{I}_{d_1} \in \mathrm{null}\{\mB_x\}$.
Putting the results \eqref{proof:partition_1}-- \eqref{proof:partition_3} together, 
we get 
\begin{align}
&\begin{bmatrix}
\sqrt{K} \bx_{c,i+1} \\
\widehat{\mU}^\top_x \mX_{i+1}
\end{bmatrix} =\begin{bmatrix}
    \sqrt{K}\bx_{c,i}
    \\  (\widehat{\Lambda}_{a_x} \widehat{\Lambda}_{c_x} - \widehat{\Lambda}^2_{b_x} )\widehat{\mU}^\top_x\mX_i
\end{bmatrix}
 - \begin{bmatrix}
\frac{\mu_x}{\sqrt{K}}
\sum_{k=1}^{K}\bm^x_{k,i}\\
\widehat{\mU}^\top_x \mZ_{x,i}
\end{bmatrix}. \label{proof:partition_final1}
\end{align}
We proceed by applying a linear transformation to the recursion \eqref{appendix:zx_update_rule}.
Multiplying both sides of 
\eqref{appendix:zx_update_rule} by \(\mU^\top_x\), we get 
\begin{align}
\mU^\top_x  \mZ_{x,i+1}= \mU^\top_x  \mZ_{x,i} + 
\mU^\top_x  \mB^2_x \mX_i + \mu_x 
\mU^\top_x \mA_x (\mM_{x,i+1} - \mM_{x,i}).
\end{align}
Following 
similar arguments
 to 
\eqref{proof:partition_1}-- \eqref{proof:partition_3}
and using
\begin{align}
\mU^\top_x \mB^2_x \mX_{i} 
= 
\begin{bmatrix}
0 &0\\
0& \widehat{\Lambda}^2_{b_x}
\end{bmatrix}
\begin{bmatrix}
\frac{1}{\sqrt{K}}
\mathds{1}^\top_K \otimes I_{d_1}\\
\widehat{\mU}^\top_x
\end{bmatrix}
\mX_{i} = \begin{bmatrix}
0\\
\widehat{\Lambda}^2_{b_x} \widehat{\mU}^\top_x \mX_i
\end{bmatrix},
\end{align}
\begin{align}
\mU^\top_x \mA_x(\mM_{x,i+1} - \mM_{x,i}) \notag 
&
= 
\begin{bmatrix}
    \mathrm{I}_{d_1} &0\\
    0& \widehat{\Lambda}_{a_x}
\end{bmatrix}
\begin{bmatrix}
\frac{1}{\sqrt{K}}\mathds{1}^\top_K \otimes \mathrm{I}_{d_1} \Big(\mM_{x,i+1} -\mM_{x,i} \Big)\\
\widehat{\mU}^\top_x\Big(\mM_{x,i+1} - \mM_{x,i}\Big)
\end{bmatrix} \notag \\
&=
\begin{bmatrix}
\frac{1}{\sqrt{K}}
(\sum_{k=1}^{K} \bm^x_{k,i+1} -  \sum_{k=1}^{K} \bm^x_{k,i})\\
\widehat{\Lambda}_{a_x} \widehat{\mU}^\top_x 
\Big(\mM_{x,i+1} - \mM_{x,i}\Big)
\end{bmatrix} .
\label{proof:partition_4}
\end{align}
 we get
\begin{align}
&\begin{bmatrix}
\frac{\mu_x}{\sqrt{K}}
\sum_{k=1}^{K} \bm^x_{k,i+1} \\
\widehat{\mU}^\top_x \mZ_{x,i+1}
\end{bmatrix}
=\begin{bmatrix}
\frac{\mu_x}{\sqrt{K}}
\sum_{k=1}^{K}\bm^x_{k,i} \\
\widehat{\mU}^\top_x \mZ_{x,i}
\end{bmatrix}
+ 
\begin{bmatrix}
0\\
\widehat{\Lambda}^2_{b_x}\widehat{\mU}^\top_x \mX_i
\end{bmatrix} 
+\mu_x \begin{bmatrix}
    \frac{1}{\sqrt{K}} (\sum_{k=1}^K \bm^x_{k,i+1}- \sum_{k=1}^{K} \bm^x_{k,i}) \\
    \widehat{\Lambda}_{a_x} \widehat{\mU}^\top_x(\mM_{x,i+1} - \mM_{x,i})
\end{bmatrix}.
\label{proof:partition_final2}
\end{align}
Note that the first block of the above relation is trivial, as identical terms on both sides cancel out.
Hence,
the nontrivial part of 
\eqref{proof:partition_final1}
 and 
\eqref{proof:partition_final2}
can be equivalently rewritten as follows:
\begin{subequations}
\begin{align}
&\bx_{c,i+1} = \bx_{c,i} - \frac{\mu_x}{K}
\sum_{k=1}^{K}
\bm^x_{k,i}, \\
&\widehat{\mU}^\top_x \mX_{i+1}
= (\widehat{\Lambda}_{a_x}\widehat{\Lambda}_{c_x} - \widehat{\Lambda}^2_{b_x})
\widehat{\mU}^\top_x \mX_{i}
- \widehat{\mU}^\top_x \mZ_{x,i} ,\\
&\widehat{\mU}^\top_x \mZ_{x,i+1}  = \widehat{\mU}^\top_{x}\mZ_{x,i}
+ \widehat{\Lambda}^2_{b_x} \widehat{\mU}^\top_x \mX_{i}  + 
\mu_x \widehat{\Lambda}_{a_x} \widehat{\mU}^\top_x
(\mM_{x,i+1} - \mM_{x,i}).
\end{align}
\end{subequations}

We can also apply a similar transformation  to $\mY_{i+1}$ and $\mZ_{y,i+1}$ 
by multiplying the $\mU^\top_{y}$
on both sides of their respective recursions.
By doing so, 
we obtain a transformed recursion regarding $\by_{c,i+1}$, $\widehat{\mU}^\top_y \mY_{i+1}$, $\widehat{\mU}^\top_y \mZ_{y, i+1}$
as follows 
\begin{subequations}
\begin{align}
&\by_{c,i+1} = \by_{c, i}\textcolor{blue}{ + }   \frac{\mu_y}{K}\sum_{k=1}^K \bm^y_{k,i} ,\\
&\widehat{\mU}^\top_y \mY_{i+1}
= (\widehat{\Lambda}_{a_y} \widehat{\Lambda}_{c_y} - \widehat{\Lambda}^2_{b_y} )\widehat{\mU}^\top_y\mY_i -\widehat{\mU}^\top_y \mZ_{y,i} , \\
 &\widehat{\mU}^\top_y \mZ_{y,i+1}  = \widehat{\mU}^\top_y \mZ_{y,i}
+\widehat{\Lambda}^2_{b_y}\widehat{\mU}^\top_y \mY_i  \textcolor{blue}{-}\mu_y 
\widehat{\Lambda}_{a_y}\widehat{\mU}^\top_y(\mM_{y,i+1} - \mM_{y,i}) .
\end{align}
\end{subequations}
In the following, we will rewrite each block of the above transformed recursion together.
At the same time, we
multiply both sides of the recursion 
regarding \(\widehat{\mU}^\top_x \mZ_{x,i+1},
\widehat{\mU}^\top_y \mZ_{y,i+1}\) 
by $\widehat{\Lambda}^{-1}_{b_x}$
and $\widehat{\Lambda}^{-1}_{b_y}$,
respectively, it then follows that
\begin{subequations}
\begin{align}
\bx_{c,i+1} &= \bx_{c,i}-\mu_x \bm^x_{c,i} \label{appendix:mx_centroid} ,\\
\by_{c,i+1} &= \by_{c,i}\textcolor{blue}{ + } \mu_y \bm^y_{c,i} , \label{appendix:my_centroid}\\
\widehat{\mU}^\top_x \mX_{i+1}
&= (\widehat{\Lambda}_{a_x} \widehat{\Lambda}_{c_x} - \widehat{\Lambda}^2_{b_x} )\widehat{\mU}^\top_x\mX_i -\widehat{\mU}^\top_x \mZ_{x,i}  \\
\widehat{\mU}^\top_y \mY_{i+1}
&= (\widehat{\Lambda}_{a_y} \widehat{\Lambda}_{c_y} - \widehat{\Lambda}^2_{b_y} )\widehat{\mU}^\top_y\mY_i -\widehat{\mU}^\top_y \mZ_{y,i} , \\
\widehat{\Lambda}^{-1}_{b_x}\widehat{\mU}^\top_x \mZ_{x,i+1}&= 
\widehat{\Lambda}^{-1}_{b_x}\widehat{\mU}^\top_x \mZ_{x,i}
+\widehat{\Lambda}_{b_x}\widehat{\mU}^\top_x \mX_i  +\mu_x  \widehat{\Lambda}^{-1}_{b_x}
\widehat{\Lambda}_{a_x}
\widehat{\mU}^\top_x(\mM_{x,i+1} - \mM_{x,i}) ,\\
\widehat{\Lambda}^{-1}_{b_y}\widehat{\mU}^\top_y \mZ_{y,i+1} &= \widehat{\Lambda}^{-1}_{b_y}\widehat{\mU}^\top_y \mZ_{y,i}
+\widehat{\Lambda}_{b_y}\widehat{\mU}^\top_y \mY_i  \textcolor{blue}{-}\mu_y \widehat{\Lambda}^{-1}_{b_y}
\widehat{\Lambda}_{a_y}\widehat{\mU}^\top_y(\mM_{y,i+1} - \mM_{y,i}) .
\end{align}
\end{subequations}
Grouping the last four equations, it follows that
\begin{subequations}
\begin{align}
\bx_{c,i+1} &= \bx_{c,i} - \mu_x \Big( \frac{1}{K}\sum_{k=1}^K
\nabla_x J_k(\bx_{k,i},\by_{k,i})
+ \bs^x_{c,i} \Big)  ,\\
\by_{c,i+1} &= \by_{c,i}\textcolor{blue}{ + } \mu_y \Big( \frac{1}{K}\sum_{k=1}^K
\nabla_y J_k(\bx_{k,i},\by_{k,i})
+ \bs^y_{c,i} \Big)  ,
\end{align}
\begin{align}
&
\begin{bmatrix}
\widehat{\mU}^\top_x \mX_{i+1} \\
\widehat{\Lambda}^{-1}_{b_x}\widehat{\mU}^\top_x   \mZ_{x,i+1}
\end{bmatrix} = 
\begin{bmatrix}
    \widehat{\Lambda}_{a_x} \widehat{\Lambda}_{c_x} -  \widehat{\Lambda}^2_{b_x} & -\widehat{\Lambda}_{b_x} \\
    \widehat{\Lambda}_{b_x} & \mathrm{I}_{(K-1)d_1}
\end{bmatrix} 
\begin{bmatrix}
\widehat{\mU}^\top_x \mX_{i} \\
\widehat{\Lambda}^{-1}_{b_x}\widehat{\mU}^\top_x  \mZ_{x,i}
\end{bmatrix} 
-\mu_x
\begin{bmatrix}
0\\
\widehat{\Lambda}^{-1}_{b_x}\widehat{\Lambda}_{a_x} \widehat{\mU}^\top_x(\mM_{x,i} - \mM_{x,i+1})
\end{bmatrix} ,
\end{align}
\begin{align}
&\begin{bmatrix}
\widehat{\mU}^\top_y \mY_{i+1} \\
\widehat{\Lambda}^{-1}_{b_y}\widehat{\mU}^\top_y \mZ_{y,i+1}
\end{bmatrix}
= 
\begin{bmatrix}
    \widehat{\Lambda}_{a_y} \widehat{\Lambda}_{c_y} -  \widehat{\Lambda}^2_{b_y} & -\widehat{\Lambda}_{b_y} \\
    \widehat{\Lambda}_{b_y} & \mathrm{I}_{(K-1)d_2}
\end{bmatrix} 
\begin{bmatrix}
\widehat{\mU}^\top_y \mY_{i} \\
\widehat{\Lambda}^{-1}_{b_y}\widehat{\mU}^\top_y \mZ_{y,i}
\end{bmatrix}
\textcolor{blue}{ + }\mu_y
\begin{bmatrix}
0\\
\widehat{\Lambda}^{-1}_{b_y}\widehat{\Lambda}_{a_y} \widehat{\mU}^\top_y(\mM_{y,i} - \mM_{y,i+1})
\end{bmatrix},
\end{align}
\end{subequations}
where 
the centroid of the block error vector $ \bs^x_{c,i},  \bs^y_{c,i}$  at communication round $i$ are denoted by
\begin{align}
 \bs^x_{c,i}  &\triangleq  \frac{1}{K}(\mathds{1}^\top_K \otimes \mathrm{I}_{d_1}) \Big( \mM_{x,i} - \nabla_x \mJ(\mX_i, \mY_i) \Big)\in \mathbb{R}^{d_1},  \\
   \bs^y_{c,i} &\triangleq  \frac{1}{K}(\mathds{1}^\top_K \otimes \mathrm{I}_{d_2})\Big( \mM_{y,i} - \nabla_y \mJ(\mX_i, \mY_i) \Big)\in \mathbb{R}^{d_2},  \\
   \nabla_x \mJ(\mX_i, \mY_i)
 &\triangleq{\rm col} \{
\nabla_x J_1(\bx_{1,i}, \by_{1,i})
,\dots, 
\nabla_x J_K(\bx_{K,i}, \by_{K,i})
\} \in \mathbb{R}^{Kd_1},  \\
\nabla_y \mJ(\mX_i, \mY_i)
&\triangleq{\rm col} \{
\nabla_y J_1(\bx_{1,i}, \by_{1,i})
,\dots, 
\nabla_y J_K(\bx_{K,i}, \by_{K,i})
\} \in \mathbb{R}^{Kd_2}.
\end{align}
For the above recursions,
we denote the transition matrix of the recursion as  
\begin{align}
\mP_x &\triangleq \begin{bmatrix}
\widehat{\Lambda}_{a_{x}}\widehat{\Lambda}_{c_{x}} -\widehat{\Lambda}^2_{b_x} & - \widehat{\Lambda}_{b_x} \\
\widehat{\Lambda}_{b_x}  & \mathrm{I}_{(K-1)d_1} 
\end{bmatrix}\in \mathbb{R}^{2(K-1)d_1\times 2(K-1)d_1}\notag, \\
\mP_y &\triangleq \begin{bmatrix}
\widehat{\Lambda}_{a_{y}}\widehat{\Lambda}_{c_{y}} -\widehat{\Lambda}^2_{b_y}& - \widehat{\Lambda}_{b_y} \\
\widehat{\Lambda}_{b_y}  & \mathrm{I}_{(K-1)d_2}
\end{bmatrix} \in \mathbb{R}^{2(K-1)d_2\times 2(K-1)d_2}.
\end{align}
Using the results 
of \cite{alghunaim2022unified} and Lemma \ref{main:lemma:jordan}, there exists some matrices $ \mT_x,  \mT_y$ and invertiable matrices 
the transition matrix $\mP_x$
and $\mP_y$
\begin{align}
\mP_x = \widehat{\mQ}_x \mT_x \widehat{\mQ}^{-1}_x, \quad 
\mP_y = \widehat{\mQ}_y \mT_y \widehat{\mQ}^{-1}_y,
\end{align}
where $\|\mT_x\| <1, \|\mT_y\| < 1$.
If we define the following scaled version of the coupled error terms
\begin{align}
\widehat{\be}_{x,i} \triangleq
   \frac{1}{\tau_x} \widehat{\mQ}^{-1}_x
\begin{bmatrix}
 \widehat{\mU}^\top_x \mX_i \\
 \widehat{\Lambda}_{b_x}^{-1} \widehat{\mU}^\top_x \mZ_{x,i}
\end{bmatrix},  
\widehat{\be}_{y,i}  \triangleq
   \frac{1}{\tau_y} \widehat{\mQ}^{-1}_y
\begin{bmatrix}
 \widehat{\mU}^\top_y \mY_i \\
 \widehat{\Lambda}_{b_y}^{-1} \widehat{\mU}^\top_y \mZ_{y,i}
\end{bmatrix}.
\label{proof:tranform:def_couplederror}
\end{align}
for some positive constants $\tau_x$ and $\tau_y$.
We can then arrive at the final transformed recursion  as follows
\begin{subequations}
\begin{align}
\bx_{c,i+1} &= \bx_{c,i}-\mu_x \Big( \frac{1}{K}\sum_{k=1}^K
\nabla_x J_k(\bx_{k,i},\by_{k,i})
+ \bs^x_{c,i} \Big) \label{eq:x_finalrecursion},\\
\by_{c,i+1} &= \by_{c,i} \textcolor{blue}{ + } \mu_y \Big( \frac{1}{K}\sum_{k=1}^K
\nabla_y J_k(\bx_{k,i},\by_{k,i})
+ \bs^y_{c,i} \Big) \label{eq:y_finalrecursion} ,\\
\widehat{\be}_{x,i+1} &= \mT_x \widehat{\be}_{x,i} 
-\frac{\mu_x}{\tau_x}\mQ^{-1}_x 
\begin{bmatrix}
0\\
\widehat{\Lambda}^{-1}_{b_x} \widehat{\Lambda}_{a_x} \widehat{\mU}^\top_x(\mM_{x,i} -\mM_{x,i+1})
\end{bmatrix} \label{eq:ex_finalrecursion} ,\\
\widehat{\be}_{y,i+1} &= \mT_y \widehat{\be}_{y,i} \textcolor{blue}{ + } \frac{\mu_y}{\tau_y}  \mQ^{-1}_y \begin{bmatrix}
0\\ 
\widehat{\Lambda}^{-1}_{b_y}\widehat{\Lambda}_{a_y} \widehat{\mU}^\top_y(\mM_{y,i} - \mM_{y,i+1})
\end{bmatrix}\label{eq:ey_finalrecursion}.
\end{align}
\end{subequations}

\end{document}